\documentclass[12pt]{amsart}
\usepackage{amssymb}
\usepackage{amscd}
\usepackage{amsthm}
\usepackage{arydshln}
\usepackage[dvips]{graphicx}
\usepackage[all,cmtip]{xy}
\usepackage{hyperref}
\usepackage{mathrsfs}
\usepackage{color}
\usepackage{tikz}
\usepackage{enumitem}
\usepackage{tikz-cd}
\usepackage{mathtools}
\usepackage{stackrel}

\usepackage{tikz}
\newcommand*{\circled}[1]{\lower.7ex\hbox{\tikz\draw (0pt, 0pt)%
		circle (.5em) node {\makebox[1em][c]{\small #1}};}}
\usepackage{geometry}
\usepackage{ulem}

\usepackage[all]{xy}

\usepackage{amsthm}

\newtheorem{Thm}{Theorem}[section]
\newtheorem{Lem}[Thm]{Lemma}
\newtheorem{Def}[Thm]{Definition}
\newtheorem{Cor}[Thm]{Corollary}
\newtheorem{Prop}[Thm]{Proposition}
\newtheorem{Ex1}[Thm]{Example}
\newtheorem{Rem1}[Thm]{Remark}

\tikzset{curve/.style={settings={#1},to path={(\tikztostart)
    .. controls ($(\tikztostart)!\pv{pos}!(\tikztotarget)!\pv{height}!270:(\tikztotarget)$)
    and ($(\tikztostart)!1-\pv{pos}!(\tikztotarget)!\pv{height}!270:(\tikztotarget)$)
    .. (\tikztotarget)\tikztonodes}},
    settings/.code={\tikzset{quiver/.cd,#1}
        \def\pv##1{\pgfkeysvalueof{/tikz/quiver/##1}}},
    quiver/.cd,pos/.initial=0.35,height/.initial=0}

\title[Covering theory]{A lecture note on covering theory\\ in representation theory of algebras$^*$}
\author{Yuming Liu, Nengqun Li, Bohan Xing, Pengyun Chen}

\address{Yuming Liu, Bohan Xing, Pengyun Chen
\newline School of Mathematical Sciences
\newline Laboratory of Mathematics and Complex Systems
\newline Beijing Normal University
\newline Beijing 100875
\newline P.R.China}
\email{ymliu@bnu.edu.cn}
\email{bhxing@mail.bnu.edu.cn}
\email{pychen@mail.bnu.edu.cn}
\address{Nengqun Li
\newline School of Mathematics
\newline Liaoning Normal University
\newline Dalian 116029
\newline P.R.China}
\email{linengqun@lnnu.edu.cn}

  \date{version of \today}

\newcommand{\lra}{\longrightarrow}

\newcommand{\ra}{\rightarrow}
\newcommand{\sdp}{\times\kern-.2em\vrule height1.1ex depth-.05ex}
\newcommand{\epi}{\lra \kern-.8em\ra}

\newcommand{\Q}{{\mathbb Q}}

\newcommand{\cal}{\mathcal}

\newcommand{\mo}{\mathrm{mod}}
\newcommand{\mm}{\mathrm{Mod}}

\newcommand{\ho}{\mathrm{Hom}}

\usepackage{graphicx}
\usepackage{tikz}
\usepackage{float}
\usepackage{caption}
\usepackage{booktabs}
\usepackage{amsmath}
\usepackage{tikz}
\usepackage{mathtools}
\usetikzlibrary{arrows}
\usepackage{xparse}
\usepackage{commutative -diagrams}
\usepackage{array}
\usepackage{tikz-cd}

\newcommand{\Hom}{\operatorname{Hom}}

\newcommand{\End}{\operatorname{End}}

\newcommand{\xra}{\xrightarrow}   

\newcommand{\Mod}{\operatorname{MOD}}
\renewcommand{\mod}{\operatorname{mod}}

\newcommand{\proj}{\operatorname{proj}}

\newcommand{\gr}{\operatorname{gr}}

\renewcommand{\cal}[1]{\mathcal #1}
\newcommand{\bb}[1]{\mathbb #1}

 \normalbaselines
\thanks{$^*$This lecture note arises from a series of lectures on “Covering theory in representation theory of algebras” given by the authors at the 24th Workshop on Representation Theory of Algebras in China, which was held at Anhui University from June 25th to 29th in 2025.}

\begin{document}
\renewcommand{\thefootnote}{\alph{footnote}}
\renewcommand{\thefootnote}{\alph{footnote}}
\setcounter{footnote}{-1} 
\renewcommand{\thefootnote}{\alph{footnote}}
\setcounter{footnote}{-1} 

\begin{abstract}
Covering theory is an important tool in representation theory of algebras, however, the results and the proofs are scattered in the literature. We give an introduction to covering theory at a level as elementary as possible and in a more general way on some basic notions than the existing literature does.
\end{abstract}

\maketitle

\setcounter{tocdepth}{2}
\tableofcontents

\section*{Some quotes from P. Dowbor and A. Skowronski \cite{DS1985,DS1987}}

``Coverings techniques in representation theory were introduced and developed by
Bongartz-Gabriel \cite{BG1982}, Gabriel \cite{Ga1981}, Green \cite{Gr1981}, and Riedtmann \cite{Riedtmann1980a,Riedtmann1980b} for the study of
algebras of finite representation type. In this theory one of the most important results is
the following theorem on Galois coverings proved by Gabriel \cite{Ga1981} and completed by
Martinez-Villa and de la Pe\~{n}a \cite{MP2}:

`Let $\Lambda$ be a locally bounded $k$-category over an algebraically closed field $k$ and let a
group $G$ act freely on $\Lambda$. Then $\Lambda$ is locally representation-finite if and only if $\Lambda/G$ is so. In
this case the push-down functor $F_{\lambda}: \mathrm{mod}\Lambda \rightarrow \mathrm{mod}(\Lambda/G)$ associated with the Galois covering
$F: \Lambda \rightarrow \Lambda/G$ induces a bijection between the $G$-orbits of isoclasses of finitely generated
indecomposable $\Lambda$-modules and the isoclasses of finitely generated indecomposable $\Lambda/G$-modules.'

Therefore, if $\Lambda$ is locally representation-finite, mod$(\Lambda/G)$ coincides with the full
subcategory mod$_1(\Lambda/G)$ formed by all modules of the form $F_{\lambda} M, M\in \mathrm{mod}\Lambda$; in the
general case we call these $F_{\lambda}M$ $\Lambda/G$-modules of the first kind. The authors showed
in \cite{DS1985}, \cite{DLS1984} that mod$(\Lambda/G)$ = mod$_1(\Lambda/G)$ holds for a wider class of locally bounded
categories consisting of all locally support-finite ones.''

``The main object investigated in this paper is the full subcategory mod$_2(\Lambda/G)$ of
mod$(\Lambda/G)$ formed by all modules having no direct summands of the first kind; we
call them $\Lambda/G$-modules of the second kind (with respect to a fixed Galois
covering $F: \Lambda\rightarrow \Lambda/G)$. Our main theorem (3.1) asserts that for some class of
Galois coverings $F: \Lambda\rightarrow \Lambda/G$ the investigation of mod$_2(\Lambda/G)$ can be reduced to
the quotient categories associated with the supports of some periodic, indecomposable, locally finite dimensional $\Lambda$-modules.''

\section{Covering of quivers with relations}

\subsection{Covering of quivers with relations}
\

Throughout this lecture note, we denote $k$ an algebraically closed field, all algebras, categories and functors will be $k$-linear. We assume that the reader is familiar with the basic notions and results in representation theory of finite dimensional algebras (see, for example  \cite{ARS}).

We first define the covering map between quivers.

\begin{Def} {\rm(\cite[Section I.10.1]{E1990})}
Let $Q=(Q_0,Q_1)$, $Q'=(Q_0',Q_1')$ be two quivers.
\begin{itemize}
\item A quiver morphism $f:Q\rightarrow Q'$ is a pair of maps $f_0:Q_0\rightarrow Q'_0$ and $f_1:Q_1\rightarrow Q'_1$ such that if $\alpha\in Q_1$ is an arrow from $x$ to $y$, then $f(\alpha)\in Q'_1$ is an arrow from $f(x)$ to $f(y)$.
\item  A quiver morphism $f$ is said to be a covering if for each $x\in Q_0$, $f$ induces a bijection between the set of arrows of $Q$ starting (resp. ending) at $x$ and the set of arrows of $Q'$ starting (resp. ending) at $f(x)$.
\end{itemize}
\end{Def}

By this definition, if $f: Q\rightarrow Q'$ is a covering map such that $Q$, $Q'$ have no double arrows, then for each $x\in Q_0$, $f$ induces a bijection between the set $x^+$ (resp. $x^-$) of successors (resp. predecessors) of $x$ and the set $f(x)^+$ (resp. $f(x)^-$) of successors (resp. predecessors) of $f(x)$.

It is important that a covering map $f$ has the unique lifting of paths property, that is, for a path $p'$ in $Q'$ and for a vertex $x\in Q_0$ with $f(x)$ the starting (resp. the ending) of $p'$, there exists a unique path $p$ starting (resp. ending) at $x$ in $Q$ such that $f(p)=p'$. In particular, $p$ and $p'$ have the same length. Moreover, by the unique lifting of paths property, it is easy to see that if $f: Q\rightarrow Q'$ is a covering map such that $Q'$ is connected, then $f$ is surjective both on vertices and on arrows.

Recall from \cite{BG1982} that a $k$-linear functor $F: M\rightarrow N$ between two $k$-categories is said to be a {\it covering functor} if for every two objects $a,b$ of $N$, $F$ induces isomorphisms $\bigoplus_{z/a}M(z,y)\rightarrow N(a,b)$ and $\bigoplus_{z/b}M(x,z)\rightarrow N(a,b)$, where $x,y$ are objects of $M$ with $a=F(x)$, $b=F(y)$. In particular, a covering functor is a faithful functor.

The covering of quivers is naturally related to the functor of the corresponding path categories, as the following well-known result shows.

\begin{Lem}  {\rm(cf. \cite[Lemma 4.2]{LL2})} \label{covering of quivers induces covering functor}
If $f:Q\rightarrow Q'$ is a covering of quivers, then $f$ induces a covering functor $kf:kQ\rightarrow kQ'$ between the associated path categories.
\end{Lem}

\begin{proof}
Clearly, any morphism $f:Q\rightarrow Q'$ induces a functor $kQ\rightarrow kQ'$ of path categories which we denote by $kf$. Since $f$ is a covering of quivers, for each $x\in Q_0$ and $a\in Q'_0$ with $f(x)=a$, $f$ maps paths of $Q$ which start at $x$ bijectively to paths of $Q'$ which start at $a$. Therefore for each $b\in Q'_0$, $f$ induces a bijection between $\mathscr{B}=\{$paths $p$ of $Q$ with $s(p)=x$ and $t(p)\in f^{-1}(b)\}$ and $\mathscr{B'}=\{$paths $p'$ of $Q'$ with $s(p')=a$ and $t(p')=b\}$. Since $\mathscr{B}$ is a basis of $\bigoplus_{y/b}kQ(x,y)$ and $\mathscr{B'}$ is a basis of $kQ'(a,b)$, $kf$ induces a bijection $\bigoplus_{y/b}kQ(x,y)\rightarrow kQ'(a,b)$. Similarly, for each $y\in Q_0$ and $a,b\in Q'_0$ with $f(y)=b$, $kf$ induces a bijection $\bigoplus_{x/a}kQ(x,y)\rightarrow kQ'(a,b)$.
\end{proof}

We now recall from \cite{MP1983} the definition of a covering of quivers with relations. Our definition here is slightly more general than the original one. Indeed, we only request that the relations are contained in the ideal generated by arrows and we do not assume that the covering is a Galois covering (that is, it is induced by an (admissible) group of automorphisms) (see Section 1.4). For the absence of a group action, we also request that every minimal relation (see definition below) can be lifted to a minimal relation by a covering map.

\begin{Def} {\rm(\cite[Definition 1.3]{MP1983})} \label{minimal-relation}
Let $Q$ be a quiver and $I$ be an ideal of the path category $kQ$ (not necessarily admissible). A relation $\rho=\sum_{i=1}^{n}\lambda_i u_i\in I(x,y)$ with $\lambda_i\in k^{*}$ and $u_i$ a path from $x$ to $y$, is a minimal relation if $n\geq 2$ and for every non-empty proper subset $K$ of $\{1,\cdots,n\}$, $\sum_{i\in K}\lambda_i u_i\notin I(x,y)$; when $n=1$, $\rho$ is called a zero relation or monomial relation.
\end{Def}

Note that every relation of $I$ is a sum of minimal relations and zero relations.

\begin{Def}  {\rm(cf. \cite[Definition 4.3]{LL2})}  \label{morphism-and-covering-of-quivers-with-relations}
Let $(Q,I)$ and $(Q',I')$ be two quivers with relations, where $I$ (resp. $I'$) is contained in the ideal of $kQ$ (resp. $kQ'$) generated by arrows.
\begin{itemize}
\item A morphism of quivers with relations $f:(Q,I)\rightarrow (Q',I')$ is a morphism of quivers $f:Q\rightarrow Q'$ such that the $k$-linear functor $kf:kQ\rightarrow kQ'$ induced by $f$ maps the morphisms in $I$ onto $I'$;
\item if, moreover, $f$ is a covering of quivers, such that for any $\rho'\in I'(a',b')$ with $\rho'$ a minimal relation or a zero relation of $I'$ and for any $a\in f^{-1}(a')$ (resp. $b\in f^{-1}(b')$), there exists some $b\in f^{-1}(b')$ (resp. $a\in f^{-1}(a')$) and some $\rho\in I(a,b)$ such that $(kf)(\rho)=\rho'$, then $f:(Q,I)\rightarrow (Q',I')$ is called a covering of quivers with relations.
\end{itemize}
\end{Def}

Note that since $f$ is a covering of quivers, the above lifting $\rho$ of the relation $\rho'$ is unique. Moreover, if $\rho'$ is a minimal relation, then so is $\rho$.

\begin{Lem}  {\rm(cf. \cite[Lemma 4.4]{LL2})}  \label{preserve-minimal-relation}
Let $f:(Q,I)\rightarrow (Q',I')$ be a covering of quivers with relations. Let $\rho\in I(a,b)$ be a relation of $I$ and let $\rho'=(kf)(\rho)\in I'(f(a),f(b))$. Then $\rho$ is a minimal relation of $I$ if and only if $\rho'$ is a minimal relation of $I'$. In particular, we have also the unique lifting of minimal relations property.
\end{Lem}

\begin{proof}
Assume that the relation $\rho'$ of $I'$ is minimal. Suppose that $\rho=\sum_{i=1}^{n}\lambda_i p_i$, where $\lambda_i\in k^{*}$ and $p_i$'s are distinct paths of $Q$ from $a$ to $b$, is not minimal, then there exists a non-empty proper subset $K$ of $\{1,\cdots ,n\}$ such that $\rho_1=\sum_{i\in K}\lambda_i p_i\in I$. Then $\rho'=\sum_{i=1}^{n}\lambda_i f(p_i)\in I'$, where $f(p_i)$'s are distinct paths of $Q'$, and $\sum_{i\in K}\lambda_i f(p_i)\in I'$. This contradicts the minimality of $\rho'$.

Conversely, let $\rho=\sum_{i=1}^{n}\lambda_i p_i$ be minimal, where $\lambda_i\in k^{*}$ and $p_i$'s are distinct paths of $Q$ from $a$ to $b$; we call it the standard form of $\rho$. Since $f:Q\rightarrow Q'$ is a covering of quivers, $f(p_i)$'s are distinct paths of $Q'$, and $\rho'=\sum_{i=1}^{n}\lambda_i f(p_i)$ is the standard form of $\rho'$. Suppose that $\rho'$ is not minimal, then there exists a non-empty proper subset $K$ of $\{1,\cdots ,n\}$ such that $\rho''=\sum_{i\in K}\lambda_i f(p_i)\in I'$. We may assume that $\rho''$ is a minimal relation or a zero relation of $I'$, so there exists some $c\in Q_0$ with $f(c)=f(b)$ and some $r\in I(a,c)$ such that $(kf)(r)=\rho''$. If $r=\sum_{j=1}^{m}\mu_j q_j$ is the standard form of $r$, then $\sum_{j=1}^{m}\mu_j f(q_j)$ is the standard form of $\rho''$. So there exists a bijection $\sigma:\{1,\cdots,m\}\rightarrow K$ such that $\mu_j=\lambda_{\sigma(j)}$ and $f(q_j)=f(p_{\sigma(j)})$ for each $j\in\{1,\cdots,m\}$. Since $f:Q\rightarrow Q'$ is a covering of quivers and since $s(q_j)=s(p_{\sigma(j)})$ for each $j\in\{1,\cdots,m\}$, by the unique lifting of paths property, we have $q_j=p_{\sigma(j)}$ for each $j\in\{1,\cdots,m\}$. Then $r=\sum_{i\in K}\lambda_i p_i\in I$, which contradicts the minimality of $\rho$.
\end{proof}

\begin{Ex1}\label{ex:simple-example}
Let $Q$ be the quiver
$$
\begin{tikzpicture}
\draw[->] (0.2,0.1) -- (1.8,0.1);
\draw[->] (1.8,-0.1) -- (0.2,-0.1);
\node at(0,0) {$1$};
\node at(2,0) {$2$};
\node at(1,0.3) {$\alpha$};
\node at(1,-0.3) {$\beta$};
\end{tikzpicture}
$$
and let $I=\langle \alpha\beta,\beta\alpha\rangle$ be an ideal in $kQ$. Moreover, let $Q'$ be the quiver
$$
\begin{tikzpicture}
\draw[->] (-0.2,0.1) arc (15:345:0.5);
\node at(0,0) {$v$};
\node at(-1.4,0) {$x$};
\end{tikzpicture}
$$
and let $I'=\langle x^2\rangle$ be an ideal in $kQ'$. Then there is a covering
$$
f:(Q,I)\rightarrow(Q',I')
$$
of quivers with relations, which is given by $f(1)=f(2)=v$ and $f(\alpha)=f(\beta)=x$.
\end{Ex1}

\begin{Ex1} \label{a-covering-of-quivers-with-relations-but-not-Galois}
Let $Q$ be the quiver
$$\begin{tikzpicture}
\draw[->] (0.2,0.1) -- (1.8,0.1);
\draw[->] (1.8,-0.1) -- (0.2,-0.1);
\draw[->] (2.2,0.1) -- (3.8,0.1);
\draw[->] (3.8,-0.1) -- (2.2,-0.1);
\draw[->] (-0.2,0.1) arc (15:345:0.5);
\draw[->] (4.2,-0.1) arc (-165:165:0.5);
\node at(0,0) {$x$};
\node at(2,0) {$y$};
\node at(4,0) {$z$};
\node at(-1.4,0) {$\alpha$};
\node at(5.4,0) {$\beta$};
\node at(1,0.3) {$\beta$};
\node at(1,-0.3) {$\beta$};
\node at(3,0.3) {$\alpha$};
\node at(3,-0.3) {$\alpha$};
\end{tikzpicture}$$
and $I=\langle\alpha^2-\beta^2,\alpha\beta,\beta\alpha\rangle$ be an ideal in $kQ$. Moreover, Let $Q'$ be the quiver
$$\begin{tikzpicture}
\draw[->] (-0.2,0.1) arc (15:345:0.5);
\draw[->] (0.2,-0.1) arc (-165:165:0.5);
\node at(0,0) {$v$};
\node at(-1.4,0) {$a$};
\node at(1.4,0) {$b$};
\end{tikzpicture}$$
and $I'=\langle a^2-b^2,ab,ba\rangle$ be an ideal in $kQ'$. Then there is a covering $f:(Q,I)\rightarrow(Q',I')$ of quivers with relations, which is given by $f(x)=f(y)=f(z)=v$ and $f(\alpha)=a$, $f(\beta)=b$.
\end{Ex1}

Since a quiver with relations $(Q,I)$ defines a quotient category $kQ/I$, it is natural to ask whether a covering of quivers with relations induces a covering functor between the corresponding quotient categories.  Note that in this direction there is a similar but different result which was stated without proof in \cite[Section 3.1]{BG1982}.

 \begin{Prop}  {\rm(cf. \cite[Proposition 4.5]{LL2})} \label{covering of quotient categories}
Let $M$, $N$ be $k$-categories and $F:M\rightarrow N$ be a covering functor. Let $R$ (resp. $R'$) be a class of morphisms of $M$ (resp. $N$), $I_M$ (resp. $I_N$) be the ideal of $M$ (resp. $N$) generated by $R$ (resp. $R'$). If the following conditions hold:
\begin{enumerate}[label=\rm{{(\arabic*)}}]
\item for every $\rho\in R$, $F(\rho)\in R'$;
\item for every $\rho'\in N(a,b)\cap R'$ and for each $x\in F^{-1}(a)$, there exist $y\in F^{-1}(b)$ and $\rho\in M(x,y)\cap R$ such that $F(\rho)=\rho'$;
\item for every $\rho'\in N(a,b)\cap R'$ and for each $y\in F^{-1}(b)$, there exist $x\in F^{-1}(a)$ and $\rho\in M(x,y)\cap R$ such that $F(\rho)=\rho'$,
\end{enumerate}
then $F$ induces a covering functor $\widetilde{F}:M/I_M\rightarrow N/I_N$ between the quotient categories.
\end{Prop}

\begin{proof}
By condition (1), $F$ maps morphisms in $I_M$ to morphisms in $I_N$, therefore it induces a functor $\widetilde{F}:M/I_M\rightarrow N/I_N$. To show $\widetilde{F}$ is a covering functor, it suffices to show that $F$ induces bijections $\bigoplus_{t/b}I_M(x,t)\rightarrow I_N(Fx,b)$ and $\bigoplus_{z/a}I_M(z,y)\rightarrow I_N(a,Fy)$ for every objects $a$, $b$ of $N$ and for every objects $x$, $y$ of $M$. Since $F$ is a covering functor, the above maps are injective.

For every objects $a$, $b$ of $N$ and for every object $x$ of $M$ such that $Fx=a$, let $s'\in I_N(a,b)$ be a morphism of the form $u'\rho'v'$, where $u'\in N(d,b)$, $v'\in N(a,c)$ and $\rho'\in N(c,d)\cap R'$. Since $F$ induces a bijection $\bigoplus_{z/c}M(x,z)\rightarrow N(a,c)$, there exists $(v_z)_{z/c}\in\bigoplus_{z/c}M(x,z)$ such that $\sum_{z/c}F(v_z)=v'$. By condition (2), for each $z\in F^{-1}(c)$, there exist $t_z\in F^{-1}(d)$ and $\rho_z\in M(z,t_z)\cap R$ such that $F(\rho_z)=\rho'$. For each $z\in F^{-1}(c)$, since $F$ induces a bijection $\bigoplus_{y/b}M(t_z,y)\rightarrow N(d,b)$, there exists $(u_{zy})_{y/b}\in\bigoplus_{y/b}M(t_z,y)$ such that $\sum_{y/b}F(u_{zy})=u'$. For each $y\in F^{-1}(b)$, let $s_y=\sum_{z/c}u_{zy}\rho_{z}v_{z}$. Then $s_y\in I_M(x,y)$ and \begin{multline*}\sum_{y/b}F(s_y)=\sum_{y/b}\sum_{z/c}F(u_{zy})F(\rho_z)F(v_z)=\sum_{z/c}(\sum_{y/b}F(u_{zy}))F(\rho_z)F(v_z)\\ =\sum_{z/c}u'F(\rho_z)F(v_z)=\sum_{z/c}u'\rho'F(v_z)=u'\rho'v'=s'.\end{multline*}
Since each morphism in $I_N(a,b)$ is a sum of morphisms of the form $u'\rho'v'$, where $\rho'\in R'$, we imply that the map $\bigoplus_{y/b}I_M(x,y)\rightarrow I_N(a,b)$ is surjective. For every objects $a$, $b$ of $N$ and for every object $y$ of $M$ such that $Fy=b$, it can be shown similarly that the map $\bigoplus_{x/a}I_M(x,y)\rightarrow I_N(a,b)$ is also surjective.
\end{proof}

\begin{Cor}  {\rm(cf. \cite[Corollary 4.6]{LL2})} \label{Cor}
If $f:(Q,I)\rightarrow (Q',I')$ is a covering of quivers with relations, then $f$ induces a covering functor $kQ/I\rightarrow kQ'/I'$.
\end{Cor}

\begin{proof}
By Lemma \ref{covering of quivers induces covering functor}, $f$ induces a covering functor $kf:kQ\rightarrow kQ'$ of associated path categories. Let $R$ (resp. $R'$) be the set of minimal relations together with zero relations in $I$ (resp. $I'$). Then $R$ (resp. $R'$) generates $I$ (resp. $I'$). According to Lemma \ref{preserve-minimal-relation}, for every $\rho\in R$, we have $(kf)(\rho)\in R'$, so condition $(1)$ in Proposition \ref{covering of quotient categories} holds. Since $f:(Q,I)\rightarrow (Q',I')$ is a covering of quivers with relations, condition $(2)$ and condition $(3)$ in Proposition \ref{covering of quotient categories} also hold. Therefore $f$ induces a covering functor $kQ/I\rightarrow kQ'/I'$.
\end{proof}

\subsection{Galois covering}
\

Let $G$ be a group which consists of quiver automorphisms of $Q$ or $Q$-automorphisms. For any vertex $x$ of $Q$, we denote $\overline{x}$ the $G$-orbit of $x$. Then we can define the {\it orbit quiver} $\overline{Q}:=Q/G$ as follows. The vertices of $\overline{Q}$ are given by the $G$-orbits of $Q_0$, and the arrows from $\overline{x}$ to $\overline{y}$ are given by the $G$-orbits of the arrows $s\rightarrow t$ in $Q_1$ where $s\in \overline{x}$ and $t\in \overline{y}$. The natural projection $Q\rightarrow \overline{Q}$ is a morphism of quivers.

Recall that a group $G$ of $Q$-automorphisms
\begin{enumerate}
\item is said to {\it act freely on $Q$} if for every $g\in G$ and for every $a\in Q_0$, $g(a)=a$ implies $g=id_{Q}$,
\item is said to be {\it admissible} if for each $x\in Q_0$, any $G$-orbit $G\cdot y$ contains at most one vertex from $x^{+}$ and at most one vertex from $x^{-}$.
\end{enumerate}

Clearly every $Q$-automorphism induces an automorphism of the path category $kQ$. For a quiver with relations $(Q,I)$, we consider a group $G$ of $(Q,I)$-automorphisms, that is, the $Q$-automorphisms preserving the ideal $I$ of $kQ$. Moreover, let $\overline{I}$ be the image of $I$ under the natural projection $kQ\rightarrow k\overline{Q}$, that is, for every two vertices $\overline{x},\overline{y}$ of $\overline{Q}$, $\overline{I}(\overline{x},\overline{y})$ is the $k$-vector space generated by morphisms $\overline{\rho}$, where $\rho\in I(a,b)$ with $a\in\overline{x}$, $b\in\overline{y}$ and $\overline{\rho}$ denotes the image of $\rho$ under the natural projection $kQ\rightarrow k\overline{Q}$.

\begin{Prop}  \label{elementary-properties-of-projection} Let $G$ be a group of $(Q,I)$-automorphisms, and let $\overline{I}$ be defined as above. Then we have the following.
\begin{enumerate}[label=\rm{{(\arabic*)}}]
\item The natural projection $kQ\rightarrow k\overline{Q}$ induces a morphism $(Q,I)\rightarrow (\overline{Q},\overline{I})$ of quivers with relations.
\item If $G$ acts freely on $Q$, then the natural projection $Q\rightarrow \overline{Q}$ is a covering of quivers.
\item If $G$ acts freely on $Q$, then the natural projection $kQ\rightarrow k\overline{Q}$ induces a covering $(Q,I)\rightarrow (\overline{Q},\overline{I})$ of quivers with relations.
\item Suppose that $Q$ is connected and without double arrows, and that the group $G$ acts admissibly on $Q$. Then $G$ acts freely on $Q$.
\end{enumerate}
\end{Prop}

\begin{proof}
(1) This is obvious.

(2) For each vertex $x$ of $Q$, the natural projection $Q\rightarrow \overline{Q}$ induces a surjection from the set of arrows in $Q$ starting (resp. ending) at $x$ to the set of arrows in $\overline{Q}$ starting (resp. ending) at $\overline{x}$. For two arrows $\alpha$, $\beta\in Q_1$ starting (resp. ending) at $x$, if $\overline{\alpha}=\overline{\beta}$, then $\alpha$ and $\beta$ belong to the same $G$-orbit. Suppose that $g(\alpha)=\beta$ for some $g\in G$, then $g(x)=g(s(\alpha))=s(g(\alpha))=s(\beta)=x$ (resp. $g(x)=g(t(\alpha))=t(g(\alpha))=t(\beta)=x$). Since $G$ acts freely on $Q$, this shows that $g=id_{Q}$. Therefore $\alpha=\beta$, and the natural projection $Q\rightarrow \overline{Q}$ is a covering of quivers.

(3) According to (2), the natural projection $Q\rightarrow \overline{Q}$ is a covering of quivers. To show that it induces a covering $(Q,I)\rightarrow (\overline{Q},\overline{I})$ of quivers with relations, it suffices to show that every minimal relation in $\overline{I}$ can be lifted to a minimal relation $I$, this is just \cite[Proposition 1.4]{MP1983} (compare to the proof of Lemma \ref{preserve-minimal-relation}).

(4) Suppose that  $G$ acts admissibly on $Q$ and that $g(x)=x$ for some $g\in G$ and $x\in Q_0$. Since $Q$ has no double arrows, to show that $g=id_Q$, it suffices to show that $g(y)=y$ for every vertex $y$ in $Q$. Since $Q$ is connected, there exists a walk (see definition in Section 1.4) $w$ in $Q$ from $x$ to $y$. We will show that $g(y)=y$ by induction on the length $l(w)$ of $w$. If $l(w)=1$, then $w$ or $w^{-1}$ is an arrow in $Q$. For simplicity we assume that $w\in Q_1$. Since $g(x)=x$, $g(w)$ is an arrow starting at $x$. Then $g(y)=g(t(w))=t(g(w))\in x^{+}$. Since $y,g(y)\in x^{+}$ and since $G$ acts admissibly on $Q$, we have $g(y)=y$. If $l(w)>1$, then we write $w=vw'$, where $v$ is walk of length $1$. By induction we have $g(t(w'))=t(w')$. By the same argument as the case $l(w)=1$, we have $g(y)=y$.
\end{proof}

\begin{Def} {\rm(\cite[Section 1]{MP1983})} \label{galois-covering}
A morphism $f:(Q,I)\rightarrow (Q',I')$ of quivers with relations is said to be a Galois covering if there exist a group $G$ of $(Q,I)$-automorphisms which acts freely on $Q$ and an isomorphism $\nu:(Q/G,\overline{I})\xrightarrow{\sim}(Q',I')$ of quivers with relations such that the diagram
$$\xymatrix{
		& (Q,I) \ar[dr]^{f}\ar[dl]_{\pi} &  \\
		(Q/G,\overline{I})\ar[rr]_{\nu}^{\sim} & & (Q',I')
	}$$
commutes, where $\pi:(Q,I)\rightarrow (Q/G,\overline{I})$ is the natural projection.
\end{Def}

According to this definition, if $G$ is a group of $(Q,I)$-automorphisms such that $G$ acts freely on $Q$, then the covering $(Q,I)\rightarrow (\overline{Q},\overline{I})$ of quivers with relations given in (3) of Proposition \ref{elementary-properties-of-projection}  is indeed a Galois covering.

Note that the covering map $f$ in Example~\ref{ex:simple-example} is a Galois covering with group $G=\mathbb{Z}/2\mathbb{Z}$. However, the covering map $f$ in Example \ref{a-covering-of-quivers-with-relations-but-not-Galois} is not a Galois covering since there is no nontrivial group of $(Q,I)$-automorphisms which acts freely on $Q$; we leave the reader to find a nontrivial Galois covering of $(Q',I')$ in this example.

We also remark that the original definition for Galois covering in \cite{MP1983} requests that $Q$ is connected and without double arrows and that the group action is admissible on $Q$, but in our general situation these conditions are not necessary.

We now turn to the discussion of Galois coverings between locally finite dimensional categories.

\begin{Def} {\rm(\cite[Definition 2.1]{BG1982})}  \label{locally-bounded-category}
A locally finite dimensional (resp. locally bounded) category is a $k$-category $\Lambda$ satisfying the following three conditions $(a)$, $(b)$, $(c)$ (resp. $(a)$, $(b)$, $(c')$):
\begin{itemize}
  \item[$(a)$] For each $x\in \Lambda$, the endomorphism algebra $\Lambda(x,x)$ is local.
  \item[$(b)$] Distinct objects of $\Lambda$ are not isomorphic.
  \item[$(c)$] For each $x,y\in \Lambda$, $\mathrm{dim}_{k}\Lambda(x,y)< \infty$.
  \item[$(c')$] For each $x,y\in \Lambda$, $\sum_{y\in \Lambda}\mathrm{dim}_{k}\Lambda(x,y)< \infty$ and $\sum_{y\in \Lambda}\mathrm{dim}_{k}\Lambda(y,x)< \infty$.
\end{itemize}
\end{Def}

The categories discussed above can be interpreted using quivers.

\begin{itemize}
	\item A locally finite dimensional category $\Lambda$ has the form $kQ/I$, where $Q$ is a quiver (coincides with the Gabriel quiver of $\Lambda$) and $kQ$ is the path category associated with $Q$, $I$ is an ideal of $kQ$ contained in $kQ^{\geq 2}$.
	\item If $\Lambda$ is moreover locally bounded, then $Q$ is locally finite and for each vertex $x\in Q_0$, there exists a natural number $N_x$ such that $I$ contains all paths of length $\geq N_x$ which start or terminate at $x$.
	\item In particular, any finite dimensional algebra $kQ/I$ given by a quiver $Q$ with an admissible ideal $I$ can be viewed as a locally bounded category.
\end{itemize}

The Galois coverings for locally finite dimensional categories are defined as follows.

\begin{Def} {\rm(\cite[Section 3.1]{Ga1981})} \label{galois-covering-functor}
Let $M,N$ be locally finite dimensional categories. A covering functor $F: M\rightarrow N$ is called a Galois covering if $F$ is surjective on objects and there exists a group $G$ of $k$-linear automorphisms of $M$ which acts freely on (objects of) $M$, such that $F\circ g=F$ for any $g\in G$ and $G$ acts transitively on $F^{-1}(a)$ for each object $a$ of $N$.
\end{Def}

Indeed, the notion of Galois coverings for locally finite dimensional categories is compatible with that for quivers with relations.

\begin{Prop}  {\rm(cf. \cite[Lemma 4.14]{LL2})} \label{induce-Galois-covering}
Let $(Q,I)$, $(Q',I')$ be quivers with relations such that $kQ/I$, $kQ'/I'$ are locally finite dimensional categories. A Galois covering of quivers with relations $f:(Q,I)\rightarrow (Q',I')$ induces a Galois covering $kQ/I\rightarrow kQ'/I'$ of associated categories.
\end{Prop}

\begin{proof}
We may assume that $f$ is the natural projection $(Q,I)\rightarrow (\overline{Q},\overline{I})$, where $\overline{Q}=Q/G$ with $G$ a group of $(Q,I)$-automorphisms which acts freely on $Q$. Since a Galois covering of quivers with relations is a covering of quivers with relations, by Corollary \ref{Cor}, $f$ induces a covering functor $F: kQ/I\rightarrow k\overline{Q}/\overline{I}$ of associated categories.

Denote by $\widetilde{G}$ the group of automorphisms of $kQ/I$ induced by $G$. Since $G$ acts freely on $Q$, we see that $\widetilde{G}$ acts freely on $kQ/I$. It is obvious that $F\circ g=F$ for any $g\in \widetilde{G}$ and $\widetilde{G}$ acts transitively on $F^{-1}(a)$ for each object $a$ of $k\overline{Q}/\overline{I}$. Therefore $F: kQ/I\rightarrow k\overline{Q}/\overline{I}$ is a Galois covering.
\end{proof}

\subsubsection{Galois covering between locally bounded categories}
\

When we restrict our discussion to locally bounded categories, the target category of a Galois covering is typically described via orbit categories. For completeness, we briefly review the definition of an orbit category below (see \cite[Proposition 3.1]{Ga1981} for further details).

Let $\Lambda$ be a locally finite dimensional category and $G$ a group of automorphisms such that the action of $G$ on $\Lambda$ is free and locally bounded (for each pair of objects $(x,y)$ in $\Lambda$, there are only finitely many $g\in G$ such that $\Lambda(x,gy)\neq 0$).
 The {\it orbit category} $\Lambda/G$: The objects of $\Lambda/G$ are orbits of $G$ in the set of objects of $\Lambda$. A morphism $f\colon X\rightarrow Y$ of $\Lambda/G$ is a family $$f=(_yf_x)\in\Pi_{x,y}\Lambda(x,y),$$
    where $x,y$ range over $X,Y$ respectively, and $f$ satisfies the relation $$g(_yf_x)=\ _{gy}f_{gx}$$ for all $g\in G$ and all $x,y$. The composition $e\circ f$ of $f:X\rightarrow Y$ and $e:Y\rightarrow Z$ in $\Lambda/G$ is defined by $ _z(e\circ f)_x=\sum_{y\in Y}\;_ze_{yy}f_x$; this sum makes sense since the action of $G$ is locally bounded.
	
Same as in Definition \ref{galois-covering}, there is a canonical projection $F: \Lambda\rightarrow\Lambda/G$ which maps an object $x$ onto its orbit; it maps a morphism $\phi\in \Lambda(x,y)$ onto the family $F\phi$ such that $_{hy}F\phi_{gx}=g\phi$ or $0$ according as $h=g$ or $h\neq g\in G$. It turns out that $F$ is a Galois covering. Moreover, if $E:\Lambda\rightarrow \Lambda'$ is a Galois covering with group $G$ and $F: \Lambda\rightarrow \Lambda/G$ is the canonical projection, then there exists an isomorphism $H:\Lambda/G\xrightarrow{\sim}\Lambda'$ such that $E=HF$ (see \cite[Remark 3.1]{Ga1981}).

For a given locally bounded category $\Lambda$ (corresponding to $(Q,I)$) and a Galois covering $f:(Q,I)\rightarrow (Q',I')$ with group $G$, denote by $\widetilde{G}$ the group of automorphisms of $\Lambda$ induced by $G$. By Proposition \ref{induce-Galois-covering}, $f$ induces a Galois covering functor $\Lambda\rightarrow\Lambda'$ with group $\widetilde{G}$, where $\Lambda'=kQ'/I'$. Therefore $\Lambda'$ is isomorphic to $\Lambda/\widetilde{G}$.

Consequently, for Galois coverings between locally bounded categories, we shall (by mild abuse of terminology) identify these two perspectives (Definition \ref{galois-covering-functor} and the orbit category viewpoint) interchangeably, particularly in Section 3.

We note that Galois coverings and orbit categories can be viewed as an effective way to understand skew group algebras, which are in general not basic algebras.

\begin{Def}
	Let $A$ be a $k$-algebra and let $G$ be a finite group acting on $A$ by algebra automorphisms. The skew group algebra $A * G$ is defined as the vector space
$$
A * G=\bigoplus_{g\in G} Ag,
$$
with multiplication determined by
$$
(a g)(b h)=a\, g(b)\, gh
$$
for all $a,b\in A$ and $g,h\in G$.
\end{Def}

Recall our setting in Definition~\ref{galois-covering}. When considering a Galois covering
$$
f:(Q,I)\rightarrow (Q',I')
$$
of quivers with relations with finite group $G$, the group $G$ acts freely on $Q$ by $(Q,I)$-automorphisms, and hence induces a group of automorphisms of the algebra $A=kQ/I$. Therefore, one can naturally define the skew group algebra $A*G$ and relate it to the orbit algebra $kQ'/I'$. This point of view is sometimes useful since $A$ is a subalgebra of the skew group algebra $A*G$ (see \cite{MMM2004}).

\begin{Prop} {\rm(cf. \cite[Section 5.2]{RR1985} and \cite[Section 3]{Asa2011})}
	Under the above notation, the skew group algebra $A*G$ is Morita equivalent to the orbit algebra $kQ'/I'$.
\end{Prop}

We conclude this section with the following simple example.

\begin{Ex1}\textnormal{({Example \ref{ex:simple-example} revisited})}
	Let $A = kQ/I$, and let
$$
G = \mathbb{Z}/2\mathbb{Z} = \langle g \rangle
$$
act on $A$ by exchanging the vertices $1, 2$ and the arrows $\alpha, \beta$. Then there is an algebra isomorphism
$$
\Phi : A \ast G \longrightarrow M_2(k[x]/\langle x^2 \rangle)
$$
given by
$$
\Phi(e_1) = E_{11}, \qquad \Phi(e_2) = E_{22},
$$
$$
\Phi(\alpha) = xE_{21}, \qquad \Phi(\beta) = xE_{12},
$$
$$
\Phi(e_1 g) = E_{12}, \qquad \Phi(e_2 g) = E_{21},
$$
$$
\Phi(\alpha g) = xE_{22}, \qquad \Phi(\beta g) = xE_{11}.
$$
A direct computation shows that $\Phi$ preserves all defining relations and is surjective. Since both algebras have dimension $8$ over $k$, $\Phi$ is an isomorphism. Therefore,
$$
A \ast G \cong M_2(k[x]/\langle x^2 \rangle).
$$
Hence $A \ast G$ is Morita equivalent to $k[x]/\langle x^2 \rangle \cong kQ'/I'.$
\end{Ex1}

\subsection{Universal cover}
\

In the following we always assume that the quivers considered are connected. For a quiver $Q$, a {\it walk} of $Q$ is a sequence of the form $\alpha_n\cdots\alpha_2\alpha_1$, where each $\alpha_i$ is an arrow or a formal inverse of an arrow in $Q$, such that the source of $\alpha_i$ is equal to the terminal of $\alpha_{i-1}$ for each $i$. Define an equivalence relation $\sim$ on the set of walks of $Q$: $\sim$ is generated by
\begin{itemize}
	\item $\alpha^{-1}\alpha\sim 1_{s(\alpha)}$ and $\alpha\alpha^{-1}\sim 1_{t(\alpha)}$ for every arrow $\alpha$ of $Q$;
    \item If $w_1\sim w_2$, then $uw_1 v\sim uw_2 v$ for every walks $u,v$ of $Q$ whenever the compositions make sense.
\end{itemize}
For a walk $w$ of $Q$, we denote $[w]$ the equivalence class of walks containing $w$.

For a vertex $x$ of $Q$, denote $\Pi(Q,x)$ the {\it fundamental group} of $Q$ at $x$, which is the group of equivalence classes of closed walks of $Q$ at $x$ under concatenation.

We now turn to the definition of the fundamental group of a quiver with relations, which is first introduced in \cite{MP1983}.

\begin{Def}\label{homotopy-relation-with-respect-to-ideal}
Let $Q$ be a quiver and $I$ be an ideal of the path category $kQ$, denote $m(I)$ the set of minimal relations of $I$. Let $\sim_{I}$ be the equivalence relation on the set of walks of $Q$ generated by
\begin{itemize}
	\item[$(E1)$] $\alpha^{-1}\alpha\sim_{I} 1_{s(\alpha)}$ and $\alpha\alpha^{-1}\sim_{I} 1_{t(\alpha)}$ for every arrow $\alpha$ of $Q$;
    \item[$(E2)$] $w_1\sim_{I} w_2$ if there exists some $\rho=\sum^{n}_{i=1}\lambda_i p_i\in m(I)$ with $w_1=p_1$ and $w_2=p_2$;
    \item[$(E3)$] If $w_1\sim_{I} w_2$, then $uw_1 v\sim_{I} uw_2 v$ for every walks $u,v$ of $Q$ whenever the compositions make sense.
\end{itemize}
\end{Def}

The fundamental group $\Pi(Q,I)$ at a vertex $x\in Q_0$ is defined to be the group of equivalence classes of closed walks of $Q$ at $x$ under the equivalence relation $\sim_{I}$. Clearly this definition is independent to the choice of the vertex $x$. It is straightforward to show that $$\Pi(Q,I)\cong \Pi(Q,x)/N(Q,m(I),x),$$
where $x$ is a vertex of $Q$ and $N(Q,m(I),x)$ is the normal subgroup of $\Pi(Q,x)$ generated by elements of the form $[\gamma^{-1}u^{-1}v\gamma]$, where $\gamma$ is a walk of $Q$ from $x$ to $y$ and $u,v$ are paths of $Q$ from $y$ to $z$ such that there exists some $\rho=\sum^{n}_{i=1}\lambda_i p_i\in m(I)$ with $u=p_1$ and $v=p_2$.

\begin{Rem1}
By standard argument in covering theory (cf. \cite[Theorem 4.1 of Chapter five]{Ma1977}), a covering $f:(Q,I)\rightarrow (Q',I')$ of quivers with relations corresponds to an injection $\Pi(Q,I)\hookrightarrow \Pi(Q',I')$ of fundamental groups.
\end{Rem1}

\begin{Def} {\rm(cf. \cite[Corollary 1.5]{MP1983})} \label{universal-cover-def}
Let $(Q,I)$ be a quiver with relation. A universal cover of $(Q,I)$ is a covering $\pi: (\widetilde{Q},\widetilde{I})\rightarrow (Q,I)$ such that
\begin{itemize}
\item[$(1)$] for any other covering $f: (\overline{Q},\overline{I})\rightarrow (Q,I)$, there exists a covering $\pi': (\widetilde{Q},\widetilde{I})\rightarrow (\overline{Q},\overline{I})$ with $f\pi'=\pi$;
\item[$(2)$] if, moreover, $\widetilde{x}\in \widetilde{Q}_0$ and $\overline{x}\in \overline{Q}_0$ are such that $\pi\widetilde{x}=f\overline{x}$, then $\pi'$ can be uniquely chosen so that $\pi'(\widetilde{x})=\overline{x}$.
\end{itemize}
\end{Def}

\begin{Lem} \label{trivial-fundamental-group}
Let $\pi:(R,J)\rightarrow (Q,I)$ be a covering of quivers with relations. If $\Pi(R,J)=\{1\}$, then $\pi$ is a universal cover of $(Q,I)$.
\end{Lem}

\begin{proof}
It suffices to show that for any covering $f:(Q',I')\rightarrow(Q,I)$ and for any fixed $x\in R_0$, $x'\in Q'_0$ with $\pi(x)=f(x')$, there exists a unique covering $\pi':(R,J)\rightarrow (Q',I')$ such that $f\pi'=\pi$ and $\pi'(x)=x'$.

For any $y\in R_0$, choose a walk $w$ in $R$ from $x$ to $y$. Then $\pi(w)$ is a walk in $Q$ starting at $\pi(x)$. Since $f$ is a covering, there exists a unique walk $w'$ in $Q'$ starting at $x'$ with $f(w')=\pi(w)$, and define $\pi'(y)$ as the terminal of $w'$. We need to show that the value of $\pi'(y)$ is independent to the choice of the walk $w$. If $v$ is another walk in $R$ from $x$ to $y$, since $\Pi(R,J)=\{1\}$, we have $w\sim_{J}v$. According to Lemma \ref{preserve-minimal-relation}, $\pi(w)\sim_{I}\pi(v)$. Let $v'$ be the walk in $Q'$ starting at $x'$ with $f(v')=\pi(v)$. Since $f:(Q',I')\rightarrow(Q,I)$ is a covering of quivers with relations, any minimal relation in $I$ can be lifted to a minimal relation in $I'$. Therefore $w'\sim_{I'}v'$, and $w'$ and $v'$ have the same terminal.

For each arrow $\alpha:y\rightarrow z$ in $R$, since $f(\pi'(y))=\pi(y)$ and since $f$ is a covering of quivers, there exists a unique arrow $\alpha'$ in $Q'$ starting at $\pi'(y)$ with $f(\alpha')=\pi(\alpha)$. We define $\pi'(\alpha)=\alpha'$. Then $\pi':R\rightarrow Q'$ becomes a covering of quivers with $f\pi'=\pi$ and $\pi'(x)=x'$.

For any $\rho\in J(y,z)$ with $\rho$ a minimal relation or a zero relation in $J$, by Lemma \ref{preserve-minimal-relation} $(k\pi)(\rho)$ is a minimal relation or a zero relation in $I$. Since $f:(Q',I')\rightarrow(Q,I)$ is a covering, there exists $z'\in Q'_0$ and $\rho'\in I'(\pi'(y),z')$ such that $(kf)(\rho')=(k\pi)(\rho)$. Since both $(k\pi')(\rho)$ and $\rho'$ are relations starting at $\pi'(y)$ and since $(kf)((k\pi')(\rho))=(kf)(\rho')$, we have $(k\pi')(\rho)=\rho'\in I'$. Therefore $k\pi':kR\rightarrow kQ'$ maps the morphisms in $I$ to the morphisms in $I'$.

For any $\rho'\in I'(y',z')$ with $\rho'$ a minimal relation or a zero relation in $I'$, and for any $y\in R_0$ with $\pi'(y)=y'$, by Lemma \ref{preserve-minimal-relation} $(kf)(\rho')$ is a minimal relation or a zero relation in $I$. Since $\pi:(R,J)\rightarrow(Q,I)$ is a covering, there exists $z\in R_0$ and $\rho\in J(y,z)$ such that $(k\pi)(\rho)=(kf)(\rho')$. Since both $(k\pi')(\rho)$ and $\rho'$ are relations starting at $y'$ and since $(kf)((k\pi')(\rho))=(kf)(\rho')$, we have $(k\pi')(\rho)=\rho'$. Similarly, for any $z\in R_0$ with $\pi'(z)=z'$, there exists $y\in R_0$ and $\rho\in J(y,z)$ such that $(k\pi')(\rho)=\rho'$. Therefore $\pi':(R,J)\rightarrow(Q',I')$ is a covering of quivers with relations such that $f\pi'=\pi$ and $\pi'(x)=x'$. Since $R$ is connected, it can be shown that such $\pi'$ is unique.
\end{proof}

\begin{Prop}  {\rm(\cite[Corollary 1.5]{MP1983})} \label{universal-cover}
Let $(Q,I)$ be a quiver with relations. Then there exists a universal cover $\pi:(\widetilde{Q},\widetilde{I})\rightarrow(Q,I)$, which is a Galois covering with group $\Pi(Q,I)$.
\end{Prop}

\begin{proof}
We first construct a covering $(\widetilde{Q},\widetilde{I})$ of $(Q,I)$ as follows. Fix a vertex $x$ of $Q$. For a walk $u$ of $Q$, denote $\widetilde{u}$ the equivalence class of $u$ under $\sim_{I}$. Let
$$\widetilde{Q}_0:=\{\widetilde{u}\mid u \mbox{ is a walk of } Q\mbox{ with }s(u)=x\},$$
$$\widetilde{Q}_1:=\{(\widetilde{u},\alpha)\mid \widetilde{u}\in \widetilde{Q}_0, \alpha\in Q_1, s(\alpha)=t(u)\},$$
and define $s(\widetilde{u},\alpha)=\widetilde{u}, t(\widetilde{u},\alpha)=\widetilde{\alpha u}$. Then it is easy to see that $\pi: \widetilde{Q}\rightarrow Q$  $(\widetilde{u}\mapsto t(u), (\widetilde{u},\alpha)\mapsto \alpha)$ is a covering of quivers.

For $a,b\in \widetilde{Q}_0$, denote $\widetilde{I}(a,b)=(k\pi)^{-1}I(\pi(a),\pi(b))$. Then $\pi: \widetilde{Q}\rightarrow Q$ induces a morphism of quivers and relations $\pi: (\widetilde{Q},\widetilde{I})\rightarrow (Q,I)$. Moreover, the fundamental group $\Pi(Q,I)$ acts on $(\widetilde{Q},\widetilde{I})$: $\widetilde{v}\cdot \widetilde{u}=\widetilde{uv^{-1}}$ for $\widetilde{u}\in \widetilde{Q}_0, \widetilde{v}\in \Pi(Q,I)$. Clearly $\Pi(Q,I)$ acts freely on $\widetilde{Q}$. We will show that $\pi: (\widetilde{Q},\widetilde{I})\rightarrow (Q,I)$ is a Galois covering of quivers with relations with group $\Pi(Q,I)$.

First we show that $\pi: (\widetilde{Q},\widetilde{I})\rightarrow (Q,I)$ is a covering of quivers with relations. For a minimal relation $\rho=\sum_{i=1}^{n}\lambda_i p_i\in I(y,z)$ and for $\widetilde{u}\in \widetilde{Q}_0$ with $\pi(\widetilde{u})=y$, denote $q_i$ the path in $\widetilde{Q}$ such that $s(q_i)=\widetilde{u}$ and $\pi(q_i)=p_i$. By the definition of $\pi$, the terminal of $q_i$ is $\widetilde{p_i u}$. Since $\rho=\sum_{i=1}^{n}\lambda_i p_i$ is a minimal relation in $I$, $\widetilde{p_i u}=\widetilde{p_j u}$ for all $i,j\in\{1,2,\cdots,n\}$. Denote $\widetilde{v}=\widetilde{p_1 u}$, then $\widetilde{\rho}=\sum_{i=1}^{n}\lambda_i q_i$ belongs to $\widetilde{I}(\widetilde{u},\widetilde{v})$ and $(k\pi)(\widetilde{\rho})=\rho$. Similarly, for a minimal relation $\rho=\sum_{i=1}^{n}\lambda_i p_i\in I(y,z)$ and for $\widetilde{v}\in \widetilde{Q}_0$ with $\pi(\widetilde{v})=z$, there exists $\widetilde{u}\in\widetilde{Q}_0$ with $\pi(\widetilde{u})=y$ and $\widetilde{\rho}\in\widetilde{I}(\widetilde{u},\widetilde{v})$ such that $(k\pi)(\widetilde{\rho})=\rho$.

Then we show that $\pi: (\widetilde{Q},\widetilde{I})\rightarrow (Q,I)$ is a Galois covering with group $\Pi(Q,I)$. Since for any $\widetilde{v}\in \Pi(Q,I)$, the following diagram $$\xymatrix{
		(\widetilde{Q},\widetilde{I})\ar[rr]^{\widetilde{v}}\ar[dr]_{\pi} & & (\widetilde{Q},\widetilde{I})\ar[dl]^{\pi} \\
		 & (Q,I) &
	}$$
commutes, $\pi: (\widetilde{Q},\widetilde{I})\rightarrow (Q,I)$ induces a morphism of quivers with relations
$$\mu: (\widetilde{Q}/\Pi(Q,I),J)\rightarrow (Q,I),$$
where $J$ is the image of $\widetilde{I}$ under the natural projection $k\widetilde{Q}\rightarrow k(\widetilde{Q}/\Pi(Q,I))$. It is straightforward to show that $\mu: \widetilde{Q}/\Pi(Q,I)\rightarrow Q$ is an isomorphism of quivers. To show that $\mu$ is an isomorphism of quivers with relations, it suffices to show that for any vertices $a,b$ of $\widetilde{Q}/\Pi(Q,I)$, $k\mu:J(a,b)\rightarrow I(\mu(a),\mu(b))$ is bijective. Since $\mu$ is an isomorphism of quivers, the map $k\mu:J(a,b)\rightarrow I(\mu(a),\mu(b))$ is injective. For any $\rho\in I(\mu(a),\mu(b))$ with $\rho$ a minimal relation or a zero relation, since $\pi: (\widetilde{Q},\widetilde{I})\rightarrow (Q,I)$ is a covering, there exists  $\widetilde{u},\widetilde{v}\in\widetilde{Q}_0$ and $\widetilde{\rho}\in\widetilde{I}(\widetilde{u},\widetilde{v})$ such that $(k\pi)(\widetilde{\rho})=\rho$. Let $\xi$ be the image of $\widetilde{\rho}$ in $k(\widetilde{Q}/\Pi(Q,I))$, then $\xi\in J(a,b)$ and $\mu(\xi)=\rho$. Therefore $k\mu$ is also surjective.

According to Lemma \ref{trivial-fundamental-group}, to show that $\pi: (\widetilde{Q},\widetilde{I})\rightarrow (Q,I)$ is a universal cover of $(Q,I)$, it suffices to show that $\Pi(\widetilde{Q},\widetilde{I})=\{1\}$. Let $\gamma$ be a closed walk of $\widetilde{Q}$ at $\widetilde{1_x}$ and let $u=\pi(\gamma)$ be a closed walk of $Q$ at $x$. Then the terminal of $\gamma$ is $\widetilde{u}$, which implies that $\widetilde{u}=\widetilde{1_x}$ or $u\sim_I1_x$. Since $\pi: (\widetilde{Q},\widetilde{I})\rightarrow (Q,I)$ is a covering of quivers with relations, $\gamma\sim_{\widetilde{I}}1_{\widetilde{1_x}}$ and $\Pi(\widetilde{Q},\widetilde{I})=\{1\}$.
\end{proof}

\begin{Cor} \label{trivial-fundamental-group-the-converse} If $\pi: (R,J)\rightarrow (Q,I)$ is a universal cover of $(Q,I)$, then $\Pi(R,J)=\{1\}$.
\end{Cor}

\begin{proof} According to Proposition \ref{universal-cover}, there is a universal cover $(\widetilde{Q},\widetilde{I})$ of $(Q,I)$ with trivial fundamental group.
However, any two universal covers of $(Q,I)$ are isomorphic, it follows that the fundamental group $\Pi(R,J)$ is also trivial.	
\end{proof}

Combining Lemma \ref{trivial-fundamental-group} with Corollary \ref{trivial-fundamental-group-the-converse}, we know that $\pi: (R,J)\rightarrow (Q,I)$ is a universal cover of $(Q,I)$ if and only if $\Pi(R,J)=\{1\}$. In particular, if $I$ is generated by zero relations, then $\widetilde{Q}$ is just the topological universal cover of $Q$ and it must be a (finite or infinite) tree since the fundamental group of any graph is a free group generated by the edges not belonging to a spanning tree of the graph (cf. \cite[Theorem 5.2 of Chapter six]{Ma1977}). Note also that if $\pi: (\widetilde{Q},\widetilde{I})\rightarrow (Q,I)$ is a universal cover as in Proposition \ref{universal-cover} and if $Q$ contains no double arrows, then $\Pi(Q,I)$ acts on $(\widetilde{Q},\widetilde{I})$ admissibly. Indeed, suppose that there are two different vertices $y,z\in x^{+}$ of $\widetilde{Q}$ with $y$, $z$ belonging to the same $\Pi(Q,I)$-orbit. Let $\alpha:x\rightarrow y$ and $\beta:x\rightarrow z$ be two arrows in $\widetilde{Q}$. Since $y$, $z$ belong to the same $\Pi(Q,I)$-orbit, the images of $y$, $z$ in $Q$ are equal. Since $Q$ contains no double arrows, the images of $\alpha$, $\beta$ in $Q$ are equal, contradicts the fact that $\widetilde{Q}\rightarrow Q$ is a covering of quivers.

\begin{Ex1}\label{Kron-example}
Let $Q$ be the quiver
$$\begin{tikzpicture}
\draw[->] (0.3,0.1) -- (1.7,0.1);
\draw[->] (0.3,-0.1) -- (1.7,-0.1);
\node at(0,0) {$y$};
\node at(2,0) {$x$};
\node at(1,0.3) {$a$};
\node at(1,-0.4) {$b$};
\end{tikzpicture}$$
and $I$ be the zero ideal in $kQ$. Then $(\widetilde{Q},\widetilde{I})$ is a universal cover of $(Q,I)$, where $\widetilde{Q}$ is the quiver
$$\begin{tikzpicture}
\draw[->] (-3.2,0.8) -- (-3.8,0.2);
\draw[->] (-2.8,0.8) -- (-2.2,0.2);
\draw[->] (-1.2,0.8) -- (-1.8,0.2);
\draw[->] (-0.8,0.8) -- (-0.2,0.2);
\draw[->] (0.8,0.8) -- (0.2,0.2);
\draw[->] (1.2,0.8) -- (1.8,0.2);
\draw[->] (2.8,0.8) -- (2.2,0.2);
\draw[->] (3.2,0.8) -- (3.8,0.2);
\node at(-4,0) {$-4$};
\node at(-3,1) {$-3$};
\node at(-2,0) {$-2$};
\node at(-1,1) {$-1$};
\node at(0,0) {$0$};
\node at(2,0) {$2$};
\node at(1,1) {$1$};
\node at(3,1) {$3$};
\node at(4,0) {$4$};
\node at(-5,0.5) {$\cdots$};
\node at(5,0.5) {$\cdots$};
\node at(-3.65,0.65) {$\alpha$};
\node at(-2.35,0.7) {$\beta$};
\node at(-1.65,0.65) {$\alpha$};
\node at(-0.35,0.7) {$\beta$};
\node at(0.35,0.65) {$\alpha$};
\node at(1.65,0.7) {$\beta$};
\node at(2.35,0.65) {$\alpha$};
\node at(3.65,0.7) {$\beta$};
\end{tikzpicture}$$
and $\widetilde{I}$ is the zero ideal in $k\widetilde{Q}$. Note that $\Pi(Q,I)$ is an infinite cyclic group generated by the walk $ba^{-1}$, which induces an automorphism of $\widetilde{Q}$ sending each vertex $i$ to $i-2$. So the action of $\Pi(Q,I)$ on $(\widetilde{Q},\widetilde{I})$ is not admissible.
\end{Ex1}

\begin{Ex1} \label{universal-cover-example1}
$$\xymatrix{
& & \cdot \ar[dr]^{\alpha_2} & & \\Q: & x\cdot \ar[ur]^{\alpha_1}\ar[dr]_{\alpha_3} & & \cdot & I_1=\langle \alpha_2\alpha_1\rangle, &I_2=\langle \alpha_2\alpha_1-\alpha_4\alpha_3\rangle. \\ & & \cdot \ar[ur]_{\alpha_4} & & &
 }$$
In case the ideal $I_1$, $\widetilde{Q}$ is the topological universal cover of $Q$ and it follows that the universal cover of $(Q,I_1)$ is $(\widetilde{Q},\widetilde{I_1})$, where $\widetilde{Q}$ is given by
$$\xymatrix{
\cdots\quad \cdot \ar[r]^{\widetilde{\alpha_1}}&\cdot \ar[r]^{\widetilde{\alpha_2}} & \cdot& \cdot\ar[l]_{\widetilde{\alpha_4}}&\cdot\ar[l]_{\widetilde{\alpha_3}}\ar[r]^{\widetilde{\alpha_1}}&\cdot \ar[r]^{\widetilde{\alpha_2}} & \cdot& \cdot\ar[l]_{\widetilde{\alpha_4}}&\cdot\ar[l]_{\widetilde{\alpha_3}} \quad \cdots
}$$
and  $\widetilde{I_1}$ is generated by all elements of the form $\widetilde{\alpha_2}\widetilde{\alpha_1}$.

\noindent In case the ideal $I_2$, $\widetilde{Q}_0=\{\widetilde{1_x}, \widetilde{\alpha_1}, \widetilde{\alpha_2\alpha_1}, \widetilde{\alpha_3}\}$ and it follows that $\widetilde{Q}\cong Q$ and $(\widetilde{Q},\widetilde{I_2})\cong(Q,I_2)$.
\end{Ex1}

\begin{Ex1} {\rm(by Riedtmann)} \label{universal-cover-example2}
$$\xymatrix{
Q: \quad & 1\ar@/^/[r]^{\beta}\ar@(dl,ul)^{\alpha} & 2\ar@/^/[l]^{\gamma} & I_1=\langle\alpha^2-\gamma\beta,\beta\gamma-\beta\alpha\gamma,\alpha^4\rangle, & I_2=\langle\alpha^2-\gamma\beta,\beta\gamma\rangle.
 }$$
The universal cover of $(Q,I_1)$ is itself since $\Pi(Q,I_1)=\{1\}$.

\noindent The universal cover of $(Q,I_2)$ is $(\widetilde{Q},\widetilde{I_2})$, where $\widetilde{Q}$ is given by
$$\xymatrix{
\cdots&\widetilde{1}\ar[r]^{\widetilde{\beta}}\ar[dr]^{\widetilde{\alpha}}& \widetilde{2}\ar[r]^{\widetilde{\gamma}}&\widetilde{1}\ar[r]^{\widetilde{\beta}}\ar[dr]^{\widetilde{\alpha}}& \widetilde{2}\ar[r]^{\widetilde{\gamma}}&\widetilde{1} \cdots\\
\cdots\widetilde{1}\ar[r]_{\widetilde{\beta}}\ar[ur]^{\widetilde{\alpha}}&\widetilde{2}\ar[r]_{\widetilde{\gamma}}&\widetilde{1}\ar[r]_{\widetilde{\beta}}\ar[ur]^{\widetilde{\alpha}}&
\widetilde{2}\ar[r]_{\widetilde{\gamma}}&\widetilde{1}\ar[r]_{\widetilde{\beta}}\ar[ur]^{\widetilde{\alpha}}&\cdots
}$$
and $\widetilde{I_2}$ is generated by all elements of the form $\widetilde{\alpha}^2-\widetilde{\gamma}\widetilde{\beta}, \widetilde{\beta}\widetilde{\gamma}$. Note that $\Pi(Q,I_2)\cong \mathbb{Z}$ and it is generated by $[\alpha]$.
\end{Ex1}

Note that when the characteristic of the field $k$ is not equal to $2$,  $kQ/I_2\cong kQ/I_1$ (by $e_i\mapsto e_i, \alpha\mapsto \alpha-\frac{1}{2}\alpha^2, \beta\mapsto \beta-\beta\alpha,\gamma\mapsto\gamma$) in Example \ref{universal-cover-example2}, therefore the universal cover depends on the choice of an ideal.

\subsection{Standardness}
\

Recall that for a locally bounded category $\Lambda$, a {\it finitely generated} $\Lambda$-module $\ell$ is a contravariant functor $\ell\colon\Lambda\rightarrow \mathrm{MOD} k$ such that $\sum_{x\in \Lambda}\mathrm{dim}_{k}\ell(x)<\infty$, where $\mathrm{MOD} k$ is the category of $k$-vector spaces. We denote by mod$\Lambda$ the category of all finitely generated $\Lambda$-modules. Note that if $\Lambda=kQ/I$ for a quiver with relations $(Q,I)$, then mod$\Lambda$ is equivalent to the category of finite dimensional right $A$-modules, where $A$ is the quiver algebra $kQ/I$ defined by $(Q,I)$. We also denote mod$\Lambda$ as mod$(Q,I)$ if $\Lambda=kQ/I$.

\begin{Def} {\rm(\cite[Definition 2.2]{BG1982})}  \label{locally-representation-finite-category}
A locally representation-finite category is a locally bounded category $\Lambda$ such that for every object $x$ of $\Lambda$, the number of isomorphism classes of finitely generated indecomposable $\Lambda$-module $\ell$ such that $\ell(x)\neq 0$ is finite.
\end{Def}

Note that if $\Lambda$ is locally representation-finite and connected, then it admits the Auslander-Reiten sequences and its Auslander-Reiten quiver is also connected.

For a translation quiver $\Gamma$, we let $k\Gamma$ be its path category, and let $k(\Gamma)$ be the mesh category of $\Gamma$, which is a factor category of $k\Gamma$ by the mesh ideal. For a locally bounded category $\Lambda$, we denote by ind$\Lambda$ the category formed by chosen representatives of the finitely generated indecomposable modules and by $[\mathrm{ind}\Lambda]$ the set of isomorphism classes of the finitely generated indecomposable modules. It is easy to see that a locally bounded category $\Lambda$ is locally representation-finite if and only if ind$\Lambda$ is locally bounded.

Since two different admissible ideal presentations $(Q,I_1)$, $(Q,I_2)$ of a locally bounded category $\Lambda$ may lead to two different universal covers (see Example \ref{universal-cover-example2}), we can not define the universal cover of a locally bounded category $\Lambda$ directly. However, if $\Lambda$ is locally representation-finite, then there exists a locally representation-finite category $\widetilde{\Lambda}$ (which is called the universal cover of $\Lambda$) which is determined by $\Lambda$ up to isomorphism together with a covering functor $\widetilde{\Lambda}\rightarrow\Lambda$ (see \cite[Section 2.1]{Ga1981}), which can be described as follows.

Let $\Lambda$ be a locally representation-finite category. According to \cite[Section 2.1]{Ga1981}, the universal cover $\widetilde{\Lambda}$ of $\Lambda$ is defined to be the full subcategory of $k(\widetilde{\Gamma}_{\Lambda})$ formed by projective vertices of $\widetilde{\Gamma}_{\Lambda}$, where $\widetilde{\Gamma}_{\Lambda}$ denotes the universal cover of the Auslander-Reiten quiver $\Gamma_{\Lambda}$ of $\Lambda$ (see \cite[Section 1.3]{BG1982}). Moreover, there exists a covering functor $F:\widetilde{\Lambda}\rightarrow\Lambda$, which is given by the commutative diagram
$$\xymatrix@R=1.3pc{
	k(\widetilde{\Gamma}_{\Lambda})\ar[r]^{E} & \mathrm{ind}\Lambda \\
    \widetilde{\Lambda}\ar[r]^{F}\ar@{^(-_>}[u]^{\mathrm{incl.}} &\Lambda\ar@{^(-_>}[u]_{(-)^{*}}, \\
}$$
where $(-)^{*}:\Lambda\rightarrow\mathrm{ind}\Lambda$, $a\mapsto\Lambda(-,a)$ is the Yoneda embedding and $E:k(\widetilde{\Gamma}_{\Lambda})\rightarrow \mathrm{ind}\Lambda$ is a well-behaved covering functor (see \cite[Section 3.1]{BG1982}). However, the above covering functor $F:\widetilde{\Lambda}\rightarrow\Lambda$ is in general not a Galois covering of locally bounded categories.

\begin{Ex1}\label{universal-cover-example3}
Let $k=\overline{k}$ be a field of characteristic $2$ and let $\Lambda$ be the locally representation-finite category given by the quiver with relations $(Q,I_1)$ in Example \ref{universal-cover-example2}. Since $\Lambda$ and $\Lambda_0=kQ/I_2$ have isomorphic Auslander-Reiten quivers, where $(Q,I_2)$ is the quiver with relations given in Example \ref{universal-cover-example2}, the universal covers of $\Lambda$ and $\Lambda_0$ are isomorphic. Since $\Lambda_0$ is a standard RFS category of type $(D_6,\frac{1}{3},1)$ (see \cite{Asashiba2003}), according to Remark \ref{remark-universal-cover-of-standard-lrf-category}, the universal cover $\widetilde{\Lambda_0}$ of $\Lambda_0$ is given by the quiver with relations $(\widetilde{Q},\widetilde{I_2})$ in Example \ref{universal-cover-example2}. Therefore the universal cover $\widetilde{\Lambda}$ of $\Lambda$ is also given by the quiver with relations $(\widetilde{Q},\widetilde{I_2})$.

Note that there is no covering map from $(\widetilde{Q},\widetilde{I_2})$ to $(Q,I_1)$. However, there exists a covering functor $F:\widetilde{\Lambda}\rightarrow\Lambda$ which can be described as follows. Let $R$ be the quiver
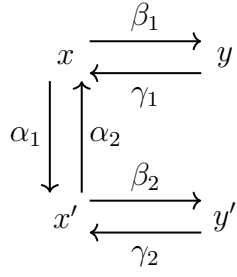
\begin{center}
\tikzset{every picture/.style={line width=0.75pt}}
\begin{tikzpicture}[x=30pt,y=30pt,yscale=1,xscale=1]
\node at(0,0) {$x'$};
\node at(2,0) {$y'$};
\node at(0,2) {$x$};
\node at(2,2) {$y$};
\node at(-0.5,1) {$\alpha_1$};
\node at(0.5,1) {$\alpha_2$};
\node at(1,2.5) {$\beta_1$};
\node at(1,1.5) {$\gamma_1$};
\node at(1,0.5) {$\beta_2$};
\node at(1,-0.5) {$\gamma_2$};
\draw[->]    (0.3,0.2)--(1.7,0.2) ;
\draw[->]    (1.7,-0.2)--(0.3,-0.2) ;
\draw[->]    (0.3,2.2)--(1.7,2.2) ;
\draw[->]    (1.7,1.8)--(0.3,1.8) ;
\draw[->]    (-0.2,1.7)--(-0.2,0.3) ;
\draw[->]    (0.2,0.3)--(0.2,1.7) ;
\end{tikzpicture}
\end{center}
and $J=\langle\alpha_2\alpha_1-\gamma_1\beta_1,\alpha_1\alpha_2-\gamma_2\beta_2,\beta_1\gamma_1,\beta_2\gamma_2\rangle$. Then $\Lambda'=kR/J$ is a locally representation-finite category and the corresponding algebra $A'$ of $\Lambda'$ is an RFS algebra of type $(D_6,\frac{2}{3},1)$ (see \cite{Asashiba2003}). Denote by $G$ the subgroup of $\Pi(Q,I_2)$ generated by $[\alpha^2]$. Then $G$ acts on $(\widetilde{Q},\widetilde{I_2})$ and there exists a Galois covering $(\widetilde{Q},\widetilde{I_2})\rightarrow(R,J)$ with group $G$, which induces a Galois covering functor $H:\widetilde{\Lambda}\rightarrow\Lambda'$. Moreover, we have a covering functor $H':\Lambda'\rightarrow\Lambda$, where $H'(x)=H'(x')=1$, $H'(y)=H'(y')=2$, $H'(\alpha_1)=\alpha-\alpha^2$, $H'(\alpha_2)=\alpha$, $H'(\beta_1)=H'(\beta_2)=\beta-\beta\alpha$, $H'(\gamma_1)=H'(\gamma_2)=\gamma$. Then $F=H'H$ becomes a covering functor from $\widetilde{\Lambda}$ to $\Lambda$.
\end{Ex1}

\begin{Def} {\rm(\cite[Definition 5.1]{BG1982})} \label{standard l.r.f. category}
A locally representation-finite category $\Lambda$ is said to be standard if $k(\Gamma_{\Lambda})$ is isomorphic to $\mathrm{ind}\Lambda$, where $\Gamma_{\Lambda}$ is the Auslander-Reiten quiver of $\Lambda$.
\end{Def}

\begin{Rem1}\label{remark-universal-cover-of-standard-lrf-category}
Let $\Lambda$ be a standard locally representation-finite category. According to \cite[Theorem 3.8]{MP1983}, there exists an admissible ideal presentation $kQ/I$ of $\Lambda$ such that the quiver $\widetilde{Q}$ of the universal cover $(\widetilde{Q},\widetilde{I})$ of $(Q,I)$ has no oriented cycles and the universal cover $\widetilde{\Lambda}$ of $\Lambda$ is isomorphic to $k\widetilde{Q}/\widetilde{I}$. Moreover, if $kQ/J$ is another admissible ideal presentation of $\Lambda$ such that the quiver $\widetilde{R}$ of the universal cover $(\widetilde{R},\widetilde{J})$ of $(Q,J)$ has no oriented cycles, then according to \cite[Lemma 2.3 and Theorem 2.7]{MP1983}, the categories $k\widetilde{Q}/\widetilde{I}$ and $k\widetilde{R}/\widetilde{J}$ are isomorphic. Therefore for each admissible ideal presentation $kQ/J$ of $\Lambda$ such that the quiver $\widetilde{R}$ of the universal cover $(\widetilde{R},\widetilde{J})$ of $(Q,J)$ has no oriented cycles, the category $k\widetilde{R}/\widetilde{J}$ is isomorphic to the universal cover $\widetilde{\Lambda}$ of $\Lambda$.
\end{Rem1}

The standardness for a general locally bounded category is defined as follows.

\begin{Def}  {\rm(\cite[Section 1]{Sk1989})} \label{simply-connected-and-standard}
Let $\Lambda$ be a connected locally bounded category with Gabriel quiver $Q$.
\begin{itemize}
\item $\Lambda$ is called simply connected if $Q$ contains no oriented cycles and for any admissible ideal $I$ of $kQ$ with $\Lambda\cong kQ/I$, $\Pi(Q,I)=\{1\}$;
\item $\Lambda$ is called standard if it admits a simply connected Galois covering $\widetilde{\Lambda}\rightarrow\Lambda$, or equivalently if there exists some admissible ideal $I$ of $kQ$ with $\Lambda\cong kQ/I$ such that $(Q,I)$ has a simply connected universal cover $(\widetilde{Q},\widetilde{I})$.
\end{itemize}
\end{Def}

It is known that if $\Lambda$ is a locally representation-finite category, then the above two definitions coincide (cf. \cite{BG1982,Br-G1983}).

\begin{Ex1}
\begin{itemize}
\item[$(1)$] The locally bounded category given by the quiver with relations $(\widetilde{Q},\widetilde{I_2})$ in Example \ref{universal-cover-example2} is simply connected (by Corollary \ref{trivial-fundamental-group-the-converse} and Remark \ref{criterion-simply-connected} below), and therefore the locally bounded category $kQ/I_2$ given by the quiver with relations $(Q,I_2)$ in Example \ref{universal-cover-example2} is standard.
\item[$(2)$]  The locally bounded category $\Lambda=kQ/I_1$ given by the quiver with relations $(Q,I_1)$ in Example \ref{universal-cover-example2} is not simply connected since $Q$ contains a loop. However, when $\mathop{\rm char} k\neq 2$, since $\Lambda\cong kQ/I_2$, it is standard by $(1)$.
\item[$(3)$] When $\mathop{\rm char} k=2$, it is known that the locally bounded category $kQ/I_1$ is not standard by Riedtmann's famous work (see \cite{Riedtmann1983}).
\end{itemize}
\end{Ex1}

\begin{Rem1}  {\rm(see \cite[Lemma 4.21]{LL2})}  \label{criterion-simply-connected}
Let $\Lambda=kQ/I$ be a connected locally bounded category, where $Q$ is the Gabriel quiver $Q$ and $I$ is an admissible ideal of $kQ$, such that $\Pi(Q,I)=\{1\}$ and $Q$ contains no oriented cycles. If for each arrow $x\xrightarrow{\alpha}y$ of $Q$, $\alpha$ is the only path of $Q$ from $x$ to $y$, then $\Lambda$ is simply connected.
\end{Rem1}

\section{Push-down and pull-up}

The contents of this section are mainly based on \cite[Section I.10]{E1990} and \cite{Ga1981}. Throughout this section all the categories are assumed to be locally bounded $k$-categories.

	\subsection{Definitions of push-down and pull-up}\label{sec:push-down-pull-up}
\

Recall that for a connected, locally bounded $k$-category $\Lambda$, denote by $\mathrm{MOD} \Lambda$ the category of $\Lambda$-modules, that is, the category of contravariant functors $M\colon\Lambda\rightarrow \mathrm{MOD} k$, and $\mm\Lambda$ the category of {\it locally finite dimensional} $\Lambda$-modules, that is, the category of $\Lambda$-modules with $\mathrm{dim}_{k}M(x)<\infty$ for all $x\in \Lambda$, and $\mo \Lambda$ the category of finitely generated $\Lambda$-modules, that is, the category of  $\Lambda$-modules with $\sum_{x\in \Lambda}\mathrm{dim}_{k}M(x)<\infty$. For convenience, for each object or morphism $x$ in $\Lambda$, we denote $M(x)$ by $M_x$ when $M$ is a module of quiver with relations.

Let $\pi:( R,L)\rightarrow( Q,I)$ be a covering of quivers with relations which is defined in Definition \ref{morphism-and-covering-of-quivers-with-relations}. Denote by $F: kR/L\rightarrow kQ/I$ the covering functor induced by $\pi$ (see Corollary \ref{Cor}). This covering functor induces the corresponding {\it push-down} functor:
$$F_{\lambda}:\mathrm{MOD}( R,L)\rightarrow\mathrm{MOD}( Q,I)$$
To be more specific, let $M\in \mathrm{MOD}( R,L)$, then $V=F_{\lambda}(M)$ is defined as follows: For any vertex $a\in  Q_0$, set
$$V_a:=\bigoplus_{\pi x=a}M_x.$$
Suppose $a\xrightarrow{\alpha}b$ is an arrow in $ Q$. For any $x\in R_0$ with $\pi x=b$ there is a unique arrow $\tilde{\alpha}$ of $ R$ ending at $x$ with $\pi(\tilde{\alpha})=\alpha$ (and necessarily $\pi$ takes the start point of $\tilde{\alpha}$ to $a$). Then $V_\alpha:V_b\rightarrow V_a$ is the linear transformation which takes $m\in M_x$ to $M_{\tilde{\alpha}}(m)\in M_y$ for each $y$ with $y=s(\tilde{\alpha})$. For a morphism $f: M\rightarrow N$ in $\mathrm{MOD}( R,L)$, we define $F_\lambda(f): F_{\lambda}(M)\rightarrow F_{\lambda}(N)$ in $\mathrm{MOD}( Q,I)$ via $$\bigoplus_{\pi x=a}f_x: \bigoplus_{\pi x=a}M_x\rightarrow \bigoplus_{\pi x=a}N_x.$$ In particular, $F_{\lambda}$ induces a functor $$F_{\lambda}:\mo( R,L)\rightarrow\mo( Q,I).$$

The covering map can also induce the corresponding {\it pull-up} functor:
$$F_\bullet :\mathrm{MOD}( Q,I)\rightarrow\mathrm{MOD}(R,L)$$
To be more specific, let $V\in \mathrm{MOD}( Q,I)$, then $M=F_\bullet (V)$ is defined as follows: For any vertex $x\in  R_0$ and any arrow $\tilde{\alpha}\in  R_1$, set
$$M_x:=V_{\pi x},\; M_{\tilde{\alpha}}:=V_{\pi\tilde{\alpha}}.$$
For a morphism $\varphi: V\rightarrow W$ in $\mathrm{MOD}( Q,I)$, we define $F_\bullet(\varphi): F_\bullet(V)\rightarrow F_\bullet(W)$ in $\mathrm{MOD}( R,L)$ via $\varphi_{\pi x}: V_{\pi x}\rightarrow W_{\pi x}$. Clearly $F_{\bullet}$ induces a functor $$F_{\bullet}: \mm( Q,I)\rightarrow\mm( R,L).$$ In general, the pull-up functor cannot induce a functor from $\mo( Q,I)$ to $\mo( R,L)$ (see for example in Example \ref{exa:pull-up-push-down}).

It is obvious that $F_{\lambda}$ and $F_\bullet $ are exact, and $(F_{\lambda},F_\bullet )$ is an adjoint pair, that is, there is a natural isomorphism (for each $(R,L)$-module $M$ and each $(Q,I)$-module $V$)
$$\ho_{(Q,I)}(F_\lambda M, V)\cong \ho_{(R,L)}(M, F_\bullet V).$$

\begin{Ex1}\textnormal{({Example \ref{Kron-example} revisited})}\label{exa:pull-up-push-down}
Let $\pi:(\widetilde{Q},\widetilde{I})\rightarrow( Q,I)$ be a covering of quivers with relations. Consider $(Q,I)$-module $V$, whose representation is given by
$$\begin{tikzpicture}
\draw[<-] (0.3,0.1) -- (1.7,0.1);
\draw[<-] (0.3,-0.1) -- (1.7,-0.1);
\node at(0,0) {$k$};
\node at(2,0) {$k$};
\node at(1,0.3) {$1$};
\node at(1,-0.4) {$0$};
\end{tikzpicture}$$
and the $(\widetilde{Q},\widetilde{I})$-module $N$, whose representation is given by
$$\begin{tikzpicture}
\draw[<-] (-3.2,0.8) -- (-3.8,0.2);
\draw[<-] (-2.8,0.8) -- (-2.2,0.2);
\draw[<-] (-1.2,0.8) -- (-1.8,0.2);
\draw[<-] (-0.8,0.8) -- (-0.2,0.2);
\draw[<-] (0.8,0.8) -- (0.2,0.2);
\draw[<-] (1.2,0.8) -- (1.8,0.2);
\draw[<-] (2.8,0.8) -- (2.2,0.2);
\draw[<-] (3.2,0.8) -- (3.8,0.2);
\node at(-4,0) {$k$};
\node at(-3,1) {$0$};
\node at(-2,0) {$k$};
\node at(-1,1) {$0$};
\node at(0,0) {$k$};
\node at(2,0) {$k$};
\node at(1,1) {$0$};
\node at(3,1) {$0$};
\node at(4,0) {$k$};
\node at(-5,0.5) {$\cdots$};
\node at(5,0.5) {$\cdots$};
\node at(-3.65,0.65) {$0$};
\node at(-2.35,0.7) {$0$};
\node at(-1.65,0.65) {$0$};
\node at(-0.35,0.7) {$0$};
\node at(0.35,0.65) {$0$};
\node at(1.65,0.7) {$0$};
\node at(2.35,0.65) {$0$};
\node at(3.65,0.7) {$0$};
\end{tikzpicture}$$
Then $F_\bullet V$ is a $(\widetilde{Q},\widetilde{I})$-module, whose representation is given by
$$\begin{tikzpicture}
\draw[<-] (-3.2,0.8) -- (-3.8,0.2);
\draw[<-] (-2.8,0.8) -- (-2.2,0.2);
\draw[<-] (-1.2,0.8) -- (-1.8,0.2);
\draw[<-] (-0.8,0.8) -- (-0.2,0.2);
\draw[<-] (0.8,0.8) -- (0.2,0.2);
\draw[<-] (1.2,0.8) -- (1.8,0.2);
\draw[<-] (2.8,0.8) -- (2.2,0.2);
\draw[<-] (3.2,0.8) -- (3.8,0.2);
\node at(-4,0) {$k$};
\node at(-3,1) {$k$};
\node at(-2,0) {$k$};
\node at(-1,1) {$k$};
\node at(0,0) {$k$};
\node at(2,0) {$k$};
\node at(1,1) {$k$};
\node at(3,1) {$k$};
\node at(4,0) {$k$};
\node at(-5,0.5) {$\cdots$};
\node at(5,0.5) {$\cdots$};
\node at(-3.65,0.65) {$1$};
\node at(-2.35,0.7) {$0$};
\node at(-1.65,0.65) {$1$};
\node at(-0.35,0.7) {$0$};
\node at(0.35,0.65) {$1$};
\node at(1.65,0.7) {$0$};
\node at(2.35,0.65) {$1$};
\node at(3.65,0.7) {$0$};
\end{tikzpicture}$$
and $F_\lambda N=\oplus_{\Pi(Q,I)}S_x$, where $S_x$ is the simple module at $x$ in $Q$.
\end{Ex1}

Next, we prove that these two functors are faithful.

\begin{Lem}
$F_\lambda$ and $F_\bullet$ are faithful.
\end{Lem}

\begin{proof}
	We only verify the case of the push-down functor $F_\lambda$. Consider two morphisms $f,g\in\mathrm{Hom}_{\mathrm{MOD}(R,L)}(M,N)$ such that $F_\lambda(f)=F_\lambda(g)$. For any $x\in R_0$, we have $$F_\lambda(f)_{\pi x}=F_\lambda(g)_{\pi x}.$$ Since $F_\lambda(M)_{\pi x}=\bigoplus_{\pi y=\pi x}M_y$ and $F_\lambda(N)_{\pi x}=\bigoplus_{\pi y=\pi x}N_y$, this equality implies that the restrictions of $f$ and $g$ to each component $M_y$ must coincide. In particular, taking $y=x$, we obtain $f_x=g_x$. Because this holds for every $x\in R_0$, it follows that $f=g$. Hence, $F_\lambda$ is faithful.
\end{proof}

We use the following example to demonstrate that these functors are not full, even when $(Q,I)$ corresponds to a representation-finite algebra.

\begin{Ex1}
	Let $Q$ be the quiver
	$$\begin{tikzcd}
\bullet \arrow["\alpha"', loop, distance=2em, in=215, out=145]
\end{tikzcd}$$
with the ideal $I=\langle\alpha^2\rangle$. The universal cover $(R,L)$ of $(Q,I)$ is given by the infinite linear quiver
$$
\begin{tikzcd}
\cdots \arrow[r, "\widetilde{\alpha}"] & \bullet \arrow[r, "\widetilde{\alpha}"] & \bullet \arrow[r, "\widetilde{\alpha}"] & \bullet \arrow[r, "\widetilde{\alpha}"] & \bullet \arrow[r, "\widetilde{\alpha}"] & \bullet \arrow[r, "\widetilde{\alpha}"] & \cdots
\end{tikzcd}$$
with relations generated by $\widetilde{\alpha}^2=0$.

Consider $(R,L)$-modules $M,N$, whose represntations are given by
$$
\begin{tikzcd}
\cdots & 0 \arrow[l, "0"'] & 0 \arrow[l, "0"'] & k \arrow[l, "0"'] & 0 \arrow[l, "0"'] & 0 \arrow[l, "0"'] & \cdots \arrow[l, "0"']
\end{tikzcd}$$
and $$
\begin{tikzcd}
\cdots & 0 \arrow[l, "0"'] & 0 \arrow[l, "0"'] & 0 \arrow[l, "0"'] & k \arrow[l, "0"'] & 0 \arrow[l, "0"'] & \cdots \arrow[l, "0"']
\end{tikzcd}$$ respectively. Then $\mathrm{Hom}_{\mathrm{MOD}(R,L)}(M,N)=0$ and $\mathrm{Hom}_{\mathrm{MOD}(Q,I)}(F_\lambda M,F_\lambda N)\cong k$. Therefore, $F_\lambda$ is not full.

Consider $(Q,I)$-modules $V$, whose represntation is given by
	$$\begin{tikzcd}
k \arrow["0"', loop, distance=2em, in=215, out=145]
\end{tikzcd}.$$ Then $\mathrm{Hom}_{\mathrm{MOD}(Q,I)}(V,V)\cong k$ and $\mathrm{Hom}_{\mathrm{MOD}(R,L)}(F_\bullet V,F_\bullet V)\cong \prod_{\mathbb{Z}}k$. Therefore, $F_\bullet$ is not full.
\end{Ex1}

\subsection{Properties of push-down and pull-up}
\

Let $\pi:( R,L)\rightarrow( Q,I)$ be a covering of quivers with relations. Assume now that $\pi$ is a Galois covering with group $G$, which means that there exists an isomorphism $\nu:(R/G, \overline{L})\xrightarrow{\sim} (Q,I)$.

\begin{Def}
	Let $M\in\mo( R,L)$ and $g\in G$. Then the translate $^gM$ is the module of $(R,L)$ with
	$$(^gM)_x:=M_{g^{-1}x}\;(x\in R_0);\;\;(^gM)_\alpha:=M_{g^{-1}\alpha}\;(\alpha\in R_1).$$
\end{Def}

\begin{Lem}{\rm(\cite[Lemma 3.2]{Ga1981} and \cite[Lemma I.10.5(c)]{E1990})}\label{pushdown-pull-up}
	For each $M\in\mo( R,L)$ and each $g\in G$, we have $F_{\lambda} \;^gM\cong F_{\lambda}M$ and $\oplus_{h\in G} \;^hM\cong F_{\bullet}F_\lambda M$ canonically.
\end{Lem}

\begin{proof}
	For each $a\in ( Q,I)$, we have $F_\lambda^gM_a=\oplus_{x\in a}M_{g^{-1}x}$ and $F_{\lambda}M_a=\oplus_{x\in a}M_{x}$. The canonical isomorphism $F_{\lambda} \;^gM\cong F_{\lambda}M$ maps the summand $M_{g^{-1}x}$ of $F_\lambda^gM_a$ with index $x$ identically onto the summand $M_{g^{-1}x}$ of $F_\lambda M_a$ with index $g^{-1}x$.
	
	Similarly, we have $F_{\bullet}F_\lambda M_x=F_{\lambda}M_{\pi x}=\oplus_{\pi y= \pi x}M_y=\oplus_{h\in G}M_{h^{-1}x}=\oplus_{h\in G}\;^hM_x$ for each object $x\in ( R,L)$.
\end{proof}

\begin{Lem}{\rm(\cite[Lemma I.10.5(d)]{E1990})}\label{pushdown-indec}
	Suppose that $M\in\mathrm{ind}(R,L)$ and that $^gM\not\cong M$ if $1\neq g\in G$. Then $F_{\lambda}M$ is indecomposable. Moreover, each $W\in\mo( R,L)$ such that $F_{\lambda}W\cong F_{\lambda}M$ is isomorphic to $^gM$ for some $g\in G$.
\end{Lem}

\begin{proof}
	Assume that $F_{\lambda}M=N\oplus N'$ where $N\neq0$. Then we have that $\oplus \;^gM\cong F_\bullet F_{\lambda} M=F_\bullet N\oplus F_\bullet N'$. Since the $^gM$ are pairwise non-isomorphic, we must have that $F_\bullet N\cong\oplus\;^hM$ and $F_\bullet N'\cong\oplus\;^fM$ for some subset $H$ and $F:=G\backslash H$ of $G$. On the other hand, we have $F_\bullet N={^gF_\bullet N}\cong\oplus\;^{gh}M$ for each $g\in G$. This implies $H=gH$ for each $g\in G$, hence $H=G$ and $N'=0$.
	
	Assume that $F_{\lambda}W=F_{\lambda}M$. Then $W$ is isomorphic to a direct summand of
	$$F_\bullet F_{\lambda}W\cong F_\bullet F_{\lambda}M\cong\oplus\;^gM.$$
	Since $F_{\lambda}W=F_{\lambda}M$ is indecomposable, we must have $W\cong\;^gM$ for some $g\in G$.
\end{proof}

We shall now study the restrictions of push-down functors to subcategories. Recall that a module $M$ of a quiver $T$ is said to be {\it sincere} if $M_x\neq 0$ for each vertex $x$ of $T$.

\begin{Prop}{\rm(cf. \cite[Proposition~I.10.6]{E1990})} \label{Sin}
Let $\pi\colon (R,L)\to (Q,I)$ be a Galois covering of quivers with relations with group $G$.
Choose a finite subquiver $T\subset R$ satisfying the following conditions:
\begin{enumerate}
\item For every $1\neq g\in G$, we have ${}^gT\neq T$.
  \item For each relation $\rho\in I$ and for every path $p$ appearing in $\rho$, if there exists a lift $\tilde{p}$ of $p$ in $T$ (that is, $\pi(\tilde{p})=p$), then there also exists a lift $\tilde{\rho}$ of $\rho$ in $T$ such that $\pi(\tilde{\rho})=\rho$.
\end{enumerate}
Set $L' := kT\cap L$.
Let $\mathrm{mod}(T,L')_s$ denote the full subcategory of $\mathrm{mod}(T,L')$ consisting of sincere modules. Then
	\begin{enumerate}[label=\textnormal{(\roman*)}]
		\item The restriction $F_{\lambda}\colon\mo(T,L')_s\rightarrow \mo(Q,I)$ preserves indecomposability and reflects isomorphism.
		
		\item There is a finitely generated $kT/L'$-$k Q/I$-bimodule $B$ which is free as a left $(T,L')$-module such that on $\mo(T,L')_s$, $F_{\lambda}\cong-\otimes_{kT/L'}B$.
	\end{enumerate}
\end{Prop}

\begin{proof}
   (i)\; We first show $F_\lambda$ preserves indecomposability. Suppose $M$ is an indecomposable $(T,L')_s$-module. Then we consider $M$ as a $kR$-module (with $M_x=0$ for $x$ not in $T$). By Lemma \ref{pushdown-indec}, it is suffices to show that $^gM\neq M$ for $1\neq g$. If $1\neq g$, then $^gT\neq T$ and hence the supports of $^gM$ and $M$ are distinct. Thus, $^gM\neq M$.

        Then we show $F_\lambda$ reflects isomorphisms. Let $F_\lambda(W)\cong F_\lambda(M)$ where $W$ and $M$ are in $\mo(T,L')_s$. Then consider $W$ and $M$ as $kR$-module as above. By Lemma \ref{pushdown-indec}, we have $W\cong ^gM$ for some $g\in G$. Hence
        $$T=\mathrm{supp}M=\mathrm{supp}W=\mathrm{supp}^gM,$$
        and $^gT=T$. Therefore, $g=1$ and $W\cong M$.

   (ii)\; Define $B$ to be the free $(T,L')$-module which is given by $B=\oplus_{x\in T_0}(kT/L')b_x$. We define a right $(Q,I)$-action on $B$ as follows: let $e\in Q_0$, $w\in kT/L'$ and $x\in T_0$. We denote by $f_x$ the idempotent of $kT/L'$ corresponding to $x$; and we set
   		$$
		(wb_x)e:=\left\{
		\begin{array}{*{3}{ll}}
			0,& \text{if $\pi x\neq e$};\\
			wf_xb_x,& \text{otherwise.}
		\end{array}
		\right.
		$$
    Suppose $\alpha:e\rightarrow f$ is an arrow in $Q$. If $x\in T_0$ with $\pi x=e$ and $\widetilde{\alpha}$ is an arrow in $T$ starting at $x$, then define $(wb_x)\alpha=(w\widetilde{\alpha})b_y$ where $y$ is the endpoint of $\widetilde{\alpha}$, and set $(wb_x)\alpha=0$ otherwise. Then through verifying the definitions, this action is a $(Q,I)$-action and $F_\lambda(M)\cong M\otimes_{kT} B$, canonically.
\end{proof}

\begin{Rem1}
The above proposition is adapted from \cite[Proposition~I.10.6]{E1990}, but we have modified two of its assumptions, corresponding to our conditions~(1) and~(2):

\begin{itemize}
  \item[(1)] Condition~(1) strengthens the requirement on the subquiver $T$ and is necessary for this proposition (communicated to us by Fei Zeng). Indeed, this condition used in the original proof without mentioning the reason. Example~\ref{exa:Zeng-counterexample} below gives a counterexample if we drop this condition.

  \item[(2)] Condition~(2) weakens the original assumption that ``$\pi$ maps $L'$ onto $I$''. In original proof, the assumption ``$\pi$ maps $L'$ onto $I$'' was used to ensure that the right $kQ/I$-action on the module $B$ constructed in conclusion~(ii) is well-defined. Our condition~(2) is sufficient for this purpose. In fact, in \cite[Propositions~I.10.7--I.10.12]{E1990}, one really uses the weaker condition (2).
\end{itemize}
\end{Rem1}

\begin{Ex1}\label{exa:Zeng-counterexample}
Let $k$ be a field whose characteristic is not equal to $2$. Let $Q$ be the quiver
$$\begin{tikzpicture}
\draw[->] (0.3,0.1) -- (1.7,0.1);
\draw[->] (0.3,-0.1) -- (1.7,-0.1);
\node at(0,0) {$x$};
\node at(2,0) {$y$};
\node at(1,0.3) {$a$};
\node at(1,-0.4) {$b$};
\end{tikzpicture}$$
and let $R$ be the quiver
$$
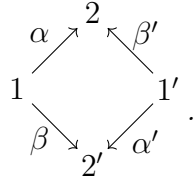
\begin{tikzpicture}
\draw[->] (0.2,0.2) -- (0.8,0.8);
\draw[->] (0.2,-0.2) -- (0.8,-0.8);
\draw[->] (1.8,0.2) -- (1.2,0.8);
\draw[->] (1.8,-0.2) -- (1.2,-0.8);
\node at(0,0) {$1$};
\node at(2,0) {$1'$};
\node at(1,1) {$2$};
\node at(1,-1) {$2'$};
\node at(0.3,0.7) {$\alpha$};
\node at(1.7,-0.7) {$\alpha'$};
\node at(0.3,-0.7) {$\beta$};
\node at(1.7,0.7) {$\beta'$};
\node at(2.3,-0.4) {.};
\end{tikzpicture}$$
Let $I$ and $L$ be the zero ideal in $kQ$ and $kR$ respectively. Then $kQ/I$ and $kR/L$ are locally bounded categories and there exists a Galois covering of quivers with relations
$$\pi:(R,L)\rightarrow (Q,I)$$
which is given by $\pi(1)=\pi(1')=x$, $\pi(2)=\pi(2')=y$, $\pi(\alpha)=\pi(\alpha')=a$, $\pi(\beta)=\pi(\beta')=b$. Note that $\pi$ is a Galois covering of quivers with relations with group $G$, where $G=\langle g\rangle$ is a cyclic group of order $2$. Denote by $F: kR/L\rightarrow kQ/I$ the covering functor induced by $\pi$.

Let $T=R$ be a subquiver of $R$. Then ${}^gT=T$. However, the push-down functor $F_\lambda \colon \mathrm{mod}(T, L')_s \longrightarrow \mathrm{mod}(Q, I)$ does not preserve indecomposability. Indeed, let $M$ be an indecomposable sincere module represented by
$$
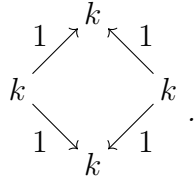
\begin{tikzpicture}
\draw[->] (0.2,0.2) -- (0.8,0.8);
\draw[->] (0.2,-0.2) -- (0.8,-0.8);
\draw[->] (1.8,0.2) -- (1.2,0.8);
\draw[->] (1.8,-0.2) -- (1.2,-0.8);
\node at(0,0) {$k$};
\node at(2,0) {$k$};
\node at(1,1) {$k$};
\node at(1,-1) {$k$};
\node at(0.3,0.7) {$1$};
\node at(1.7,-0.7) {$1$};
\node at(0.3,-0.7) {$1$};
\node at(1.7,0.7) {$1$};
\node at(2.3,-0.4) {.};
\end{tikzpicture}$$
Then $F_\lambda M$ is decomposable, which is isomorphic to the direct sum of two modules $N_1$ and $N_2$ represented by
$$\begin{tikzpicture}
\draw[->] (0.3,0.1) -- (1.7,0.1);
\draw[->] (0.3,-0.1) -- (1.7,-0.1);
\node at(0,0) {$k$};
\node at(2,0) {$k$};
\node at(1,0.3) {$1$};
\node at(1,-0.4) {$1$};
\end{tikzpicture}$$
and
$$\begin{tikzpicture}
\draw[->] (0.3,0.1) -- (1.7,0.1);
\draw[->] (0.3,-0.1) -- (1.7,-0.1);
\node at(0,0) {$k$};
\node at(2,0) {$k$};
\node at(1,0.3) {$1$};
\node at(1,-0.4) {$-1$};
\end{tikzpicture}$$
respectively.

Moreover, $F_\lambda$ does not reflect isomorphism. Let $M_1$ and $M_2$ be indecomposable sincere modules given respectively by
$$
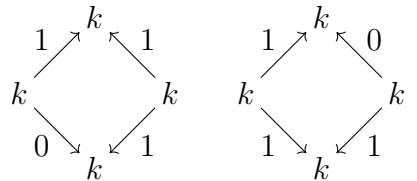
\begin{tikzpicture}
\draw[->] (0.2,0.2) -- (0.8,0.8);
\draw[->] (0.2,-0.2) -- (0.8,-0.8);
\draw[->] (1.8,0.2) -- (1.2,0.8);
\draw[->] (1.8,-0.2) -- (1.2,-0.8);
\node at(0,0) {$k$};
\node at(2,0) {$k$};
\node at(1,1) {$k$};
\node at(1,-1) {$k$};
\node at(0.3,0.7) {$1$};
\node at(1.7,-0.7) {$1$};
\node at(0.3,-0.7) {$0$};
\node at(1.7,0.7) {$1$};
\draw[->] (3.2,0.2) -- (3.8,0.8);
\draw[->] (3.2,-0.2) -- (3.8,-0.8);
\draw[->] (4.8,0.2) -- (4.2,0.8);
\draw[->] (4.8,-0.2) -- (4.2,-0.8);
\node at(3,0) {$k$};
\node at(5,0) {$k$};
\node at(4,1) {$k$};
\node at(4,-1) {$k$};
\node at(3.3,0.7) {$1$};
\node at(4.7,-0.7) {$1$};
\node at(3.3,-0.7) {$1$};
\node at(4.7,0.7) {$0$};
\node at(5.3,-0.4) {.};
\end{tikzpicture}$$
Then $M_1 \not\cong M_2$, but $F_\lambda M_1 \cong F_\lambda M_2$.

\end{Ex1}

\subsection{Applications for determining representation types}
\

\subsubsection{Preliminaries on representation types of finite dimensional algebras}
\

\begin{Def}{\rm(\cite[Section I.4]{E1990})}
	Let $\Lambda$ be a finite dimensional $k$-algebra. Then $\Lambda$ is said to be of finite (representation) type provided there are finitely many indecomposable $\Lambda$-modules. Otherwise, $\Lambda$ is of infinite type.
\end{Def}

Compared with Definition \ref{locally-representation-finite-category}, for each quiver with relations $(Q,I)$ with $kQ/I$ finite dimensional, if $(Q,I)$ is locally representation-finite, then $kQ/I$ is of finite type.

\begin{Def}{\rm(\cite[Section I.2]{E1990})}
	Let $\Lambda$ be a finite dimensional $k$-algebra. Then $\Lambda$ is tame provided $\Lambda$ is not of finite type, whereas for any dimension $d>0$, there are a finite number of $k[T]$-$\Lambda$-bimodules $M_i$ which are free as left $k[T]$-modules such that all but a finite number of indecomposable $\Lambda$-modules of dimension $d$ are isomorphic to $N\otimes_{k[T]}M_i$ for some $i$ and some simple $k[T]$-module $N$.
\end{Def}

\begin{Def}{\rm(\cite[DefinitionI.4.4]{E1990})}
	A finite dimensional $k$-algebra $\Lambda$ is wild if there is a finitely generated $k\langle X,Y\rangle$-$\Lambda$-bimodule $B$ which is free as a left $k\langle X,Y\rangle$-module such that the functor $-\otimes_{k\langle X,Y\rangle}B$ from $\mo\text{-}k\langle X,Y\rangle$ to $\mo$-$\Lambda$ preserve indecomposability and reflects isomorphisms.
	
	In particular, a subcategory $\mathcal{C}$ of $\mo\Lambda$ is called wild if there is a finitely generated $k\langle X,Y\rangle$-$\Lambda$-bimodule $B$ fits the conditions above and $\mathrm{Im}(-\otimes_{k\langle X,Y\rangle}B)\subseteq \mathcal{C}$.
\end{Def}

\begin{Thm}{\rm(\cite{D})}
	Suppose $\Lambda$ is a finite dimensional algebra of infinite type. Then $\Lambda$ is either tame or wild.
\end{Thm}

Note that more generally, the notions of tameness and wildness also make sense for locally bounded categories (and these will be used in next section), and it is known that a locally bounded category is tame if and only if any its finite full subcategory is tame (see \cite{DS1986}).

We now recall some classical results which can help us to recognize whether a finite dimensional algebra is wild or not.

\begin{Thm}{\rm(cf. \cite{ARS})}
	Let $Q$ be a finite quiver. A path algebra $kQ$ is of finite type if and only if the underlying graph $\bar{Q}$ is a Dynkin diagram. $kQ$ is tame if and only if the underlying graph $\bar{Q}$ is an Euclidean diagram.
\end{Thm}

\begin{Prop}{\rm(see \cite[Proposition I.4.7]{E1990})}\label{eAe}
	Suppose $\Lambda$ is a finite dimensional algebra. Then
	\begin{enumerate}[label = (\alph*)]
		\item If $I$ is some ideal of $\Lambda$ such that $\Lambda/I$ is wild, then so is $\Lambda$.
		
		\item If $e$ is an idempotent of $\Lambda$ such that $e\Lambda e$ is wild, then so is $\Lambda$.
		
		\item $\Lambda$ is wild if and only if its basic algebra is wild.
		
		\item If $\Lambda$ is wild, then so is $\Lambda^{op}$.			
	\end{enumerate}
	The same statements hold with wild replaced by ``of infinite type''. Moreover, (c) and (d) are true with tame instead of wild, and
	\begin{enumerate}[label=(b')]
		\item If $\Lambda$ is tame and $e$ is an idempotent of $\Lambda$, then $e\Lambda e$ is tame or of finite type.
	\end{enumerate}
\end{Prop}

\subsubsection{Determining finite type}
\

Recall that a locally bounded category $\Lambda$ is called locally representation-finite, if for every object $x$ of $\Lambda$, the number of isomorphism classes of finitely generated indecomposable $\Lambda$-module $\ell$ such that $\ell(x)\neq 0$ is finite (see Definition \ref{locally-representation-finite-category}).

\begin{Prop}{\rm(cf. \cite[Lemma 3.3]{Ga1981})}\label{Galois-cover-pre-representation-finite}
	Let $\pi:( R,L)\rightarrow( Q,I)$ be a Galois covering of quivers with relations with group $G$. If $kQ/I$ is a locally representation-finite category, then so is $kR/L$.
\end{Prop}

\begin{proof}
	We call a module $N\in \mo( Q,I)$ {\it orbicular} if $N$ is indecomposable and $F_\bullet N$ is a direct sum of an indecomposable finitely generated $( R,L)$-module $\ell$ and of some translates $^g\ell$ with $g\in G$. Given an object $x\in ( R,L)$, let $N^1,\cdots,N^r$ be representatives of the orbicular modules $N$ such that $N_{\pi x}\neq 0$. Choose an indecomposable direct summand $M^i$ of $F_\bullet N^i$ and denote by $G_i$ the set formed by the $g\in G$ such that $M^i_{g^{-1}x}\neq 0$. Note that since the finitely generated module $M^i$ has finite dimension and $G$ acts freely on $R$, each $G_i$ is a finite set.
	
	We claim that each indecomposable finitely generated $M\in \mo( R,L)$ such that $M_x\neq 0$ is isomorphic to some $^gM^i$ with $g\in G_i$. Indeed, consider an indecomposable decomposition $F_\lambda M=P^1\oplus\cdots\oplus P^s$. Then by Lemma \ref{pushdown-pull-up}, we have
	$$\oplus_{g\in G}\;^gM\cong F_{\bullet}F_\lambda M=F_\bullet P^1\oplus\cdots\oplus F_\bullet P^s,$$
	By the Krull-Schmidt theorem in \cite[Theorem 1]{War1969}, each $F_\bullet P^j$ is isomorphic to the direct sum of some modules $^gM$, which means $P^j$ is orbicular. In particular, there exists a $P^j$ such that $F_\bullet P^j=\oplus_{g\in H}\;^gM$ with $1\in H\subseteq G$, which implies that $P^j_{\pi x}\neq 0$. Therefore, $P^j$ is isomorphic to some $N^i$. Now suppose that $F_\bullet N^i=\oplus_{g\in H_i}\;^gM^i$ for some subset $H_i$ of $G$. Moreover,
	$$\oplus_{g\in H_i}\;^gM^i=F_\bullet N^i\cong F_\bullet P^j=\oplus_{g\in H}\;^gM,$$ and we get again by Krull-Schmidt theorem $M\cong\;^gM^i$ for some $j$ and some $g\in H_i$. Since $M_x\neq 0$ implies $M^i_{g^{-1}x}=(^gM^i)_x\neq 0$, it follows that $g\in G_i$.
\end{proof}

\begin{Rem1}
\begin{itemize}
\item[(1)] \cite[Theorem 5]{MP2} demonstrates that the converse of Proposition \ref{Galois-cover-pre-representation-finite} also holds.
\item[(2)] If $\pi:(R,L)\rightarrow(Q,I)$ is a Galois covering of quivers with relations with group $G$ with $kQ/I$ locally representation-finite, then all indecomposable $(Q,I)$-modules are orbicular. The reason is as follows: Since $kQ/I$ is locally representation-finite, by Proposition \ref{Galois-cover-pre-representation-finite} so is $kR/L$. Since $G$ acts freely on $R$, according to \cite[Theorem 4]{MP2}, it acts freely on $[\mathrm{ind}(R,L)]$. By \cite[Theorem 3.6(d)]{Ga1981}, each $N\in\mathrm{ind}(Q,I)$ is isomorphic to some $F_{\lambda}M$ for some $M\in\mathrm{ind}(R,L)$. Then $F_{\bullet}N\cong F_{\bullet}F_{\lambda}M\cong\oplus_{h\in G} \;^hM$. For a construction of non-orbicular modules, see \cite{D20012}.
\end{itemize}
\end{Rem1}

Proposition \ref{Galois-cover-pre-representation-finite} shows that the representation type of a finite dimensional algebra can be determined through its (universal) cover. We provide an illustrative example through the following proposition.

\begin{Prop}{\rm(\cite[Proposition I.10.11]{E1990})}
	Suppose $\Lambda$ is an algebra with quiver $Q$ such that for each vertex $e$ of $Q$ at least two arrows start and at least two arrows end at $e$. Then $\Lambda$ is not of finite type.
\end{Prop}

\begin{proof}
	We may assume that $\Lambda$ is basic. Consider the algebra $\Lambda/\mathrm{rad}(\Lambda)^2$. It has the same quiver as $\Lambda$ and may be defined by zero relations only. It has a universal cover which contains infinite lines of the form
	$$
	\begin{tikzcd}
		\cdots & \bullet \arrow[l] \arrow[r] & \bullet & \bullet \arrow[l] \arrow[r] & \bullet & \bullet \arrow[l] \arrow[r] & \cdots
	\end{tikzcd}$$
	Consequently, the universal cover of $\Lambda/\mathrm{rad}(\Lambda)^2$ is not locally representation-finite. By Proposition \ref{Galois-cover-pre-representation-finite}, this implies that $\Lambda/\mathrm{rad}(\Lambda)^2$ is also not locally representation-finite. Applying Proposition \ref{eAe}, we conclude that $\Lambda$ is not of finite type.
\end{proof}

\subsubsection{Determining wild type}
\

We demonstrate that covering theory, particularly through universal covers of finite dimensional algebras, can also be applied to determine wild type.

\begin{Thm}{\rm(\cite{D})}\label{Sincere-wildness}
	A finite dimensional algebra is wild if and only if the category $(\mo \;\Lambda)_s$ is wild.
\end{Thm}

Therefore, by Proposition \ref{Sin} and Theorem \ref{Sincere-wildness}, we can show that a finite dimensional algebra $\Lambda=k Q/I$ is of wild type by choosing a finite subquiver $T$ of its Galois covering $ R$ which fits the conditions in Proposition \ref{Sin} such that $kT/L'$ is wild. Finally, we will show some applications of this method.

\begin{Prop}{\rm(\cite[Proposition I.10.8]{E1990})}
	The algebra $\Lambda$ is wild whose quiver contains one of the following quivers or their duals:
	\begin{enumerate}[label=\textnormal{(\roman*)}]
		\item $
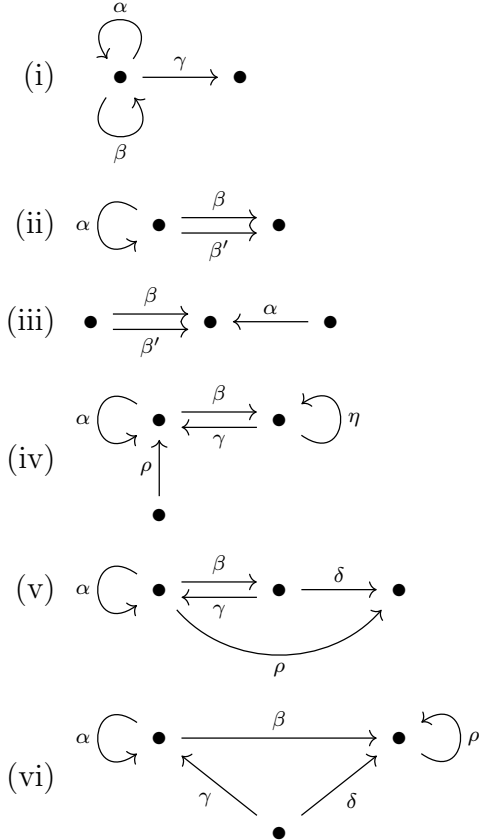
		\begin{tikzcd}
			\bullet \arrow["\alpha"', loop, distance=2em, in=125, out=55] \arrow["\beta"', loop, distance=2em, in=305, out=235] \arrow[r, "\gamma"] & \bullet
		\end{tikzcd}$
		\item $
		\begin{tikzcd}
			\bullet \arrow["\alpha"', loop, distance=2em, in=215, out=145] \arrow[r, "\beta", shift left] \arrow[r, "\beta'"', shift right] & \bullet
		\end{tikzcd}$
		\item $
		\begin{tikzcd}
			\bullet \arrow[r, "\beta", shift left] \arrow[r, "\beta'"', shift right] & \bullet & \bullet \arrow[l, "\alpha"']
		\end{tikzcd}$
		\item $
		\begin{tikzcd}
			\bullet \arrow["\alpha"', loop, distance=2em, in=215, out=145] \arrow[r, "\beta", shift left] & \bullet \arrow["\eta"', loop, distance=2em, in=35, out=325] \arrow[l, "\gamma", shift left] \\
			\bullet \arrow[u, "\rho"]                                                                     &
		\end{tikzcd}$
		\item $
		\begin{tikzcd}
			\bullet \arrow[r, "\beta", shift left] \arrow["\alpha"', loop, distance=2em, in=215, out=145] \arrow[rr, "\rho"', bend right=49] & \bullet \arrow[l, "\gamma", shift left] \arrow[r, "\delta"] & \bullet
		\end{tikzcd}$
		\item $
		\begin{tikzcd}
			\bullet \arrow[rr, "\beta"] \arrow["\alpha"', loop, distance=2em, in=215, out=145] &                                                    & \bullet \arrow["\rho"', loop, distance=2em, in=35, out=325] \\
			& \bullet \arrow[lu, "\gamma"] \arrow[ru, "\delta"'] &
		\end{tikzcd}$
	\end{enumerate}
	
\end{Prop}

\begin{proof}
	The algebra $\Lambda/\mathrm{rad}(\Lambda)^2$ is generated by zero relations. So it has a universal cover as described above. In each case we find a tree $\cong \tilde{\tilde{E_7}}$ in the universal cover of $\Lambda/\mathrm{rad}(\Lambda)^2$, namely
	\begin{enumerate}[label=(\roman*)]
		\item $
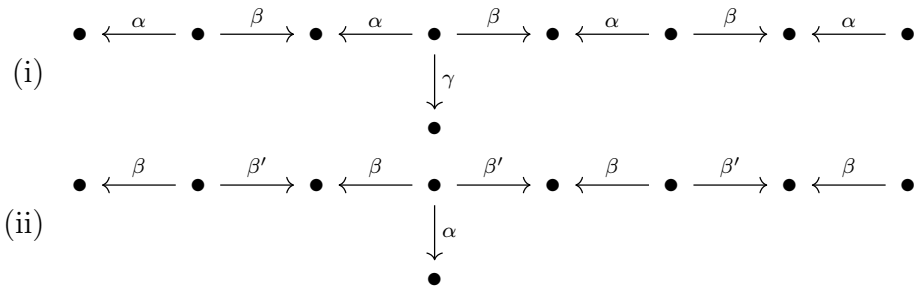
		\begin{tikzcd}
			\bullet & \bullet \arrow[l, "\alpha"'] \arrow[r, "\beta"] & \bullet & \bullet \arrow[l, "\alpha"'] \arrow[d, "\gamma"] \arrow[r, "\beta"] & \bullet & \bullet \arrow[l, "\alpha"'] \arrow[r, "\beta"] & \bullet & \bullet \arrow[l, "\alpha"'] \\
			&                                                 &         & \bullet                                                             &         &                                                 &         &
		\end{tikzcd}	$		
		\item$
		\begin{tikzcd}
			\bullet & \bullet \arrow[l, "\beta"'] \arrow[r, "\beta'"] & \bullet & \bullet \arrow[l, "\beta"'] \arrow[d, "\alpha"] \arrow[r, "\beta'"] & \bullet & \bullet \arrow[l, "\beta"'] \arrow[r, "\beta'"] & \bullet & \bullet \arrow[l, "\beta"'] \\
			&                                                 &         & \bullet                                                             &         &                                                 &         &
		\end{tikzcd}$
		\item$
		\begin{tikzcd}
			\bullet \arrow[r, "\beta"] & \bullet & \bullet \arrow[l, "\beta'"'] \arrow[r, "\beta"] & \bullet                      & \bullet \arrow[l, "\beta'"'] \arrow[r, "\beta"] & \bullet & \bullet \arrow[l, "\beta'"'] \arrow[r, "\beta"] & \bullet \\
			&         &                                                 & \bullet \arrow[u, "\alpha"'] &                                                 &         &                                                 &
		\end{tikzcd}$
		\item$
		\begin{tikzcd}
			\bullet \arrow[r, "\eta"] & \bullet & \bullet \arrow[l, "\beta"'] \arrow[r, "\alpha"] & \bullet                    & \bullet \arrow[l, "\gamma"'] \arrow[r, "\eta"] & \bullet & \bullet \arrow[l, "\beta"'] \arrow[r, "\alpha"] & \bullet \\
			&         &                                                 & \bullet \arrow[u, "\rho"'] &                                                &         &                                                 &
		\end{tikzcd}$
		\item$
		\begin{tikzcd}
			\bullet & \bullet \arrow[l, "\gamma"'] \arrow[r, "\delta"] & \bullet & \bullet \arrow[l, "\rho"'] \arrow[d, "\beta"] \arrow[r, "\alpha"] & \bullet & \bullet \arrow[l, "\gamma"'] \arrow[r, "\delta"] & \bullet & \bullet \arrow[l, "\rho"'] \\
			&                                                  &         & \bullet                                                           &         &                                                  &         &
		\end{tikzcd}$			
		\item$
		\begin{tikzcd}
			\bullet \arrow[r, "\gamma"] & \bullet & \bullet \arrow[l, "\alpha"'] \arrow[r, "\beta"] & \bullet                    & \bullet \arrow[l, "\delta"'] \arrow[r, "\gamma"] & \bullet & \bullet \arrow[l, "\alpha"'] \arrow[r, "\beta"] & \bullet \\
			&         &                                                 & \bullet \arrow[u, "\rho"'] &                                                  &         &                                                 &
		\end{tikzcd}$
	\end{enumerate}
	Hence by Proposition \ref{Sin}, $\Lambda/\mathrm{rad}(\Lambda)^2$ is wild. By Proposition \ref{eAe}, $\Lambda$ is wild.
\end{proof}

\begin{Lem}{\rm(\cite[Lemma I.10.9]{E1990})}
	Let $\Lambda$ be an algebra which is not wild and let $Q$ be the quiver of $\Lambda$. Suppose that for each vertex $e$ of $Q$, at least two arrows start and two arrows end at $e$. Given any arrow $\alpha$ of $Q$ then there is at most one arrow $\beta$ such that $\beta\alpha$ is not involved in any relation.
\end{Lem}

\begin{proof}
	We may assume that $\Lambda$ is basic. Suppose $\beta\alpha$ and $\gamma\alpha$ are both not involved in any relation. Let $I:=\mathrm{rad}(\Lambda)^2\backslash\{\beta\alpha,\gamma\alpha\}$. This is an ideal of $\Lambda$, and the factor algebra $\Lambda/I$ may be generated by zero relations. Hence $\Lambda/I$ has a universal cover.
	
	By the hypothesis, $\Lambda$ has enough arrows, and hence the universal cover contains a tree of the form $\tilde{\tilde{E_7}}$, namely
	$$\begin{tikzcd}
		\bullet & \bullet \arrow[l] \arrow[r] & \bullet & \bullet \arrow[l, "\beta"'] \arrow[r, "\gamma"] & \bullet & \bullet \arrow[l] \arrow[r] & \bullet & \bullet \arrow[l] \\
		&                             &         & \bullet \arrow[u, "\alpha"']                    &         &                             &         &
	\end{tikzcd}$$
	Hence by Proposition \ref{Sin}, $\Lambda/I$ is wild. By Proposition \ref{eAe}, $\Lambda$ is wild.
\end{proof}

\section{Galois coverings of representation-infinite algebras}\label{sec-Galois-coverings-of-representation-infinite-algebras}

Let $F: \Lambda\rightarrow \Lambda/G$ be a Galois covering of locally bounded categories with group $G$, that is, $\Lambda$ is a locally bounded category and $G$ is a group of $k$-linear automorphisms of $\Lambda$ which acts freely on the objects of $\Lambda$. According to \cite[Proposition 2]{DS1985}, if $\Lambda/G$ is tame then so is $\Lambda$. It is natural to ask whether the converse is also true.

Our main objective in this section is to study modules of the second kind over $\Lambda/G$, that is, $\Lambda/G$-modules which are direct sums of indecomposable $\Lambda/G$-modules that are not belong to the image of the push-down functor $F_{\lambda}$. As an application, we show that under some conditions $\Lambda/G$ is tame if and only if $\Lambda$ is tame.
The main references of this section are \cite{DS1987, DS1985}.

\subsection{Preliminaries}
\

We keep the following notations throughout this section: \\
$A$: $k$-algebra (not necessarily finite dimensional). \\
MOD$A$: the category of right $A$-modules. \\
$\Lambda$: a connected, locally bounded $k$-category.\\
supp$M$: the support of $\Lambda$-module $M$, that is, the full subcategory of $\Lambda$ formed by all objects $x\in \Lambda$ such that $M(x)\neq 0$. \\
MOD$\Lambda$: the category of $\Lambda$-modules. \\
Mod$\Lambda$: the category of locally finite dimensional $\Lambda$-modules, that is, the category of $\Lambda$-modules $M$ with $\mathrm{dim}_{k}M(x)<\infty$ for all $x\in \Lambda$. \\
Ind$\Lambda$: the full subcategory of Mod$\Lambda$ formed by chosen representatives of all indecomposable objects. \\
mod$\Lambda$: the category of finite dimensional $\Lambda$-modules, that is, the category of $\Lambda$-modules $M$ with $\sum_{x\in \Lambda}\mathrm{dim}_{k}M(x)<\infty$. \\
ind$\Lambda$: the full subcategory of mod$\Lambda$ formed by chosen representatives of all indecomposable objects.\\
$[\mathrm{ind}\Lambda]$: the set of isomorphism classes of the finitely generated indecomposable $\Lambda$-modules.\\
$G$: a group of $k$-linear automorphisms of $\Lambda$ which acts freely on $[\mathrm{ind}\Lambda]$, that is, for any $M\in \mathrm{ind}\Lambda$, we have $^gM\not\cong M$ if $1\neq g\in G$. (Note that if $G$ acts freely on $[\mathrm{ind}\Lambda]$, then $G$ also acts freely on $\Lambda$; conversely, if $G$ acts freely on $\Lambda$ and is torsion-free, then $G$ acts freely on $[\mathrm{ind}\Lambda]$.) \\
$F$: the Galois covering functor $\Lambda\rightarrow \Lambda/G$. \\
$F_{\lambda}$: the push-down functor mod$\Lambda\rightarrow$mod$(\Lambda/G)$ associated with $F$. \\
$G_L$: the stabilizer $\{g\in G\mid gL=L\}$ of the full subcategory $L$ of $\Lambda$. \\
$\prescript{g}{}{M}$: the $\Lambda$-module $M\circ g^{-1}$, where $g\in G$. \\
$G_M$: the stabilizer $\{g\in G\mid \prescript{g}{}{M}\cong M\}$ of the $\Lambda$-module $M$. \\
$\Lambda_x$: the full subcategory of $\Lambda$ consisting of the points of all supp$M$, where $M\in\mathrm{ind}\Lambda$ with $M(x)\neq 0$.

\medskip
We have the following two general results.

\begin{Prop}
The push-down functor $F_{\lambda}: \mathrm{mod}\Lambda\rightarrow\mathrm{mod}(\Lambda/G)$ induces an injection from the set $[\mathrm{ind}\Lambda]/G$ of $G$-orbits of $[\mathrm{ind}\Lambda]$ into $[\mathrm{ind}(\Lambda/G)]$.
\end{Prop}

\begin{proof}
Since $G$ acts freely on $[\mathrm{ind}\Lambda]$, by Lemma \ref{pushdown-indec}, $F_{\lambda}M$ is indecomposable of each $M\in\mathrm{ind}\Lambda$. If $M\in\mathrm{ind}\Lambda$ and $g\in G$, then by Lemma \ref{pushdown-pull-up} $F_{\lambda} \;^gM\cong F_{\lambda}M$. So $F_{\lambda}$ induces a map $[\mathrm{ind}\Lambda]/G\rightarrow[\mathrm{ind}\Lambda/G]$. For $M,N\in\mathrm{ind}\Lambda$ with $F_{\lambda}M\cong F_{\lambda}N$, by Lemma \ref{pushdown-indec} we have $N\cong \;^gM$ for some $g\in G$. Therefore the above map is injective.
\end{proof}

\begin{Prop} {\rm(see \cite[Theorem 3.6]{Ga1981})}
The push-down functor $F_{\lambda}$ preserves Auslander-Reiten sequences and induces an isomorphism of the quotient $\Gamma_{\Lambda}/G$ of the Auslander-Reiten quiver $\Gamma_\Lambda$ of $\Lambda$ onto the union of some connected components of $\Gamma_{\Lambda/G}$.
\end{Prop}

Let $\mathrm{ind}_{1}(\Lambda/G)$ be the full subcategory of ind$(\Lambda/G)$ consists of all objects isomorphic to $F_{\lambda}M$ for some $M\in\mathrm{ind}\Lambda$, and let $\mathrm{ind}_{2}(\Lambda/G)$ be the full subcategory of ind$(\Lambda/G)$ formed by remaining indecomposables. It follows from \cite[Theorem 3.6]{Ga1981} that the push-down functor $F_{\lambda}: \mathrm{mod}\Lambda\rightarrow\mathrm{mod}(\Lambda/G)$ induces a Galois covering $\mathrm{ind}\Lambda\rightarrow\mathrm{ind}_{1}(\Lambda/G)$ with group $G$. A finite dimensional $\Lambda/G$-module $N$ is called of the first (resp. second) kind if it is isomorphic to a direct sum of modules in $\mathrm{ind}_{1}(\Lambda/G)$ (resp. $\mathrm{ind}_{2}(\Lambda/G)$). Denote $\mathrm{mod}_{1}(\Lambda/G)$ (resp. $\mathrm{mod}_{2}(\Lambda/G)$) the full subcategory of mod$(\Lambda/G)$ formed by all modules of the first (resp. second) kind.

A module $Y\in\mathrm{Ind}\Lambda$ is called weakly-$G$-periodic if supp$Y$ is infinite and $(\mathrm{supp}Y)/G_Y$ is finite. A locally bounded category $\Lambda$ is called locally support-finite if $\Lambda_x$ is finite for all $x\in \Lambda$.

A full subcategory $L$ of $\Lambda$ is said to be convex if each path of the ordinary quiver $Q_\Lambda$ of $\Lambda$ with source and terminal in $L$ has all its points in $L$. A convex full subcategory $L$ of $\Lambda$ is called a line if $L$ is isomorphic to the path category of a linear quiver (of type $A_{n}$, $A_{\infty}$ or $A_{\infty}^{\infty}$). A line $L$ is called $G$-periodic if $G_L\neq\{1\}$.

\begin{Lem} \label{extend}
Let $L$ be a convex full subcategory of $\Lambda$. Then each $L$-module $M$ can be extended uniquely to a $\Lambda$-module $\widetilde{M}$ such that $\widetilde{M}(x)=0$ for all $x\notin L$.
\end{Lem}

\begin{proof}
For each $x\in \Lambda$, let \begin{equation*} \widetilde{M}(x)=
\begin{cases}
M(x), \text{ if } x\in L; \\
0, \text{ otherwise,}
\end{cases}
\end{equation*}
and for each $\alpha\in \Lambda(x,y)$, let \begin{equation*} \widetilde{M}(\alpha)=
\begin{cases}
M(\alpha), \text{ if } x,y\in L; \\
0, \text{ otherwise.}
\end{cases}
\end{equation*}
To show that $\widetilde{M}$ is a $\Lambda$-module, it suffices to show that $\widetilde{M}(\alpha)\widetilde{M}(\beta)=\widetilde{M}(\beta\alpha)$ for each nonzero morphism $\alpha\in \Lambda(x,y)$ and each nonzero morphism $\beta\in \Lambda(y,z)$. If $x\notin L$ or $z\notin L$, we have $\widetilde{M}(\alpha)\widetilde{M}(\beta)=0=\widetilde{M}(\beta\alpha)$. If $x,z\in L$, since there exist nonzero morphisms $\alpha:x\rightarrow y$ and $\beta:y\rightarrow z$ in $\Lambda$, there exist a path $p$ from $x$ to $y$ and a path $q$ from $y$ to $z$ in the ordinary quiver $Q_\Lambda$ of $\Lambda$. Then there exists a path $qp$ in $Q_\Lambda$ from $x$ to $z$. Since $x,z\in L$ and since $y$ is a point on the path $qp$, we have $y\in L$. Then $\widetilde{M}(\alpha)\widetilde{M}(\beta)=M(\alpha)M(\beta)=M(\beta\alpha)=\widetilde{M}(\beta\alpha)$.
\end{proof}

For each line $L$ in $\Lambda$, we can associate a $\Lambda$-module $B_L$ as follows: Let $Y$ be an $L$-module with $Y(x)=k$ for each $x\in L$ and $Y(\alpha)\neq 0$ for each nonzero morphism $\alpha$ in $L$ (since $L$ is isomorphic to the path category of a linear quiver, such $L$-module $Y$ exists and is unique up to isomorphism). By Lemma \ref{extend}, $Y$ extends uniquely to a $\Lambda$-module $B_L$ with $(B_L)(x)=0$ for all $x\notin L$.

\begin{Lem}\label{invariant-subgroups}
Let $L$ be a line in $\Lambda$ and $B_L$ be the associated $\Lambda$-module. Then $G_L=G_{B_L}$.
\end{Lem}

\begin{proof}
Since supp$(B_L)=L$, we have $G_{B_L}\leq G_L$. Conversely, if $g\in G_L$, then $gL=L$, so $\prescript{g}{}{B_L}=B_L\circ g^{-1}$ is a $\Lambda$-module with supp$(\prescript{g}{}{B_L})=L$, such that $(\prescript{g}{}{B_L})(x)=k$ for each $x\in L$ and $(\prescript{g}{}{B_L})(\alpha)=(B_L)(g^{-1}\alpha)\neq 0$ for each nonzero morphism $\alpha$ in $L$. Since $L$ is a line, it follows that $B_L\cong \prescript{g}{}{B_L}$. So $g\in G_{B_L}$.
\end{proof}

\begin{Lem} \label{invariant-subgroup-of-line}
Let $L$ be a line in $\Lambda$ and $B_L$ be the associated $\Lambda$-module. If $B_L$ is weakly-$G$-periodic, then $G_{B_L}\cong\mathbb{Z}$.
\end{Lem}

\begin{proof}
Since supp$B_L=L$ is infinite, $L$ is isomorphic to the path category $kQ$ of a quiver $Q$ of type $A_{\infty}$ or $A_{\infty}^{\infty}$. Since $(\mathrm{supp}B_L)/G_{B_L}$ is finite, $G_{B_L}$ is infinite. Suppose that $Q$ is of type $A_{\infty}$. For each $g\in G_{B_L}$, $g$ induces an automorphism of $L$, which also induces an automorphism of the quiver $Q$. Since $Q$ is a quiver of type $A_{\infty}$, each automorphism of $Q$ is identity. Therefore the restriction of $g$ on $L$ is identity on objects. Since $G$ acts freely on $\Lambda$, we have $g=id_{\Lambda}$. Then $G_{B_L}$ is trivial, a contradiction. So $Q$ is of type $A_{\infty}^{\infty}$.

For each $g\in G_{B_L}$, let $\phi(g)$ be the automorphism of $Q$ induced from $g$. Since $G$ acts freely on $\Lambda$, the map $\phi:G_{B_L}\rightarrow\mathrm{Aut}(Q)$, $g\mapsto\phi(g)$ is injective, and the subgroup im$\phi$ of Aut$(Q)$ acts freely on $Q$. Since $Q$ is of type $A_{\infty}^{\infty}$, each automorphism of $Q$ is either a translation or a reflection which has a fixed point. Since im$\phi$ acts freely on $Q$, it contains only translations. Then im$\phi$ is cyclic. Since $G_{B_L}\cong\mathrm{im}\phi$ is infinite, it is an infinite cyclic group.
\end{proof}

\begin{Rem1}
If $\Lambda$ is tame and $G$ is torsion-free, then it was shown in \cite{D2001} that the stabilizer $G_Y$ of every weakly-$G$-periodic $\Lambda$-module $Y$ is an infinite cyclic group.
\end{Rem1}

\begin{Def} {\rm(\cite[Section 3.9]{Ga1981})}
A $\Lambda$-action $\nu$ of $G$ on a $\Lambda$-module $M$ consists of $k$-linear maps $\nu(g,x):M(x)\rightarrow M(gx)$ for each $g\in G$ and each $x\in \Lambda$, which satisfy
\begin{itemize}
\item $\nu(1,x)=1_{M(x)}$ for each $x\in \Lambda$;
\item $\nu(h,gx)\nu(g,x)=\nu(hg,x)$ for each $x\in \Lambda$ and each $g,h\in G$;
\item $\nu(g,x)M(\alpha)=M(g\alpha)\nu(g,y)$ for each $\alpha\in \Lambda(x,y)$ and each $g\in G$ (in other words, for each $g\in G$, the maps $\{\nu(g,x)\}_{x\in \Lambda}$ provide us with a morphism $\nu(g):M\rightarrow\prescript{g^{-1}}{}{M}$).
\end{itemize}
\end{Def}

\begin{Lem} {\rm(see \cite[Lemma 4.1]{D1996})} \label{R-action}
Let $M$ be a $\Lambda$-module such that $G_M$ is an infinite cyclic group generated by $g$. If $\mu:M\rightarrow\prescript{g^{-1}}{}{M}$ is an isomorphism, then there is a $\Lambda$-action $\nu$ of $G_M$ on $M$ given by
\begin{equation*}
\nu(g^n,x)=\begin{cases}
\mu_{g^{n-1}x}\cdots\mu_{gx}\mu_x, \text{ if } n>0; \\
\mu_{g^n x}^{-1}\mu_{g^{n+1} x}^{-1}\cdots\mu_{g^{-1} x}^{-1}, \text{ if } n<0; \\
1_{M(x)}, \text{ if } n=0.
\end{cases}
\end{equation*}
\end{Lem}

\begin{proof}
By definition it is straightforward to show that $\nu(g^m,g^n x)\nu(g^n,x)=\nu(g^{m+n},x)$ for each $m,n\in\mathbb{Z}$ and each $x\in \Lambda$. For $\alpha\in \Lambda(x,y)$ and $n\in\mathbb{Z}$, if $n\geq 0$, then
\begin{multline*}
\nu(g^n,x)M(\alpha)=\mu_{g^{n-1}x}\cdots\mu_{gx}\mu_x M(\alpha)=\mu_{g^{n-1}x}\cdots\mu_{gx}M(g\alpha)\mu_y=\cdots \\
=\mu_{g^{n-1}x}M(g^{n-1}\alpha)\mu_{g^{n-2}y}\cdots\mu_y=M(g^n \alpha)\mu_{g^{n-1}y}\cdots\mu_{gy}\mu_y=M(g^n \alpha)\nu(g^n,y);
\end{multline*}
if $n<0$, then
\begin{multline*}
\nu(g^n,x)M(\alpha)=\mu_{g^n x}^{-1}\mu_{g^{n+1} x}^{-1}\cdots\mu_{g^{-1} x}^{-1} M(\alpha)=\mu_{g^n x}^{-1}\mu_{g^{n+1} x}^{-1}\cdots\mu_{g^{-2} x}^{-1}M(g^{-1}\alpha)\mu_{g^{-1}y}^{-1}=\cdots \\
=\mu_{g^{n}x}^{-1}M(g^{n+1}\alpha)\mu_{g^{n+1} y}^{-1}\cdots\mu_{g^{-1} y}^{-1}=M(g^n \alpha)\mu_{g^n y}^{-1}\mu_{g^{n+1} y}^{-1}\cdots\mu_{g^{-1} y}^{-1}=M(g^n \alpha)\nu(g^n,y).
\end{multline*}
Therefore $\nu(h,x)M(\alpha)=M(h\alpha)\nu(h,y)$ for each $\alpha\in \Lambda(x,y)$ and each $h\in G_M$.
\end{proof}

\medskip
\subsection{Modules of the second kind}
\

\begin{Def} {\rm(\cite[Definition 3.1]{DS1987})} \label{separating family of subcategories}
A family $\mathscr{S}$ of full subcategories of $\Lambda$ is called separating (with respect to $G$) if $\mathscr{S}$ satisfies the following conditions:
\begin{itemize}
\item[$(i)$] for each $L\in\mathscr{S}$ and $g\in G$, $gL\in\mathscr{S}$;
\item[$(ii)$] for each $L\in\mathscr{S}$ and for each $G$-orbit $\mathscr{O}$ of $\Lambda$, $L\cap\mathscr{O}$ is contained in finitely many $G_L$-orbits;
\item[$(iii)$] for any two different $L,L'\in\mathscr{S}$, $L\cap L'$ is locally support-finite;
\item[$(iv)$] for each weakly-$G$-periodic $\Lambda$-module $Y$ there exists an $L\in\mathscr{S}$ such that $\mathrm{supp}Y\subseteq L$.
\end{itemize}
\end{Def}

\begin{Ex1}
If there exists a weakly-$G$-periodic $\Lambda$-module $Y$ such that $\mathrm{supp}Y=\Lambda$, then $\mathscr{S}=\{\Lambda\}$ is a separating family of subcategories of $\Lambda$ with respect to $G$.
\end{Ex1}

\begin{Thm} {\rm(\cite[Theorem 3.1]{DS1987})} \label{modules-of-the-second-kind}
Let $\Lambda$ be a locally bounded $k$-category and $G$ a group of automorphisms of $\Lambda$ which acts freely on $[\mathrm{ind}\Lambda]$. Let $\mathscr{S}$ be a separating family of subcategories of $\Lambda$ with respect to $G$ and $\mathscr{S}_0$ a fixed set of representatives of $G$-orbits of $\mathscr{S}$. Then there is an equivalence of categories
$$E:\coprod_{L\in\mathscr{S}_0}(\mathrm{mod}L/G_L)/(\mathrm{mod}_1 L/G_L)\rightarrow(\mathrm{mod}(\Lambda/G))/(\mathrm{mod}_1 (\Lambda/G)).$$

As a consequence, the AR-quiver $\Gamma_{\Lambda/G}$ of $\Lambda/G$ is isomorphic to the disjoint union of translation quivers $(\Gamma_{\Lambda}/G)\sqcup(\sqcup_{L\in\mathscr{S}_0}(\Gamma_{L/G_L})_2)$, where $(\Gamma_{L/G_L})_2$ is the union of connected components of $\Gamma_{L/G_L}$ whose points are $L/G_L$-modules of second kind.
\end{Thm}

Let $A$ be a $k$-algebra (not necessarily finite dimensional). An $A$-$\Lambda$-bimodule $Q$ is a contravariant functor from $\Lambda$ to MOD$A^{op}$. Each $A$-$\Lambda$-bimodule $Q$ induces a functor $-\otimes_{A}Q:$MOD$A\rightarrow$MOD$\Lambda$, where $V\otimes_{A}Q$ is the $\Lambda$-module given by $(V\otimes_{A}Q)(x)=V\otimes_{A}Q(x)$ for each $V\in$MOD$A$ and $x\in \Lambda$.

For a weakly-$G$-periodic $\Lambda$-module $B$ with a $\Lambda$-action $\nu$ of $G_B$ on $B$, $F_{\lambda}B$ carries the structure of a $kG_B$-$\Lambda/G$-bimodule: for each $Gx\in \Lambda/G$, the action of $G_B$ on $(F_{\lambda}B)(Gx)=\oplus_{z\in Gx}B(z)$ is given by $g\cdot b=\nu(g,z)(b)$ for each $g\in G_B$ and each $b\in B(z)$. Note that if $W_x$ is a set of representatives of the $G_B$-orbits of $Gx$, and if $\{b_{y1},\cdots,b_{y,n_y}\}$ is a $k$-basis of $B(y)$ for each $y\in W_x$, then $(F_{\lambda}B)(Gx)$ is a free $kG_B$-module on the set $\{b_{yi}\mid y\in W_x, 1\leq i\leq n_y\}$. In particular, if $B=B_L$ is weakly-$G$-periodic, where $L$ is a line in $\Lambda$, according to Lemma \ref{invariant-subgroup-of-line} and Lemma \ref{R-action}, there is a $\Lambda$-action $\nu$ of $G_B\cong\mathbb{Z}$ on $B$, and then $F_{\lambda}B$ is a $k[T,T^{-1}]$-$\Lambda/G$-bimodule. Denote $\Phi^L$ the functor $-\otimes_{k[T,T^{-1}]}F_{\lambda}B:\mathrm{mod}k[T,T^{-1}]\rightarrow\mathrm{mod}(\Lambda/G)$.

\begin{Ex1} \label{Kronecker-quiver}
Let $\Lambda$ be the path category of the quiver
$$\begin{tikzpicture}
\draw[->] (-3.2,0.8) -- (-3.8,0.2);
\draw[->] (-2.8,0.8) -- (-2.2,0.2);
\draw[->] (-1.2,0.8) -- (-1.8,0.2);
\draw[->] (-0.8,0.8) -- (-0.2,0.2);
\draw[->] (0.8,0.8) -- (0.2,0.2);
\draw[->] (1.2,0.8) -- (1.8,0.2);
\draw[->] (2.8,0.8) -- (2.2,0.2);
\draw[->] (3.2,0.8) -- (3.8,0.2);
\node at(-4,0) {$-4$};
\node at(-3,1) {$-3$};
\node at(-2,0) {$-2$};
\node at(-1,1) {$-1$};
\node at(0,0) {$0$};
\node at(2,0) {$2$};
\node at(1,1) {$1$};
\node at(3,1) {$3$};
\node at(4,0) {$4$};
\node at(-5,0.5) {$\cdots$};
\node at(5,0.5) {$\cdots$};
\node at(-3.65,0.65) {$\alpha$};
\node at(-2.35,0.7) {$\beta$};
\node at(-1.65,0.65) {$\alpha$};
\node at(-0.35,0.7) {$\beta$};
\node at(0.35,0.65) {$\alpha$};
\node at(1.65,0.7) {$\beta$};
\node at(2.35,0.65) {$\alpha$};
\node at(3.65,0.7) {$\beta$};
\end{tikzpicture}$$
and let $G=\langle g\rangle$ be a group of automorphisms of $\Lambda$, where $g$ is the automorphism of $\Lambda$ induced from the quiver automorphism $i\mapsto i+2$. Then $\Lambda/G$ is the path category of the quiver
$$\begin{tikzpicture}
\draw[->] (0.3,0.1) -- (1.7,0.1);
\draw[->] (0.3,-0.1) -- (1.7,-0.1);
\node at(0,0) {$y$};
\node at(2,0) {$x$};
\node at(1,0.3) {$a$};
\node at(1,-0.4) {$b$};
\end{tikzpicture}.$$
Let $L=\Lambda$ be a line in $\Lambda$. Then $B=B_L$ is the $\Lambda$-module
$$\begin{tikzpicture}
\draw[->] (-3.8,0.2) -- (-3.2,0.8);
\draw[->] (-2.2,0.2) -- (-2.8,0.8);
\draw[->] (-1.8,0.2) -- (-1.2,0.8);
\draw[->] (-0.2,0.2) -- (-0.8,0.8);
\draw[->] (0.2,0.2) -- (0.8,0.8);
\draw[->] (1.8,0.2) -- (1.2,0.8);
\draw[->] (2.2,0.2) -- (2.8,0.8);
\draw[->] (3.8,0.2) -- (3.2,0.8);
\node at(-4,0) {$k$};
\node at(-3,1) {$k$};
\node at(-2,0) {$k$};
\node at(-1,1) {$k$};
\node at(0,0) {$k$};
\node at(2,0) {$k$};
\node at(1,1) {$k$};
\node at(3,1) {$k$};
\node at(4,0) {$k$};
\node at(-5,0.5) {$\cdots$};
\node at(5,0.5) {$\cdots$};
\node at(-3.65,0.65) {$1$};
\node at(-2.35,0.7) {$1$};
\node at(-1.65,0.65) {$1$};
\node at(-0.35,0.7) {$1$};
\node at(0.35,0.65) {$1$};
\node at(1.65,0.7) {$1$};
\node at(2.35,0.65) {$1$};
\node at(3.65,0.7) {$1$};
\end{tikzpicture}$$
with a $\Lambda$-action $\nu$ of $G_B=G$ on $B$ given by $\nu(g^i,j)=id_{k}:B(j)\rightarrow B(2i+j)$.
It follows that $F_{\lambda}B$ is the $k[T,T^{-1}]$-$R/G$-bimodule
$$\begin{tikzpicture}
\draw[->] (2.4,0.1) -- (1,0.1);
\draw[->] (2.4,-0.1) -- (1,-0.1);
\node at(0,0) {$k[T,T^{-1}]$};
\node at(3.4,0) {$k[T,T^{-1}]$};
\node at(1.7,0.3) {$T$};
\node at(1.7,-0.3) {$1$};
\end{tikzpicture}.$$
\end{Ex1}

\medskip
Let $\mathscr{L}$ be the set of all subcategories supp$Y\subseteq \Lambda$, where $Y$ ranges over all weakly-$G$-periodic $\Lambda$-modules, and let $\mathscr{L}_0$ be a fixed set of representatives of the $G$-orbits of $\mathscr{L}$.

\begin{Thm} {\rm(\cite[Theorem 3.6]{DS1987})} \label{modules-of-the-second-kind-2}
Let $\Lambda$ be a locally bounded $k$-category and let $G$ be a group of automorphisms of $\Lambda$ which acts freely on $[\mathrm{ind}\Lambda]$. Assume that $\mathscr{L}$ consists only of lines in $\Lambda$. The family of functors $\Phi^L$, $L\in\mathscr{L}_0$, induces an equivalence of categories
$$\Phi:\coprod_{\mathscr{L}_0}\mathrm{mod}k[T,T^{-1}]\rightarrow(\mathrm{mod}(\Lambda/G))/(\mathrm{mod}_1 (\Lambda/G)).$$

In particular, $(\Gamma_{\Lambda/G})_2=\sqcup_{\mathscr{L}_0}\Gamma_{k[T,T^{-1}]}$, where $\Gamma_{k[T,T^{-1}]}$ is the translation quiver of the category of finite dimensional $k[T,T^{-1}]$-modules.

Moreover, $\Lambda/G$ is tame if and only if so is $\Lambda$.

\end{Thm}

Assume that $\Lambda=kQ/I$ is a locally bounded $k$-category satisfying the following conditions: $(a)$ $\Lambda$ is Schurian (that is, $\mathrm{dim}_{k}\Lambda(x,y)\leq 1$ for all $x,y\in \Lambda$), $(b)$ $Q$ is connected, directed and interval-finite (that is, for all $x,y\in Q_0$, the number of paths of $Q$ from $x$ to $y$ is finite), $(c)$ $\Pi(Q,I)=\{1\}$, $(d)$ the support of any indecomposable finite dimensional $\Lambda$-module is representation-finite or belongs to the Bongartz-Happel-Vossieck list (see \cite[Section 10.8]{BGRS1985}) of critical algebras. Let $G$ be a group of $k$-linear automorphisms of $\Lambda$ acting freely in the objects of $\Lambda$. It follows that $G$ acts also freely on $[\mathrm{ind}\Lambda]$.

\begin{Prop} {\rm(\cite[Proposition 5.1]{DS1987})} \label{tame}
Assume that $\Lambda$ and $G$ satisfy the conditions above. Then every weakly-$G$-periodic $\Lambda$-module is linear. As a consequence $\Lambda/G$ is tame.
\end{Prop}

\subsection{Examples and applications}
\

The following example is related to the covering map in Example \ref{Kronecker-quiver}.

\begin{Ex1} \
Let $\Lambda$ and $G$ be the locally bounded category and the group of automorphisms of $\Lambda$ in Example \ref{Kronecker-quiver} respectively. If $Y$ is a weakly-$G$-periodic $\Lambda$-module, then $L=\mathrm{supp}Y$ is an infinite connected subcategory of $\Lambda$ with $G_L$ infinite. It follows that $L=\Lambda$. Since $G$ is torsion-free which acts freely on $\Lambda$, it also acts freely on $[\mathrm{ind}\Lambda]$. Since each $M\in\mathrm{ind}\Lambda$ is isomorphic to some $B_{L'}$ with $L'$ a finite line in $\Lambda$, each indecomposable $\Lambda/G$-module of the first kind belongs to one of the following forms.
$$\begin{tikzpicture}
\draw[->] (3.6,0.1) -- (0.4,0.1);
\draw[->] (3.6,-0.1) -- (0.4,-0.1);
\draw[dashed] (2,0.2) -- (2,0.6);
\draw[dashed] (2,-0.2) -- (2,-0.6);
\node at(0,0) {$k^n$};
\node at(4,0) {$k^{n+1}$};
\node at(2,0.4) {$\begin{bmatrix}
I_n & 0
\end{bmatrix}$};
\node at(2,-0.4) {$\begin{bmatrix}
0 & I_n
\end{bmatrix}$};
\end{tikzpicture}$$

$$\begin{tikzpicture}
\draw[->] (3.6,0.1) -- (0.4,0.1);
\draw[->] (3.6,-0.1) -- (0.4,-0.1);
\draw[dashed] (1.75,0.8) -- (2.25,0.8);
\draw[dashed] (1.75,-0.8) -- (2.25,-0.8);
\node at(0,0) {$k^{n+1}$};
\node at(4,0) {$k^{n}$};
\node at(2,0.8) {$\begin{bmatrix}
I_n \\
0
\end{bmatrix}$};
\node at(2,-0.8) {$\begin{bmatrix}
0 \\
I_n
\end{bmatrix}$};
\end{tikzpicture}$$

$$\begin{tikzpicture}
\draw[->] (3.6,0.1) -- (0.4,0.1);
\draw[->] (3.6,-0.1) -- (0.4,-0.1);
\node at(0,0) {$k^{n}$};
\node at(4,0) {$k^{n}$};
\node at(2,0.4) {$I_n$};
\node at(2,-1.4) {$\begin{bmatrix}
0 & 1 & & \\
 & 0 & \ddots & \\
& & \ddots & 1 \\
& & & 0
\end{bmatrix}$};
\end{tikzpicture}$$

$$\begin{tikzpicture}
\draw[->] (3.6,0.1) -- (0.4,0.1);
\draw[->] (3.6,-0.1) -- (0.4,-0.1);
\node at(0,0) {$k^{n}$};
\node at(4,0) {$k^{n}$};
\node at(2,1.4) {$\begin{bmatrix}
0 & & & \\
1 & 0 & & \\
& \ddots & \ddots & \\
& & 1 & 0
\end{bmatrix}$};
\node at(2,-0.4) {$I_n$};
\end{tikzpicture}$$

According to Theorem \ref{modules-of-the-second-kind-2}, the functor $$\Phi^\Lambda=-\otimes_{k[T,T^{-1}]}F_{\lambda}(B_\Lambda):\mathrm{mod}k[T,T^{-1}]\rightarrow\mathrm{mod}(\Lambda/G)$$
induces an equivalence of categories
$$\mathrm{mod}k[T,T^{-1}]\rightarrow(\mathrm{mod}(\Lambda/G))/[\mathrm{mod}_1 (\Lambda/G)],$$
where $F_{\lambda}(B_\Lambda)$ is the $k[T,T^{-1}]$-$\Lambda/G$-bimodule
$$\begin{tikzpicture}
\draw[->] (2.4,0.1) -- (1,0.1);
\draw[->] (2.4,-0.1) -- (1,-0.1);
\node at(0,0) {$k[T,T^{-1}]$};
\node at(3.4,0) {$k[T,T^{-1}]$};
\node at(1.7,0.3) {$T$};
\node at(1.7,-0.3) {$1$};
\end{tikzpicture}.$$
Since each indecomposable finite dimensional $k[T,T^{-1}]$-module is isomorphic to a $k[T,T^{-1}]$-module of the form $k[T]/(T-\lambda)^n$ for some positive integer $n$ and some $\lambda\in k^{*}$, each indecomposable $\Lambda/G$-module of the second kind is of the form
$$\begin{tikzpicture}
\draw[->] (3,0.1) -- (1.4,0.1);
\draw[->] (3,-0.1) -- (1.4,-0.1);
\node at(0,0) {$k[T]/(T-\lambda)^n$};
\node at(4.4,0) {$k[T]/(T-\lambda)^n$};
\node at(2.2,0.3) {$T$};
\node at(2.2,-0.3) {$1$};
\end{tikzpicture},$$
which is isomorphic to
$$\begin{tikzpicture}
\draw[->] (3.6,0.1) -- (0.4,0.1);
\draw[->] (3.6,-0.1) -- (0.4,-0.1);
\node at(0,0) {$k^n$};
\node at(4,0) {$k^n$};
\node at(2,0.4) {$J_n(\lambda)$};
\node at(2,-0.3) {$1$};
\end{tikzpicture}$$
where $J_{n}(\lambda)$ is the matrix
$$\begin{pmatrix}
\lambda & & & \\
1 & \lambda & & \\
& \ddots & \ddots & \\
& & 1 & \lambda
\end{pmatrix}_{n\times n}.$$
\end{Ex1}

\bigskip
Next we give an application to determining the indecomposable modules of a special biserial algebra. Recall that a finite dimensional algebra $A$ is said to be special biserial if it is Morita equivalent to an algebra of the form $kQ/I$, where $Q$ is a quiver and $I$ is an admissible ideal in $kQ$ such that
\begin{itemize}
\item[$(R1)$] For each vertex $x$ of $Q$, there are at most two arrows of $Q$ starting at $x$ and there are at most two arrows of $Q$ ending at $x$.
\item[$(R2)$] For each arrow $\alpha$ of $Q$, there is at most one arrow $\beta$ (resp. $\gamma$) of $Q$ such that $\beta\alpha\notin I$ (resp. $\alpha\gamma\notin I$).
\end{itemize}
Moreover, a special biserial algebra $A$ is called a string algebra if it is Morita equivalent to an algebra of the form $kQ/I$, where $(Q,I)$ satisfies $(R1), (R2)$ and $I$ is generated by paths.

\medskip
In order to study the representations of special biserial algebras we only need to consider string algebras (see \cite[II.1.3]{E1990}).

Let $A$ be a string algebra given by the connected quiver with relations $(Q,I)$, where $(Q,I)$ satisfies $(R1), (R2)$ and $I$ is generated by paths. According to Proposition \ref{universal-cover}, there exists a universal cover $\pi:(\widetilde{Q},\widetilde{I})\rightarrow(Q,I)$ of $(Q,I)$, which is a Galois covering with group $G=\Pi(Q,I)$. It follows that $\pi$ induces a Galois covering $F:\Lambda=k\widetilde{Q}/\widetilde{I}\rightarrow A=kQ/I$ with group $G$ of associated categories (here we consider $A$ as a $k$-category). Since the ideal $I$ of $kQ$ is generated by zero relations, $G=\Pi(Q,I)$ is isomorphic to the fundamental group of $Q$, which is free. It follows that $G$ is torsion-free. Since $G$ acts freely on $\Lambda$, it also acts freely on $[\mathrm{ind}\Lambda]$.

\begin{Prop} {\rm(\cite[Proposition 5.2]{DS1987})}
Let $A=kQ/I$ be a string algebra. Then $A$ is tame, and each indecomposable $A$-module of the first (resp. second) kind is just a string (resp. band) module.
\end{Prop}

\begin{proof}
Since the ideal $I$ of $kQ$ is generated by paths, $\Pi(Q,I)$ is isomorphic to the fundamental group of the quiver $Q$, and $\widetilde{Q}$ is a tree (cf. the remarks after Corollary \ref{trivial-fundamental-group-the-converse}). Therefore the conditions $(a)$, $(b)$, $(c)$ before Proposition \ref{tame} hold for $\Lambda=k\widetilde{Q}/\widetilde{I}$. Moreover, it is known that the support of each module in $\mathrm{ind}\Lambda$ is a finite line (see \cite{BR1981} and \cite{PS1983}). So the condition $(d)$ before Proposition \ref{tame} also holds for $\Lambda$. By Proposition \ref{tame}, $A\cong \Lambda/G$ is tame.

Denote $\mathscr{L}$ the set of all subcategories $\mathrm{supp}Y\subseteq \Lambda$, where $Y$ ranges over all weakly-$G$-periodic $\Lambda$-modules, and denote $\mathscr{L}_0$ a fixed set of representatives of the $G$-orbits of $\mathscr{L}$. For any $X\in\mathrm{Ind}\Lambda$, it follows from \cite[Section 4]{DS1987} that $\mathrm{supp}X=\cup_{n\geq 0}\mathrm{supp}Y_n$ for a sequence of modules $\{Y_n\}_{n\geq 0}$ in $\mathrm{ind}\Lambda$ such that $\mathrm{supp}Y_n\subseteq \mathrm{supp}Y_{n+1}$ for all $n\geq 0$. Since the support of each $Y_n$ is a finite line in $\Lambda$, the support of $X$ is also a line in $\Lambda$. According to Theorem \ref{modules-of-the-second-kind-2}, the family of functors $\Phi^L$, $L\in\mathscr{L}_0$, induces an equivalence of categories
$$\Phi:\coprod_{\mathscr{L}_0}\mathrm{mod}k[T,T^{-1}]\rightarrow(\mathrm{mod}A)/[\mathrm{mod}_1 A].$$

Let $M\in\mathrm{ind}\Lambda$. Suppose that the support of $M$ is a finite line $L$, then $M\cong B_L$. It follows that $F_{\lambda}M$ is a string module. Conversely, for each string module $X$ of $A$ given by the string $w$, we can lift $w$ to a finite line $L$ in $\Lambda$, and $F_{\lambda}(B_L)\cong X$. Therefore each indecomposable $A$-module of the first kind is just a string module.

For each weakly-$G$-periodic $\Lambda$-module $Y$, $\mathrm{supp}Y$ is a $G$-periodic line, that is, a line $L$ of type $A_{\infty}^{\infty}$ with $G_L\cong\mathbb{Z}$. Conversely, for each $G$-periodic line $L$ in $\Lambda$, there exists a weakly-$G$-periodic $\Lambda$-module $B_L$ such that $\mathrm{supp}B_L=L$. Therefore $\mathscr{L}$ is the set of $G$-periodic lines in $\Lambda$. Each $G$-periodic line can be seen as an infinite reduced walk
$$w=\cdots\alpha_2\alpha_1\alpha_0\alpha_{-1}\alpha_{-2}\cdots$$
of $\widetilde{Q}$ with $\alpha_i\in\widetilde{Q}_{1}\cup\widetilde{Q}_{1}^{-1}$ such that
\begin{itemize}
\item there is no subpath of $w$ or $w^{-1}$ which belongs to $\widetilde{I}$;
\item there exist some $1\neq g\in G$ and some positive integer $m$ such that $g\alpha_i=\alpha_{i+m}$ for all $i\in\mathbb{Z}$.
\end{itemize}
Let $m$ be the smallest positive integer such that there exists some $g\in G$ with $g\alpha_i=\alpha_{i+m}$ for all $i\in\mathbb{Z}$. It follows that such a walk $w$ corresponds to a band $\pi(\alpha_{m-1})\cdots\pi(\alpha_1)\pi(\alpha_0)$ in $A=kQ/I$ (see \cite[Section II.2]{E1990}), and two such walks
$$w=\cdots\alpha_2\alpha_1\alpha_0\alpha_{-1}\alpha_{-2}\cdots$$
and
$$v=\cdots\beta_2\beta_1\beta_0\beta_{-1}\beta_{-2}\cdots$$
correspond to the same band in $A$ if and only if there exist $g\in G$ and $m\in\mathbb{Z}$ such that $\beta_i=g\alpha_{i+m}$ for all $i\in\mathbb{Z}$. So there is a one-to-one correspondence between the $G$-orbits of $G$-periodic lines in $\Lambda$ and the equivalence classes of bands in $A$. Moreover, for each $L\in\mathscr{L}_0$ and each finite dimensional indecomposable $k[T,T^{-1}]$-module $V$, if $L$ corresponds to the band $b$ in $A$, then $\Phi^L(V)$ is a band module over $A$ given by the band $b$ and the $k[T,T^{-1}]$-module $V$ (see \cite[Section II.3]{E1990}). Therefore each indecomposable $A$-module of the second kind is just a band module.
\end{proof}

The algebra in the following example is monomial but not special biserial. We will show that it is tame by using Proposition \ref{tame}.

\begin{Ex1} {\rm(\cite[Example 5.3]{DS1987})}
Let $A=kQ/I$, where $Q$ is the quiver
$$
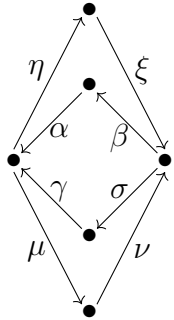
\begin{tikzpicture}
\draw[->] (0.9,0.9) -- (0.1,0.1);
\draw[->] (0.9,-0.9) -- (0.1,-0.1);
\draw[->] (1.9,0.1) -- (1.1,0.9);
\draw[->] (1.9,-0.1) -- (1.1,-0.9);
\draw[->] (0,0.15) -- (0.9,1.9);
\draw[->] (1.1,1.9) -- (2,0.15);
\draw[->] (0,-0.15) -- (0.9,-1.9);
\draw[->] (1.1,-1.9) -- (2,-0.15);
\node at(0.6,0.4) {$\alpha$};
\node at(1.4,0.3) {$\beta$};
\node at(0.6,-0.4) {$\gamma$};
\node at(1.4,-0.4) {$\sigma$};
\node at(0.3,1.2) {$\eta$};
\node at(1.7,1.2) {$\xi$};
\node at(0.3,-1.2) {$\mu$};
\node at(1.7,-1.2) {$\nu$};
\node at(0,0) {$\bullet$};
\node at(2,0) {$\bullet$};
\node at(1,1) {$\bullet$};
\node at(1,-1) {$\bullet$};
\node at(1,2) {$\bullet$};
\node at(1,-2) {$\bullet$};
\end{tikzpicture}$$
and $I$ is the ideal in $kQ$ generated by
$$\beta\xi,\sigma\nu,\eta\alpha\beta,\mu\alpha\beta,\eta\gamma\sigma,\mu\gamma\sigma,\nu\mu\alpha,\nu\mu\gamma,\xi\eta\alpha,\xi\eta\gamma.$$
There is a universal Galois covering $F:\Lambda\rightarrow \Lambda/G=A$ with group $G=\Pi(Q,I)$, where $\Lambda=k\widetilde{Q}/\widetilde{I}$ is given by the following quiver $\widetilde{Q}$ and the ideal $\widetilde{I}$ is generated by all elements of the form $\beta\xi,\sigma\nu,\eta\alpha\beta,\mu\alpha\beta,\eta\gamma\sigma,\mu\gamma\sigma,\nu\mu\alpha,\nu\mu\gamma,\xi\eta\alpha,\xi\eta\gamma$.
$$\begin{tikzpicture}
\draw[->] (0.9,0.9) -- (0.1,0.1);
\draw[->] (-0.9,0.9) -- (-0.1,0.1);
\draw[->] (-0.1,-0.1) -- (-0.9,-1.1);
\draw[->] (0.1,-0.1) -- (0.9,-1.1);
\draw[->] (1.9,1.9) -- (1.1,1.1);
\draw[->] (2.9,2.9) -- (2.1,2.1);
\draw[->] (3.9,3.9) -- (3.1,3.1);
\draw[->] (4.9,4.9) -- (4.1,4.1);
\draw[->] (2.1,1.9) -- (2.9,1.1);
\draw[->] (1.1,2.9) -- (1.9,2.1);
\draw[->] (3.1,0.9) -- (3.9,0.1);
\draw[->] (4.1,-0.1) -- (4.9,-0.9);
\draw[->] (3.9,-0.1) -- (3.1,-0.9);
\draw[->] (4.9,0.9) -- (4.1,0.1);
\draw[->] (-1.9,1.9) -- (-1.1,1.1);
\draw[->] (-2.9,2.9) -- (-2.1,2.1);
\draw[->] (-3.9,3.9) -- (-3.1,3.1);
\draw[->] (-4.9,4.9) -- (-4.1,4.1);
\draw[->] (-2.1,1.9) -- (-2.9,1.1);
\draw[->] (-1.1,2.9) -- (-1.9,2.1);
\draw[->] (-3.1,0.9) -- (-3.9,0.1);
\draw[->] (-4.1,-0.1) -- (-4.9,-0.9);
\draw[->] (-3.9,-0.1) -- (-3.1,-0.9);
\draw[->] (-4.9,0.9) -- (-4.1,0.1);
\draw[->] (-1.1,-1.3) -- (-1.9,-2.3);
\draw[->] (-2.1,-2.5) -- (-2.9,-3.5);
\draw[->] (-3.1,-3.7) -- (-3.9,-4.7);
\draw[->] (-4.1,-4.9) -- (-4.9,-5.9);
\draw[->] (1.1,-1.3) -- (1.9,-2.3);
\draw[->] (2.1,-2.5) -- (2.9,-3.5);
\draw[->] (3.1,-3.7) -- (3.9,-4.7);
\draw[->] (4.1,-4.9) -- (4.9,-5.9);
\draw[->] (-2.9,-1.5) -- (-2.1,-2.3);
\draw[->] (-1.9,-2.5) -- (-1.1,-3.3);
\draw[->] (-4.9,-3.9) -- (-4.1,-4.7);
\draw[->] (-3.9,-4.9) -- (-3.1,-5.7);
\draw[->] (2.9,-1.5) -- (2.1,-2.3);
\draw[->] (1.9,-2.5) -- (1.1,-3.3);
\draw[->] (4.9,-3.9) -- (4.1,-4.7);
\draw[->] (3.9,-4.9) -- (3.1,-5.7);
\node at(0.4,0.6) {$\alpha$};
\node at(1.4,1.6) {$\beta$};
\node at(2.4,2.6) {$\nu$};
\node at(3.4,3.6) {$\mu$};
\node at(4.4,4.6) {$\alpha$};
\node at(-0.4,0.6) {$\gamma$};
\node at(-1.4,1.6) {$\sigma$};
\node at(-2.4,2.6) {$\xi$};
\node at(-3.4,3.6) {$\eta$};
\node at(-4.4,4.6) {$\gamma$};
\node at(4.6,-0.4) {$\eta$};
\node at(3.6,0.6) {$\gamma$};
\node at(2.6,1.6) {$\sigma$};
\node at(1.6,2.6) {$\xi$};
\node at(4.4,0.6) {$\alpha$};
\node at(3.4,-0.4) {$\mu$};
\node at(-4.6,-0.4) {$\mu$};
\node at(-3.6,0.6) {$\alpha$};
\node at(-2.6,1.6) {$\beta$};
\node at(-1.6,2.6) {$\nu$};
\node at(-4.4,0.6) {$\gamma$};
\node at(-3.4,-0.4) {$\eta$};
\node at(-0.6,-0.5) {$\mu$};
\node at(-1.6,-1.7) {$\nu$};
\node at(-2.6,-2.9) {$\beta$};
\node at(-3.6,-4.1) {$\alpha$};
\node at(-4.6,-5.3) {$\mu$};
\node at(-4.4,-4.2) {$\gamma$};
\node at(-3.4,-5.2) {$\eta$};
\node at(0.6,-0.5) {$\eta$};
\node at(1.6,-1.7) {$\xi$};
\node at(2.6,-2.9) {$\sigma$};
\node at(3.6,-4.1) {$\gamma$};
\node at(4.6,-5.3) {$\eta$};
\node at(4.4,-4.2) {$\alpha$};
\node at(3.4,-5.2) {$\mu$};
\node at(-2.6,-2.1) {$\xi$};
\node at(-1.6,-3) {$\sigma$};
\node at(2.6,-2.1) {$\nu$};
\node at(1.6,-3) {$\beta$};
\node at(-3.1,-1.3) {$\cdot$};
\node at(-3.2,-1.2) {$\cdot$};
\node at(-3.3,-1.1) {$\cdot$};
\node at(-0.9,-3.5) {$\cdot$};
\node at(-0.8,-3.6) {$\cdot$};
\node at(-0.7,-3.7) {$\cdot$};
\node at(-5.1,-3.7) {$\cdot$};
\node at(-5.2,-3.6) {$\cdot$};
\node at(-5.3,-3.5) {$\cdot$};
\node at(-2.9,-5.9) {$\cdot$};
\node at(-2.8,-6) {$\cdot$};
\node at(-2.7,-6.1) {$\cdot$};
\node at(-5.1,-6.1) {$\cdot$};
\node at(-5.2,-6.22) {$\cdot$};
\node at(-5.3,-6.34) {$\cdot$};
\node at(3.1,-1.3) {$\cdot$};
\node at(3.2,-1.2) {$\cdot$};
\node at(3.3,-1.1) {$\cdot$};
\node at(0.9,-3.5) {$\cdot$};
\node at(0.8,-3.6) {$\cdot$};
\node at(0.7,-3.7) {$\cdot$};
\node at(5.1,-3.7) {$\cdot$};
\node at(5.2,-3.6) {$\cdot$};
\node at(5.3,-3.5) {$\cdot$};
\node at(2.9,-5.9) {$\cdot$};
\node at(2.8,-6) {$\cdot$};
\node at(2.7,-6.1) {$\cdot$};
\node at(5.1,-6.1) {$\cdot$};
\node at(5.2,-6.22) {$\cdot$};
\node at(5.3,-6.34) {$\cdot$};
\node at(5.1,5.1) {$\cdot$};
\node at(5.2,5.2) {$\cdot$};
\node at(5.3,5.3) {$\cdot$};
\node at(5.1,1.1) {$\cdot$};
\node at(5.2,1.2) {$\cdot$};
\node at(5.3,1.3) {$\cdot$};
\node at(0.9,3.1) {$\cdot$};
\node at(0.8,3.2) {$\cdot$};
\node at(0.7,3.3) {$\cdot$};
\node at(2.9,-1.1) {$\cdot$};
\node at(2.8,-1.2) {$\cdot$};
\node at(2.7,-1.3) {$\cdot$};
\node at(5.1,-1.1) {$\cdot$};
\node at(5.2,-1.2) {$\cdot$};
\node at(5.3,-1.3) {$\cdot$};
\node at(-5.1,5.1) {$\cdot$};
\node at(-5.2,5.2) {$\cdot$};
\node at(-5.3,5.3) {$\cdot$};
\node at(-5.1,1.1) {$\cdot$};
\node at(-5.2,1.2) {$\cdot$};
\node at(-5.3,1.3) {$\cdot$};
\node at(-0.9,3.1) {$\cdot$};
\node at(-0.8,3.2) {$\cdot$};
\node at(-0.7,3.3) {$\cdot$};
\node at(-2.9,-1.1) {$\cdot$};
\node at(-2.8,-1.2) {$\cdot$};
\node at(-2.7,-1.3) {$\cdot$};
\node at(-5.1,-1.1) {$\cdot$};
\node at(-5.2,-1.2) {$\cdot$};
\node at(-5.3,-1.3) {$\cdot$};
\node at(0,0) {$\bullet$};
\node at(-1,1) {$\bullet$};
\node at(-1,-1.2) {$\bullet$};
\node at(1,-1.2) {$\bullet$};
\node at(1,1) {$\bullet$};
\node at(-2,2) {$\bullet$};
\node at(-3,3) {$\bullet$};
\node at(-3,1) {$\bullet$};
\node at(-4,4) {$\bullet$};
\node at(-4,0) {$\bullet$};
\node at(2,2) {$\bullet$};
\node at(3,3) {$\bullet$};
\node at(3,1) {$\bullet$};
\node at(4,4) {$\bullet$};
\node at(4,0) {$\bullet$};
\node at(-2,-2.4) {$\bullet$};
\node at(-3,-3.6) {$\bullet$};
\node at(-4,-4.8) {$\bullet$};
\node at(2,-2.4) {$\bullet$};
\node at(3,-3.6) {$\bullet$};
\node at(4,-4.8) {$\bullet$};
\end{tikzpicture}$$
Let $M\in\mathrm{ind}\Lambda$ such that $\mathrm{supp}M$ is not contained in each full subcategory of $\Lambda$ of the form
$$
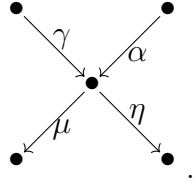
\begin{tikzpicture}
\draw[->] (0.9,0.9) -- (0.1,0.1);
\draw[->] (-0.9,0.9) -- (-0.1,0.1);
\draw[->] (-0.1,-0.1) -- (-0.9,-0.9);
\draw[->] (0.1,-0.1) -- (0.9,-0.9);
\node at(0,0) {$\bullet$};
\node at(1,1) {$\bullet$};
\node at(-1,1) {$\bullet$};
\node at(-1,-1) {$\bullet$};
\node at(1,-1) {$\bullet$};
\node at(0.6,0.4) {$\alpha$};
\node at(0.6,-0.4) {$\eta$};
\node at(-0.4,0.6) {$\gamma$};
\node at(-0.4,-0.6) {$\mu$};
\end{tikzpicture}.$$
Then $\mathrm{supp}M$ contains a vertex $x$ of $\widetilde{Q}$ of the form
$$
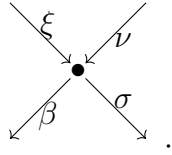
\begin{tikzpicture}
\draw[->] (0.9,0.9) -- (0.1,0.1);
\draw[->] (-0.9,0.9) -- (-0.1,0.1);
\draw[->] (-0.1,-0.1) -- (-0.9,-0.9);
\draw[->] (0.1,-0.1) -- (0.9,-0.9);
\node at(0,0) {$\bullet$};
\node at(0.6,0.4) {$\nu$};
\node at(0.6,-0.4) {$\sigma$};
\node at(-0.4,0.6) {$\xi$};
\node at(-0.4,-0.6) {$\beta$};
\node at(1.2,-1) {$\cdot$};
\end{tikzpicture}$$
Denote $L$ the full subcategory of $\Lambda$ formed by vertices of $\widetilde{Q}$ that can be connected to $x$ via a reduced walk which contains no subwalks of the form $\mu\alpha$, $\eta\alpha$, $\mu\gamma$, $\eta\gamma$ and their inverses. Then $L$ is a special biserial category. Let $N\in\mathrm{ind}L$ be an indecomposable direct summand of $M|_{L}$ with $N(x)\neq 0$. Since $L$ is a special biserial category whose quiver is simply connected, $\mathrm{supp}N$ is a finite line. Note that since $L$ is a convex full subcategory of $\Lambda$, by Lemma \ref{extend} we may view $N$ as a $\Lambda$-modules.

To show that $M=N$, it suffices to show that $N$ is a direct summand of $M$ (here we consider $N$ as a $\Lambda$-module). Suppose that $M|_{L}=N\oplus N'$. We will show that the following two assertions (1),(2) hold, and therefore $M=N\oplus M'$.
\begin{itemize}
\item[(1)] $N$ is a submodule of $M$.
\item[(2)] $N'\in\mathrm{mod}L$ can be extended to a submodule $M'$ of $M$ with
\begin{equation*}
M'(z)=\begin{cases}
N'(z), \text{ if } z\in L; \\
M(z), \text{ if } z\notin L.
\end{cases}
\end{equation*}
\end{itemize}

If $\mathrm{supp}N$ contains no vertex of the form
$$
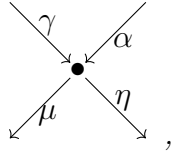
\begin{tikzpicture}
\draw[->] (0.9,0.9) -- (0.1,0.1);
\draw[->] (-0.9,0.9) -- (-0.1,0.1);
\draw[->] (-0.1,-0.1) -- (-0.9,-0.9);
\draw[->] (0.1,-0.1) -- (0.9,-0.9);
\node at(0,0) {$\bullet$};
\node at(0.6,0.4) {$\alpha$};
\node at(0.6,-0.4) {$\eta$};
\node at(-0.4,0.6) {$\gamma$};
\node at(-0.4,-0.6) {$\mu$};
\node at(1.2,-1) {,};
\end{tikzpicture}$$
then it is straightforward to show that the assertions (1),(2) hold. Otherwise, for each vertex $y$ in $\mathrm{supp}N$ of the form
$$
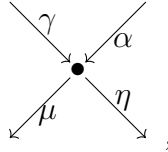
\begin{tikzpicture}
\draw[->] (0.9,0.9) -- (0.1,0.1);
\draw[->] (-0.9,0.9) -- (-0.1,0.1);
\draw[->] (-0.1,-0.1) -- (-0.9,-0.9);
\draw[->] (0.1,-0.1) -- (0.9,-0.9);
\node at(0,0) {$\bullet$};
\node at(0.6,0.4) {$\alpha$};
\node at(0.6,-0.4) {$\eta$};
\node at(-0.4,0.6) {$\gamma$};
\node at(-0.4,-0.6) {$\mu$};
\node at(1.2,-1) {,};
\end{tikzpicture}$$
consider the full subcategory $E$ of $\Lambda$:
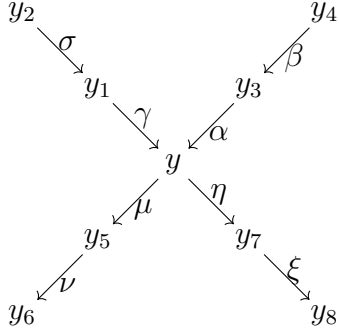
$$\begin{tikzpicture}
\draw[->] (0.8,0.8) -- (0.2,0.2);
\draw[->] (-0.8,0.8) -- (-0.2,0.2);
\draw[->] (-0.2,-0.2) -- (-0.8,-0.8);
\draw[->] (0.2,-0.2) -- (0.8,-0.8);
\draw[->] (1.8,1.8) -- (1.2,1.2);
\draw[->] (1.2,-1.2) -- (1.8,-1.8);
\draw[->] (-1.8,1.8) -- (-1.2,1.2);
\draw[->] (-1.2,-1.2) -- (-1.8,-1.8);
\node at(0,0) {$y$};
\node at(-1,1) {$y_1$};
\node at(-2,2) {$y_2$};
\node at(1,1) {$y_3$};
\node at(2,2) {$y_4$};
\node at(-1,-1) {$y_5$};
\node at(-2,-2) {$y_6$};
\node at(1,-1) {$y_7$};
\node at(2,-2) {$y_8$};
\node at(0.6,0.4) {$\alpha$};
\node at(0.6,-0.4) {$\eta$};
\node at(-0.4,0.6) {$\gamma$};
\node at(-0.4,-0.6) {$\mu$};
\node at(1.6,1.4) {$\beta$};
\node at(1.6,-1.4) {$\xi$};
\node at(-1.4,1.6) {$\sigma$};
\node at(-1.4,-1.6) {$\nu$};
\end{tikzpicture}$$
Let $\Lambda_1$ (resp. $\Lambda_2$, $\Lambda_3$, $\Lambda_4$) be the full subcategory of $\Lambda$ formed by vertices that can be connected to $y$ via a reduced walk which contains the subwalk $\gamma:y_1\rightarrow y$ (resp. $\alpha:y_3\rightarrow y$, $\mu^{-1}:y_5\rightarrow y$, $\eta^{-1}:y_7\rightarrow y$). It follows that $\Lambda=\Lambda_1\sqcup \Lambda_2\sqcup \Lambda_3\sqcup \Lambda_4\sqcup\{y\}$. Let $w$ be the reduced walk of $\widetilde{Q}$ from $x$ to $y$. Since $x,y\in\mathrm{supp}N$ and since $\mathrm{supp}N$ is a line, each point on $w$ belongs to $\mathrm{supp}N$. It follows that if $x\in \Lambda_i$, then $y_{2i-1},y_{2i}\in\mathrm{supp}N$. Moreover, if $x\in \Lambda_1\sqcup \Lambda_2$ (resp. $x\in \Lambda_3\sqcup \Lambda_4$), then $L\cap E$ is the full subcategory of $E$ formed by objects $y$, $y_1$, $y_2$, $y_3$, $y_4$ (resp. $y$, $y_5$, $y_6$, $y_7$, $y_8$).

If $x\in \Lambda_3$, then the restriction of $N$ on the full subcategory
$$\begin{tikzpicture}
\node at(0,0) {$y$};
\node at(1,0) {$y_5$};
\node at(2,0) {$y_6$};
\node at(0.5,0.2) {$\mu$};
\node at(1.5,0.2) {$\nu$};
\draw[->] (0.2,0) -- (0.8,0);
\draw[->] (1.2,0) -- (1.8,0);
\end{tikzpicture}$$
is isomorphic to
$$\begin{tikzpicture}
\node at(0,0) {$k$};
\node at(1,0) {$k$};
\node at(2,0) {$k$};
\node at(0.5,0.2) {$1$};
\node at(1.5,0.2) {$1$};
\draw[->] (0.8,0) -- (0.2,0);
\draw[->] (1.8,0) -- (1.2,0);
\end{tikzpicture}.$$
Since $\nu\mu\gamma=0$ and $\nu\mu\alpha=0$ in $\Lambda$, the restrictions of $M(\gamma):M(y)\rightarrow M(y_1)$ and $M(\alpha):M(y)\rightarrow M(y_3)$ on $N(y)$ are $0$. Similarly, if $x\in \Lambda_4$, then the restrictions of $M(\gamma):M(y)\rightarrow M(y_1)$ and $M(\alpha):M(y)\rightarrow M(y_3)$ on $N(y)$ are also $0$. Therefore $N$ is a submodule of $M$ and the assertion (1) holds.

If $x\in \Lambda_1$, then the restriction of $N$ on the full subcategory
$$\begin{tikzpicture}
\node at(0,0) {$y_2$};
\node at(1,0) {$y_1$};
\node at(2,0) {$y$};
\node at(0.5,0.2) {$\sigma$};
\node at(1.5,0.2) {$\gamma$};
\draw[->] (0.2,0) -- (0.8,0);
\draw[->] (1.2,0) -- (1.8,0);
\end{tikzpicture}$$
is isomorphic to
$$\begin{tikzpicture}
\node at(0,0) {$k$};
\node at(1,0) {$k$};
\node at(2,0) {$k$};
\node at(0.5,0.2) {$1$};
\node at(1.5,0.2) {$1$};
\draw[->] (0.8,0) -- (0.2,0);
\draw[->] (1.8,0) -- (1.2,0);
\end{tikzpicture}.$$
Since $\mu\gamma\sigma=0$ and $\eta\gamma\sigma=0$ in $\Lambda$, the images of $M(\mu):M(y_5)\rightarrow M(y)$ and $M(\eta):M(y_7)\rightarrow M(y)$ are contained in $N'(y)$. Similarly, if $x\in \Lambda_2$, then the images of $M(\mu):M(y_5)\rightarrow M(y)$ and $M(\eta):M(y_7)\rightarrow M(y)$ are also contained in $N'(y)$. Therefore $N'\in\mathrm{mod}L$ can be extended to a submodule $M'$ of $M$ with
\begin{equation*}
M'(z)=\begin{cases}
N'(z), \text{ if } z\in L; \\
M(z), \text{ if } z\notin L,
\end{cases}
\end{equation*}
and the assertion (2) holds.

The above discussions show that the support of any indecomposable finite dimensional $\Lambda$-module is either a finite line or is contained in an extended Dynkin quiver of type $\widetilde{D}_4$; so $\Lambda$ satisfies the assumptions of Proposition \ref{tame}, and $A=\Lambda/G$ is tame.
\end{Ex1}

\subsection{Further remarks on tameness}
\

By the end of this section, we mention the following conjecture, which was kindly informed to us by Piotr Dowbor: If $G$ acts freely on the set $[\mathrm{ind}\Lambda]$ and if $\Lambda$ is tame, then $\Lambda/G$ is also tame.

Note that if $G$ does not act freely on the set $[\mathrm{ind}\Lambda]$, then $\Lambda/G$ can be wild:

\begin{Ex1} {\rm(see \cite{GP1993})} \label{counter-example}
Suppose that char$(k)=2$. Let $\Lambda=kQ/I$ be the locally bounded $k$-category given by the quiver $Q$
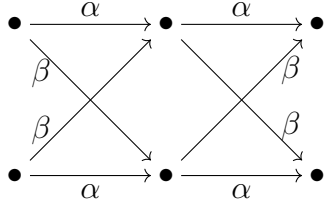
$$\begin{tikzpicture}
\draw[->] (0.2,2) -- (1.8,2);
\draw[->] (2.2,2) -- (3.8,2);
\draw[->] (0.2,0) -- (1.8,0);
\draw[->] (2.2,0) -- (3.8,0);
\draw[->] (0.2,0.2) -- (1.8,1.8);
\draw[->] (0.2,1.8) -- (1.8,0.2);
\draw[->] (2.2,0.2) -- (3.8,1.8);
\draw[->] (2.2,1.8) -- (3.8,0.2);
\node at(1,2.2) {$\alpha$};
\node at(3,2.2) {$\alpha$};
\node at(1,-0.2) {$\alpha$};
\node at(3,-0.2) {$\alpha$};
\node at(0.35,0.6) {$\beta$};
\node at(0.35,1.35) {$\beta$};
\node at(3.65,0.65) {$\beta$};
\node at(3.65,1.4) {$\beta$};
\node at(0,0) {$\bullet$};
\node at(2,0) {$\bullet$};
\node at(4,0) {$\bullet$};
\node at(0,2) {$\bullet$};
\node at(2,2) {$\bullet$};
\node at(4,2) {$\bullet$};
\end{tikzpicture}.$$
with $I$ generated by all relations $\alpha^2-\beta^2$ and $\alpha\beta-\beta\alpha$; and let $G=\langle g\rangle$ be the group of automorphisms of $\Lambda$ generated by $g$, where $g$ is induced by the unique quiver automorphism of $Q$ of order $2$. It follows that $\Lambda/G$ is given by the quiver
$$
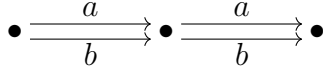
\begin{tikzpicture}
\draw[->] (0.2,0.1) -- (1.8,0.1);
\draw[->] (0.2,-0.1) -- (1.8,-0.1);
\draw[->] (2.2,0.1) -- (3.8,0.1);
\draw[->] (2.2,-0.1) -- (3.8,-0.1);
\node at(1,0.3) {$a$};
\node at(1,-0.3) {$b$};
\node at(3,0.3) {$a$};
\node at(3,-0.3) {$b$};
\node at(0,0) {$\bullet$};
\node at(2,0) {$\bullet$};
\node at(4,0) {$\bullet$};
\end{tikzpicture}.$$
with relations $a^2-b^2$ and $ab-ba$.

Note that the algebra corresponding to $\Lambda$ is skew-gentle and always tame, but $\Lambda/G$ is wild when char$(k)=2$, as can be shown using Proposition \ref{Sin} and Theorem \ref{Sincere-wildness} (see \cite[Claim 3.2]{GP1993}).
\end{Ex1}

\section{Covering theory for general $k$-categories}
	
In \cite{Ga1981}, Gabriel proved that a push-down functor of a Galois covering between two locally bounded categories induces a Galois covering between the associated Auslander-Reiten quivers. This covering technique was extended by Asashiba in \cite{Asa97} to study the relationship between the associated derived categories. However, this technique can only be used on skeletal categories, that is, categories in which isomorphic objects are equal. To overcome this limitation, Asashiba defined in \cite{Asa2011} the precovering between (general $k$-)categories and proved that the canonical functor $F: \Lambda\ra \Lambda/G$ induces a push-down functor $F_\lambda\colon \cal{K}^b(\proj \Lambda)\ra \cal{K}^b(\proj (\Lambda/G))$ which is a precovering.

In \cite{BL}, Bautista and Liu defined the Galois $G$-covering between categories and showed that the Galois $G$-covering functor preserves the Auslander-Reiten theory for Krull-Schmidt categories and can induces the push-down functor between the associated derived categories. This section is a brief introduction on the work of Bautista and Liu \cite{BL} without proofs. We also present some concrete examples related to graded algebras and graded module categories.
	
	\subsection{$G$-precoverings and Galois $G$-coverings}\

		In this section we introduce the definition of precoverings and Galois coverings for categories, and the notion of graded adjoint pairs. Throughout this section, when we say a group $G$ acts on a category $\cal{A}$, we mean that $G$ is a group of automorphisms of $\cal{A}$.
		
		\begin{Def}
			We say that $\cal{A}$ has direct sums if any family of objects $\{X_i\}_{ i\in I}$ has a direct sum, where $I$ is an arbitrary index set. In this case, we denote the canonical injections by $q_j\colon X_j\ra \bigoplus_{i\in I}X_i$, $j\in I$. Then there exist unique morphisms $p_j\colon \bigoplus_{i\in I}X_i\ra X_j$ in $\cal{A}$ such that \[p_i q_j =
			\begin{cases*}
				1_{X_i}, & if $i=j$; \\
				0, & if $i\neq j$,
			\end{cases*}
			\]
			for all $i,j\in I$. An object $M\in \cal{A}$ is called essential in $\bigoplus_{i\in I}X_i$ if for any morphism $f\colon M\ra \bigoplus_{i\in I}X_i$, that $f=0$ if and only if $p_jf=0$ for all $j\in I$. If every object in $\cal{A}$ is essential in $\bigoplus_{i\in I}X_i$, then $\bigoplus_{i\in I}X_i$ is called an essential direct sum. If each family of objects $\{X_i\}_{i\in I}$ has an essential direct sum, then we say that $\cal{A}$ has essential direct sums.
		\end{Def}

		\begin{Rem1}
			\textnormal{(1)} Finite direct sums in a category are always essential. Let $\Lambda$ be a $k$-algebra. The category $\Mod \Lambda$ has essential direct sums, see \cite[Lemma 1.1]{BL}.

			\textnormal{(2)} Let $\cal{A}$ be an abelian category. If $\cal{A}$ has essential direct sums, then the derived category $\cal{D}(\cal{A})$ of $\cal{A}$ has direct sums, see \cite[Theorem 1.8]{BL}. 			
		\end{Rem1}

		\begin{Def}{\rm(\cite[Definition 2.1]{BL})}
			Let $\cal{A}$ be a category with $G$ a group acting on $\cal{A}$. The $G$-action on $\cal{A}$ is called
			\begin{enumerate}[label=\textnormal{(\arabic*)}]
				\item free, if $g\cdot X\not\cong X$ for any indecomposable $X\in\cal{A}$ and any non-identity $g\in G$;
				\item locally bounded, if for any indecomposable $X,Y\in\cal{A}$, $\cal{A}(X,g\cdot Y)=0$ for all but finitely many $g\in G$;
				\item admissible, if it is both free and locally bounded.
			\end{enumerate}
		\end{Def}

		\begin{Def}{\rm(\cite[Definition 2.3]{BL})}
			Let $\cal{A},\cal{B}$ be categories with $G$ a group acting on $\cal{A}$.
			A functor $F\colon \cal{A}\ra\cal{B}$ is called $G$-stable if there exist functorial isomorphisms $\delta_g\colon F\circ g\ra F$, $g\in G$, such that $\delta_{h,X}\circ\delta_{g,h\cdot X}=\delta_{gh,X}$ for all $g,h\in G$ and $X\in\cal{A}$. In this case, we call $\delta=(\delta_g)_{g\in G}$ a $G$-stabilizer for $F$.
		\end{Def}

		\begin{Def}{\rm(\cite[Definition 2.5]{BL}, see also \cite[Definition 2.1]{Asa2011})}
			Let $\cal{A},\cal{B}$ be categories with $G$ a group acting on $\cal{A}$. A functor $F\colon \cal{A}\ra\cal{B}$ is called a $G$-precovering if $F$ has a $G$-stabilizer $\delta$ such that, for any $X,Y\in\cal{A}$, the following two maps are isomorphisms:
			\[F_{X,Y}:\bigoplus_{g\in G}\mathcal{A}(X,g\cdot Y)\to\mathcal{B}\left(F(X),F(Y)\right):(u_{g})_{g\in G}\mapsto\sum_{g\in G}\delta_{g,Y}\circ F(u_{g}).\]
			\[F^{X,Y}:\bigoplus_{g\in G}\mathcal{A}(g\cdot X,Y)\to\mathcal{B}\left(F(X),F(Y)\right):(v_g)_{g\in G}\mapsto\sum_{g\in G}F(v_g)\circ\delta_{g,X}^{-1}.\]
		\end{Def}

		\begin{Rem1}
			In the above definition, it is sufficient to require that all $F_{X,Y}$ be isomorphisms, or all $F^{X,Y}$ be isomorphisms, see \cite[Proposition 1.6]{Asa2011}.
		\end{Rem1}

		\begin{Def}{\rm(\cite[Definition 2.8]{BL})}\label{BL, Def 2.8}
			Let $\cal{A},\cal{B}$ be categories with $G$ a group acting admissibly on $\cal{A}$. A $G$-precovering $F\colon \cal{A}\ra\cal{B}$ is called a Galois $G$-covering if the following conditions hold:
			\begin{enumerate}[label=\textnormal{(\arabic*)}]
				\item The functor $F$ is almost dense, that is, each indecomposable object in $\cal{B}$ is isomorphic to an object lying in the image of $F$.
				\item If $X\in\cal{A}$ is indecomposable, then $F(X)$ is indecomposable.
				\item If $X,Y\in\cal{A}$ are indecomposable with $F(X)\cong F(Y)$, then there exists some $g\in G$ such that $Y=g\cdot X$.
			\end{enumerate}
		\end{Def}

		\begin{Rem1}
			\textnormal{(1)} In case $\cal{A},\cal{B}$ are Krull-Schmidt, a Galois $G$-covering $F\colon \cal{A}\ra \cal{B}$ is a dense functor, and consequently, $F$ is an equivalence if and only if $G$ is trivial.

			\textnormal{(2)} If $\cal{A},\cal{B}$ are locally bounded categories, then a Galois covering defined by Definition \ref{galois-covering-functor} is a $G$-invariant Galois $G$-covering, that is, the $G$-stabilizer $\delta=(\delta_g)_{g\in G}$ satisfies $\delta_g=1_F$ for all $g\in G$.
		\end{Rem1}
		
		To help the reader understand these concepts, we present a concrete example of a $G$-precovering and Galois $G$-covering. In our discussion, we focus on graded module categories of graded (finite dimensional) algebras (specifically $\bb{Z}$-graded algebras for simplicity). We prove that:
		\begin{enumerate}[label=\textnormal{(\arabic*)}]
			\item The forgetful functor $F$ on the (finite dimensional) graded module category $\gr \Lambda$ forms a $\bb{Z}$-precovering;
			\item This functor induces a Galois $\bb{Z}$-covering from the graded module category $\gr \Lambda$ to a specified subcategory $\mod_\infty\Lambda$ of the (finite dimensional) module category $\mod \Lambda$.
		\end{enumerate}

		\begin{Ex1}\label{Kronecker 1}
			Let $\Lambda=\bigoplus_{i \in \mathbb{Z}} \Lambda_i$ be a $\mathbb{Z}$-graded finite dimensional algebra. The graded module category $\gr \Lambda$ is defined as follows:

			\begin{itemize}
				\item Objects: finite dimensional $\mathbb{Z}$-graded (right) $\Lambda$-modules $X = \bigoplus_{i \in \mathbb{Z}} X_i$, where the grading is compatible with the $\Lambda$-action, i.e., $X_i \cdot \Lambda_j \subseteq X_{i+j}$ for all $i, j \in \mathbb{Z}$.
				\item Morphisms: $\Lambda$-linear maps $f \colon X \to Y$  such that $f(X_i) \subseteq Y_i$ for all $i \in \mathbb{Z}$.
			\end{itemize}
			There are several important structures associated with $\gr\Lambda$:

			\begin{enumerate}[label=\textnormal{(\arabic*)}]
				\item The forgetful functor $F: \gr \Lambda \to \mod \Lambda$ sends a graded module $X = \bigoplus_{i \in \mathbb{Z}} X_i$ to its underlying ungraded $\Lambda$-module, ignoring the grading. On morphisms, $F$ acts by retaining the module homomorphism while disregarding the degree condition.
				\item For each $n \in \mathbb{Z}$, the grading shift functor $\sigma(n) \colon \gr \Lambda \to \gr \Lambda$ acts on objects by: $$(\sigma(n)(X))_i=X_{i-n}.$$ On morphisms, $(\sigma(n)(f))_i= f_{i-n}$, since $f$ preserves degrees and the shift is applied uniformly.
				\item The shift functors induce a $\mathbb{Z}$-action on $\gr \Lambda$ via: $$n\cdot X\coloneqq \sigma(n)(X)$$ and for morphisms, $n \cdot f \coloneqq \sigma(n)(f)$. This action is natural in the sense that $(n + m) \cdot X = n \cdot (m \cdot X)$ and $0 \cdot X = X$.
			\end{enumerate}
			We donote the full subcategory of $\mod \Lambda$ consisting of the gradable modules by $$\mod_\infty\Lambda= \{F (X)\mid X\in \gr \Lambda\}.$$
			
			Let $\delta=(\delta_n)_{n\in \bb{Z}}$ be the family of functorial isomorphisms $\delta_n\colon F\circ n\ra F$, satisfying that $\delta_{n,X}$ is the identity map of $\Hom_\Lambda(F\circ \sigma(n)(X),F (X))=\End_\Lambda(F (X))$, where $n\in \bb{Z}$ and $X\in\gr\Lambda$. Then the forgetful functor $F$ is a ($\bb{Z}$-invariant) $\bb{Z}$-precovering with the $\bb{Z}$-stabilizer $\delta$, and it induces a Galois $\bb{Z}$-covering $F\colon\gr\Lambda\ra\mod_\infty\Lambda,X\mapsto F(X)$.
		\end{Ex1}

		\subsection{Auslander-Reiten Theory via Galois $G$-covering}
		\

		In this section, we show that the Galois $G$-covering functor preserves the Auslander-Reiten theory for Krull-Schmidt categories.

		\begin{Def}{\rm(\cite[Section 1]{Liu})}
			Let $\cal{A}$ be a Krull-Schmidt category. Let $\eta\colon X\xra{f} Y\xra{g} Z$ be a sequence of morphisms in $\cal{A}$. We call $f$ a pseudo-kernel of $g$ if
			\[\Hom_\cal{A} (M,X)\xrightarrow{f_*} \Hom_\cal{A} (M,Y)\xrightarrow{g_*} \Hom_\cal{A} (M,Z)\]
			is an exact sequence for all $M\in \cal{A}$. Dually, we call $g$ a pseudo-cokernel of $f$ if
			\[\Hom_\cal{A} (Z,N)\xrightarrow{g^*} \Hom_\cal{A} (Y,N)\xrightarrow{f^*} \Hom_\cal{A} (X,N)\]
			is an exact sequence for all $N\in \cal{A}$. And we call the sequence $\eta$ short pseudo-exact if $f$ is a pseudo-kernel of $g$ and $g$ is a pseudo-cokernel of $f$.
		\end{Def}

		Note that a pseudo-kernel (resp. a pseudo-cokernel) is a kernel (resp. a cokernel) if it is a monomorphism (resp. an epimorphism).

\begin{Def}{\rm(\cite[Section 3]{BL})}
			Let $\cal{A}$ be a Krull-Schmidt category. A short pseudo-exact sequence $\eta\colon X\xra{f} Y\xra{g} Z$ in $\cal{A}$ is called an almost split sequence if $Y$ is non-zero, $f$ is a minimal left almost split morphism and $g$ is a minimal right almost split morphism.
		\end{Def}

Note that in an abelian category, an almost split sequence $\eta\colon X\xra{f} Y\xra{g} Z$ in $\cal{A}$ in the sense of above definition must be a short exact sequence (see \cite[Proposition 1.5]{Liu}).

\begin{Lem}{\rm(\cite[Lemma 3.6]{BL})}
			Let $\cal{A},\cal{B}$ be Krull-Schmidt categories with $G$ a group acting admissibly on $\cal{A}$, and let $F\colon \cal{A}\ra\cal{B}$ be a Galois $G$-covering. If $\eta$ is a short sequence in $\cal{A}$, then $\eta$ is pseudo-exact if and only if $F(\eta)$ is pseudo-exact.
		\end{Lem}

		\begin{Rem1}
			In particular, if $\cal{A},\cal{B}$ are Krull-Schmidt abelian categories, then the Galois $G$-covering functor $F\colon \cal{A}\ra\cal{B}$ is an exact functor.
		\end{Rem1}

		\begin{Thm}{\rm(\cite[Theorem 3.7]{BL})}
			Let $\cal{A},\cal{B}$ be Krull-Schmidt categories with $G$ a group acting admissibly on $\cal{A}$, and let $F\colon \cal{A}\ra\cal{B}$ be a Galois $G$-covering. Then
			\begin{enumerate}[label=\textnormal{(\arabic*)}]
				\item A short sequence $\eta$ in $\cal{A}$ is almost split if and only if $F(\eta)$ is almost split.
				\item An object $X\in\cal{A}$ is the starting term or the ending term of an almost split sequence if and only if $F(X)$ is the starting term or the ending term of an almost split sequence, respectively.
			\end{enumerate}
		\end{Thm}

		Now we introduce the notion of Galois $G$-coverings for quivers and translation quivers, which is similar but with some additional conditions compared to the one for quivers with relations. And the following theorem shows that the Galois $G$-covering functor between Hom-finite Krull-Schmidt categories induces a Galois $G$-covering between the associated Auslander-Reiten quivers.

		\begin{Def}{\rm(\cite[Definition 4.1]{BL})}
			Let $Q,Q'$ be quivers with a group $G$ acting freely on $Q$. A quiver morphism $\phi=(\phi_0,\phi_1)\colon Q\ra Q'$ is called a Galois $G$-covering if the following conditions hold:
			\begin{enumerate}[label=\textnormal{(\arabic*)}]
				\item The associated vertex map $\phi_0$ is surjective.
				\item If $g\in G$, then $\phi\circ g=\phi$.
				\item If $x,y\in Q_0$ with $\phi_0(x)=\phi_0(y)$, then $y=g\cdot x$ for some $g\in G$.
				\item If $x\in Q_0$, then $\phi_1$ induces two bijections $x^+\ra \phi_0(x)^+$ and $x^-\ra \phi_0(x)^-$.
			\end{enumerate}
		\end{Def}

Note that a Galois $G$-covering $Q\ra Q'$ is equivalent to a Galois covering of $(Q,0)\rightarrow (Q',0)$ of quivers with empty relations with group $G$ (cf. Definition \ref{galois-covering}).

		\begin{Def}{\rm(\cite[Definition 4.5]{BL})}
			A morphism of translation quivers $\phi\colon (\Delta,\tau)\ra (\Omega,\rho)$ is a quiver morphism $\phi\colon \Delta\ra \Omega$ such that for any non-projective vertex $x\in \Delta_0$, the image $\phi(x)$ is also non-projective and $\rho(\phi(x))=\phi(\tau(x))$.
		\end{Def}

		\begin{Def}{\rm(\cite[Definition 4.6]{BL})}
			Let $(\Delta,\tau)$ be a translation quiver with a free action of a group $G$. A morphism of translation quivers $\phi\colon (\Delta,\tau)\ra (\Omega,\rho)$ is called a Galois $G$-covering if $\phi\colon \Delta\ra \Omega$ is a Galois $G$-covering with the following property: for any projective vertex $x\in \Delta_0$, the image $\phi(x)$ is also projective.
		\end{Def}

		\begin{Thm}{\rm(\cite[Theorem 4.7]{BL})}\label{BL, Theorem 4.7}
			Let $\cal{A},\cal{B}$ be Hom-finite Krull-Schmidt categories with $G$ a group acting admissibly on $\cal{A}$, and let $F\colon \cal{A}\ra\cal{B}$ be a Galois $G$-covering. Then
			\begin{enumerate}[label=\textnormal{(\arabic*)}]
				\item The functor $F$ induces a Galois $G$-covering $\phi\colon \Gamma_\cal{A}\ra \Gamma_\cal{B}$ of translation quivers between the associated Auslander-Reiten quivers.
				\item If $\Gamma$ is a connected component of $\Gamma_\cal{A}$, then $\phi(\Gamma)$ is a connected component of $\Gamma_\cal{B}$, and $\phi$ restricts to a Galois $G_\Gamma$-covering $\phi_\Gamma\colon \Gamma\ra \phi(\Gamma)$, where $G_\Gamma\coloneqq \{g\in G\mid g\cdot \Gamma=\Gamma\}$ is the induced subgroup of $G$.
			\end{enumerate}
		\end{Thm}

		In Example \ref{Kronecker 1}, we demonstrate that for any $\mathbb{Z}$-graded algebra $\Lambda$, there exists a canonical Galois $\mathbb{Z}$-covering
		$$F\colon\gr\Lambda\ra\mod_\infty\Lambda$$ where $F$ is induced by the forgetful functor. By Theorem \ref{BL, Theorem 4.7}, this covering $F$ further induces a Galois $\mathbb{Z}$-covering between their corresponding Auslander-Reiten quivers $$\phi\colon \Gamma_{\gr\Lambda}\ra \Gamma_{\mod_\infty\Lambda}.$$ We illustrate this phenomenon explicitly in the following example.

		\begin{Ex1}\label{Kronecker 2}
			Let $Q$ be the quiver
			\begin{tikzcd}
				1 & 2
				\arrow["\alpha", shift left, from=1-1, to=1-2]
				\arrow["\beta"', shift right, from=1-1, to=1-2]
			\end{tikzcd}
			and let $\Lambda$ be the $\bb{Z}$-graded algebra \[\Lambda=kQ=\bigoplus_{i=0}^2 \Lambda_i\] with $\Lambda_0=ke_1\oplus ke_2$, $\Lambda_1=k\alpha$ and $\Lambda_2=k\beta$.
			Then there are four components in the Auslander-Reiten quiver $\Gamma_{\mod_\infty\Lambda}$ of $\mod_\infty\Lambda$:
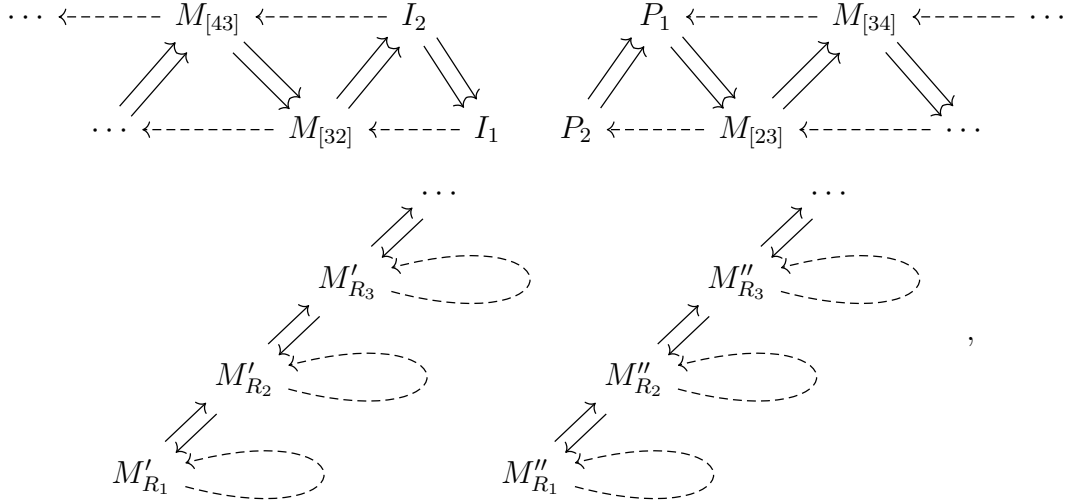
				\[\begin{tikzcd}[column sep=tiny]
					\cdots&& M_{[43]} && I_2 \\
					&\cdots && M_{[32]} && {I_1}
					\arrow[shift left, from=2-2, to=1-3]
					\arrow[shift right, from=2-2, to=1-3]
					\arrow[shift right, from=1-3, to=2-4]
					\arrow[shift left, from=1-3, to=2-4]
					\arrow[shift left, from=2-4, to=1-5]
					\arrow[shift right, from=2-4, to=1-5]
					\arrow[shift right, from=1-5, to=2-6]
					\arrow[shift left, from=1-5, to=2-6]
					\arrow[dashed, from=2-4, to=2-2]
					\arrow[dashed, from=2-6, to=2-4]
					\arrow[dashed, from=1-5, to=1-3]
					\arrow[dashed, from=1-3, to=1-1]
				\end{tikzcd}\quad
				\begin{tikzcd}[column sep=tiny]
					& {P_1} && M_{[34]} && \cdots\\
					{P_2} && M_{[23]} && \cdots
					\arrow[shift left, from=2-1, to=1-2]
					\arrow[shift right, from=2-1, to=1-2]
					\arrow[shift right, from=1-2, to=2-3]
					\arrow[shift left, from=1-2, to=2-3]
					\arrow[shift left, from=2-3, to=1-4]
					\arrow[shift right, from=2-3, to=1-4]
					\arrow[shift right, from=1-4, to=2-5]
					\arrow[shift left, from=1-4, to=2-5]
					\arrow[dashed, from=2-3, to=2-1]
					\arrow[dashed, from=2-5, to=2-3]
					\arrow[dashed, from=1-4, to=1-2]
					\arrow[dashed, from=1-6, to=1-4]
				\end{tikzcd}\]
				\[\begin{tikzcd}[row sep=15, column sep=6]
					&&& \cdots &&&& \cdots \\
					&& {M_{R_3}'} &&&& {M_{R_3}''} \\
					& {M_{R_2}'} &&&& {M_{R_2}''} \\
					{M_{R_1}'} &&&& {M_{R_1}''}
					\arrow[shift left, from=4-1, to=3-2]
					\arrow[shift left, from=3-2, to=4-1]
					\arrow[shift left, from=3-2, to=2-3]
					\arrow[shift left, from=2-3, to=3-2]
					\arrow[shift left, from=2-3, to=1-4]
					\arrow[shift left, from=1-4, to=2-3]
					\arrow[shift left, from=4-5, to=3-6]
					\arrow[shift left, from=3-6, to=4-5]
					\arrow[shift left, from=3-6, to=2-7]
					\arrow[shift left, from=2-7, to=3-6]
					\arrow[shift left, from=2-7, to=1-8]
					\arrow[shift left, from=1-8, to=2-7]
					\arrow[from=4-1, to=4-1, dashed, distance=6em, in=15, out=345]
					\arrow[from=3-2, to=3-2, dashed, distance=6em, in=15, out=345]
					\arrow[from=2-3, to=2-3, dashed, distance=6em, in=15, out=345]
					\arrow[from=4-5, to=4-5, dashed, distance=6em, in=15, out=345]
					\arrow[from=3-6, to=3-6, dashed, distance=6em, in=15, out=345]
					\arrow[from=2-7, to=2-7, dashed, distance=6em, in=15, out=345]
				\end{tikzcd},\]
					where the representations of $M_{R_1}'$ and $M_{R_1}''$ are given by
				$\begin{tikzcd}
					k & k
					\arrow["1", shift left, from=1-1, to=1-2]
					\arrow["0"', shift right, from=1-1, to=1-2]
				\end{tikzcd}$
				and
				$\begin{tikzcd}
					k & k
					\arrow["0", shift left, from=1-1, to=1-2]
					\arrow["1"', shift right, from=1-1, to=1-2]
				\end{tikzcd}$.

				Indeed, for all $\lambda\in k\setminus\{0\}$, the components of the Auslander-Reiten quiver of $\mod \Lambda$ containing the module
				$\begin{tikzcd}
					k & k
					\arrow["1", shift left, from=1-1, to=1-2]
					\arrow["\lambda"', shift right, from=1-1, to=1-2]
				\end{tikzcd}$
				are ungradable. Otherwise, by \cite[Theorem 3.4]{GG}, each module in such a component is gradable. However, consider the module
				$\begin{tikzcd}
					k & k
					\arrow["1", shift left, from=1-1, to=1-2]
					\arrow["\lambda"', shift right, from=1-1, to=1-2]
				\end{tikzcd}$. If we assign degree $d$ to the vector space $k$ at vertex $1$, then the vector space $k$ at vertex $2$ would need to simultaneously have degrees $d+1$ (for the arrow $\alpha$) and $d+2$ (for the arrow $\beta$), which contradicts the definition of a graded module.

			At the same time, the Auslander-Reiten quiver $\Gamma_{\gr\Lambda}$ of $\gr\Lambda$ also has four components as follows.
			
				
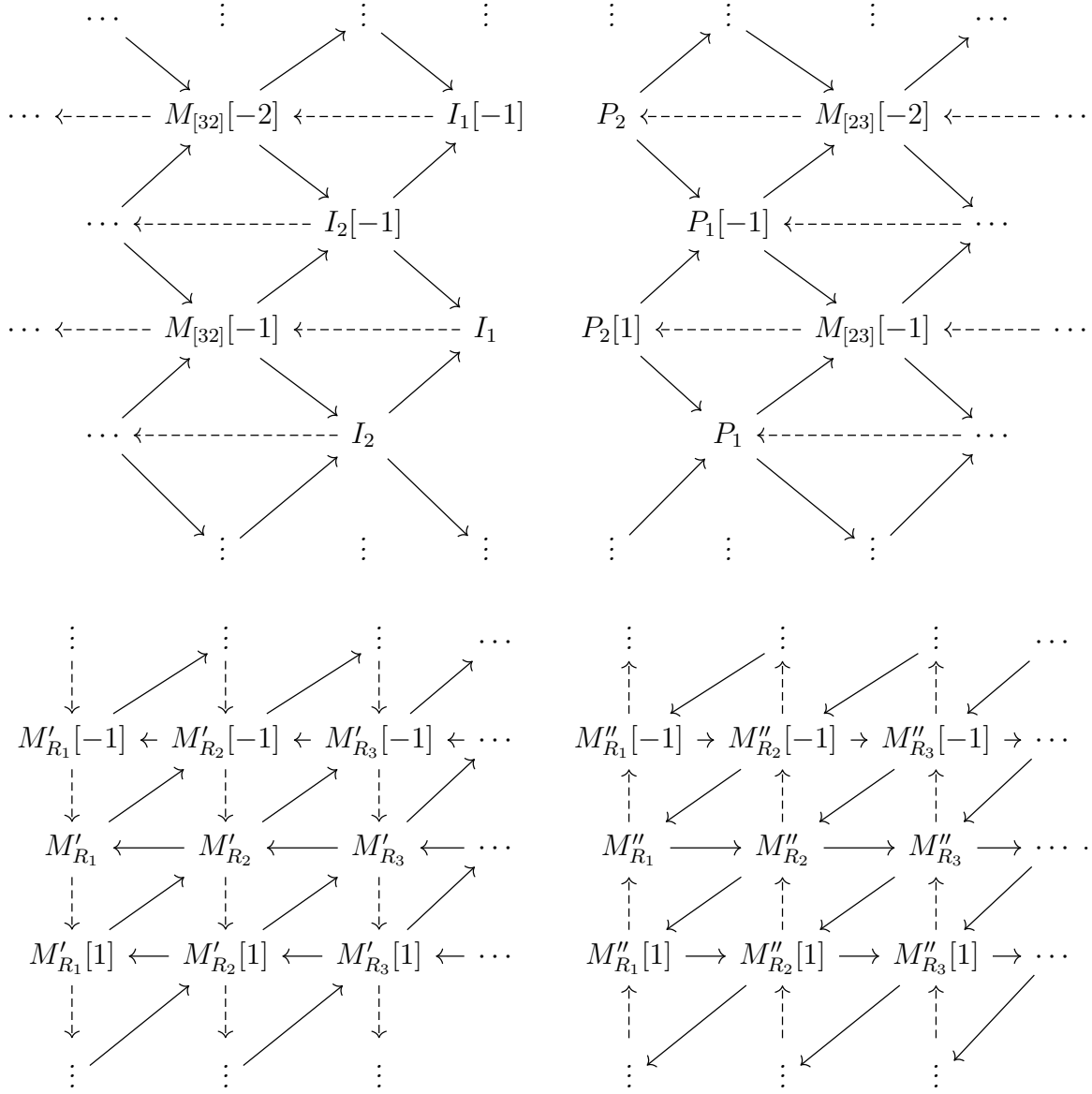
\[\begin{tikzcd}[column sep=tiny]
					& \cdots & \vdots & \vdots & \vdots \\
					\cdots && {M_{[32]}[-2]} && {I_1[-1]} \\
					& \cdots && {I_2[-1]} \\
					\cdots && {M_{[32]}[-1]} && {I_1} \\
					& \cdots && {I_2} \\
					&& \vdots & \vdots & \vdots
					\arrow[from=1-2, to=2-3]
					\arrow[from=1-4, to=2-5]
					\arrow[from=2-3, to=1-4]
					\arrow[dashed, from=2-3, to=2-1]
					\arrow[from=2-3, to=3-4]
					\arrow[dashed, from=2-5, to=2-3]
					\arrow[from=3-2, to=2-3]
					\arrow[from=3-2, to=4-3]
					\arrow[from=3-4, to=2-5]
					\arrow[dashed, from=3-4, to=3-2]
					\arrow[from=3-4, to=4-5]
					\arrow[from=4-3, to=3-4]
					\arrow[dashed, from=4-3, to=4-1]
					\arrow[from=4-3, to=5-4]
					\arrow[dashed, from=4-5, to=4-3]
					\arrow[from=5-2, to=4-3]
					\arrow[from=5-2, to=6-3]
					\arrow[from=5-4, to=4-5]
					\arrow[dashed, from=5-4, to=5-2]
					\arrow[from=5-4, to=6-5]
					\arrow[from=6-3, to=5-4]
				\end{tikzcd}\quad
				\begin{tikzcd}[column sep=tiny]
					\vdots & \vdots & \vdots & \cdots \\
					{P_2} && {M_{[23]}[-2]} && \cdots \\
					& {P_1[-1]} && \cdots \\
					{P_2[1]} && {M_{[23]}[-1]} && \cdots \\
					& {P_1} && \cdots \\
					\vdots & \vdots & \vdots
					\arrow[from=1-2, to=2-3]
					\arrow[from=2-1, to=1-2]
					\arrow[from=2-1, to=3-2]
					\arrow[from=2-3, to=1-4]
					\arrow[dashed, from=2-3, to=2-1]
					\arrow[from=2-3, to=3-4]
					\arrow[dashed, from=2-5, to=2-3]
					\arrow[from=3-2, to=2-3]
					\arrow[from=3-2, to=4-3]
					\arrow[dashed, from=3-4, to=3-2]
					\arrow[from=4-1, to=3-2]
					\arrow[from=4-1, to=5-2]
					\arrow[from=4-3, to=3-4]
					\arrow[dashed, from=4-3, to=4-1]
					\arrow[from=4-3, to=5-4]
					\arrow[dashed, from=4-5, to=4-3]
					\arrow[from=5-2, to=4-3]
					\arrow[from=5-2, to=6-3]
					\arrow[dashed, from=5-4, to=5-2]
					\arrow[from=6-1, to=5-2]
					\arrow[from=6-3, to=5-4]
				\end{tikzcd}\]

				\[\begin{tikzcd}[column sep=tiny]
					\vdots & \vdots & \vdots & \cdots && \vdots & \vdots & \vdots & \cdots \\
					{M'_{R_1}[-1]} & {M'_{R_2}[-1]} & {M'_{R_3}[-1]} & \cdots && {M''_{R_1}[-1]} & {M''_{R_2}[-1]} & {M''_{R_3}[-1]} & \cdots \\
					{M'_{R_1}} & {M'_{R_2}} & {M'_{R_3}} & \cdots && {M''_{R_1}} & {M''_{R_2}} & {M''_{R_3}} & \cdots \\
					{M'_{R_1}[1]} & {M'_{R_2}[1]} & {M'_{R_3}[1]} & \cdots && {M''_{R_1}[1]} & {M''_{R_2}[1]} & {M''_{R_3}[1]} & \cdots \\
					\vdots & \vdots & \vdots &&& \vdots & \vdots & \vdots
					\arrow[dashed, from=1-1, to=2-1]
					\arrow[dashed, from=1-2, to=2-2]
					\arrow[dashed, from=1-3, to=2-3]
					\arrow[from=1-7, to=2-6]
					\arrow[from=1-8, to=2-7]
					\arrow[from=1-9, to=2-8]
					\arrow[from=2-1, to=1-2]
					\arrow[dashed, from=2-1, to=3-1]
					\arrow[from=2-2, to=1-3]
					\arrow[from=2-2, to=2-1]
					\arrow[dashed, from=2-2, to=3-2]
					\arrow[from=2-3, to=1-4]
					\arrow[from=2-3, to=2-2]
					\arrow[dashed, from=2-3, to=3-3]
					\arrow[from=2-4, to=2-3]
					\arrow[dashed, from=2-6, to=1-6]
					\arrow[from=2-6, to=2-7]
					\arrow[dashed, from=2-7, to=1-7]
					\arrow[from=2-7, to=2-8]
					\arrow[from=2-7, to=3-6]
					\arrow[dashed, from=2-8, to=1-8]
					\arrow[from=2-8, to=2-9]
					\arrow[from=2-8, to=3-7]
					\arrow[from=2-9, to=3-8]
					\arrow[from=3-1, to=2-2]
					\arrow[dashed, from=3-1, to=4-1]
					\arrow[from=3-2, to=2-3]
					\arrow[from=3-2, to=3-1]
					\arrow[dashed, from=3-2, to=4-2]
					\arrow[from=3-3, to=2-4]
					\arrow[from=3-3, to=3-2]
					\arrow[dashed, from=3-3, to=4-3]
					\arrow[from=3-4, to=3-3]
					\arrow[dashed, from=3-6, to=2-6]
					\arrow[from=3-6, to=3-7]
					\arrow[dashed, from=3-7, to=2-7]
					\arrow[from=3-7, to=3-8]
					\arrow[from=3-7, to=4-6]
					\arrow[dashed, from=3-8, to=2-8]
					\arrow[from=3-8, to=3-9]
					\arrow[from=3-8, to=4-7]
					\arrow[from=3-9, to=4-8]
					\arrow[from=4-1, to=3-2]
					\arrow[dashed, from=4-1, to=5-1]
					\arrow[from=4-2, to=3-3]
					\arrow[from=4-2, to=4-1]
					\arrow[dashed, from=4-2, to=5-2]
					\arrow[from=4-3, to=3-4]
					\arrow[from=4-3, to=4-2]
					\arrow[dashed, from=4-3, to=5-3]
					\arrow[from=4-4, to=4-3]
					\arrow[dashed, from=4-6, to=3-6]
					\arrow[from=4-6, to=4-7]
					\arrow[dashed, from=4-7, to=3-7]
					\arrow[from=4-7, to=4-8]
					\arrow[from=4-7, to=5-6]
					\arrow[dashed, from=4-8, to=3-8]
					\arrow[from=4-8, to=4-9]
					\arrow[from=4-8, to=5-7]
					\arrow[from=4-9, to=5-8]
					\arrow[from=5-1, to=4-2]
					\arrow[from=5-2, to=4-3]
					\arrow[dashed, from=5-6, to=4-6]
					\arrow[dashed, from=5-7, to=4-7]
					\arrow[dashed, from=5-8, to=4-8]
				\end{tikzcd}. \]
			By Theorem \ref{BL, Theorem 4.7}, the Galois $\bb{Z}$-covering $F\colon \gr\Lambda\ra \mod_\infty\Lambda$ induces a Galois $\bb{Z}$-covering $\phi\colon \Gamma_{\gr\Lambda}\ra \Gamma_{\mod_\infty\Lambda}$ of the translation quivers.

		Note that if the gradings on $kQ$ is given by paths: $(kQ)_0=ke_1\oplus ke_2$ and $(kQ)_1=k\alpha\oplus k\beta$, then all the modules in $\mod kQ$ are gradable and $\Gamma_{\gr kQ}$ is just the $\bb{Z}$-indexed copies of $\Gamma_{\mod kQ}$. See also \cite{GG}.
		\end{Ex1}

	\subsection{$G$-graded adjoint pairs and push-down functors}\

		In this section, we employ covering functors between locally bounded categories to construct $G$-graded adjoint pairs. We then apply these results to produce $G$-precoverings and Galois $G$-coverings for their associated module categories and derived categories.

		\begin{Def}{\rm(\cite[Definition 2.11]{BL})}
			Let $\cal{A},\cal{B}$ be categories such that $\cal{A}$ has direct sums and admits an action of a group $G$. Let $F\colon \cal{A}\ra\cal{B}$ and $E\colon \cal{B}\ra\cal{A}$ be functors such that $(F,E)$ is an adjoint pair with adjoint isomorphism $\phi$. We say that $(F,E)$ is $G$-graded if the following conditions hold:
			\begin{enumerate}[label=\textnormal{(\arabic*)}]
				\item There exists a functorial isomorphism $\gamma\colon\bigoplus_{g\in G}g\ra E\circ F$.
				\item There exists a $G$-stabilizer $\delta$ for $F$ such that
				\[\phi_{X,FY}(\gamma_Y\circ j_{g,Y}\circ u)=\delta_{g,Y}\circ F(u),\]
				for any $X,Y\in \cal{A}$, $g\in G$ and $u\in \cal{A}(X,g\cdot Y)$, where $j_{g,Y}\colon g\cdot Y\ra \cal{G}Y$ is the canonical injection which sends $g\cdot Y$ to $\cal{G}Y\coloneqq \bigoplus_{g\in G}g\cdot Y$.
			\end{enumerate}
		\end{Def}

			In Section \ref{sec:push-down-pull-up}, we introduced the notion of the push-down and pull-up functors. Recall that for a locally bounded category $\Lambda$, we denote by $\Mod\Lambda$ the categories of $\Lambda$-modules. Let $\Lambda,A$ be locally bounded catgories with a group $G$ acting admissibly on $\Lambda$, a Galois covering functor $F\colon \Lambda\ra A$ induces a push-down functor $F_\lambda\colon \Mod \Lambda\ra \Mod A$ and a pull-up functor $F_\bullet\colon \Mod A\ra \Mod \Lambda$.

		The following Proposition shows that $(F_\lambda,F_\bullet)$ forms a $G$-graded adjoint pair.

		\begin{Prop}{\rm(\cite[Proposition 6.4]{BL})}\label{BL, Prop 6.4}
			Let $\Lambda,A$ be locally bounded categories with $G$ a group acting admissibly on $\Lambda$. Let $F\colon \Lambda\ra A$ be a $G$-invariant Galois $G$-covering. Then the push-down functor $F_\lambda\colon \Mod\Lambda\ra \Mod A$ and the pull-up functor $F_\bullet\colon \Mod A\ra \Mod \Lambda$ form a $G$-graded adjoint pair $(F_\lambda,F_\bullet)$.
		\end{Prop}

		\subsubsection{Applications to module categories}
		\

		A full subcategory $\cal{C}$ of $\cal{A}$ is called a $G$-\textit{subcategory} if $g\cdot X\in \cal{C}$ for any $X\in \cal{C}$ and $g\in G$. In this case, the $G$-action on $\cal{A}$ restricts to a $G$-action on $\cal{C}$. This $G$-action is called $\cal{A}$\textit{-essential} if $X$ is essential in $\cal{G}(Y)\coloneqq\bigoplus_{g\in G}g\cdot Y$ for any $X,Y\in \cal{C}$.

		\begin{Thm}{\rm(\cite[Theorem 2.12]{BL})}\label{BL, Theorem 2.12}
			Let $\cal{A},\cal{B}$ be categories such that $\cal{A}$ has direct sums and admits an action of a group $G$. Let $F\colon \cal{A}\ra\cal{B}$ and $E\colon \cal{B}\ra\cal{A}$ be functors forming a $G$-graded adjoint pair $(F,E)$. Let $\cal{C}$ be a $G$-subcategory of $\cal{A}$ with a locally bounded and $\cal{A}$-essential $G$-action, and $\cal{D}$ a full subcategory of $\cal{B}$. If $F$ sends $\cal{C}$ into $\cal{D}$, then it restricts to a $G$-precovering $F'\colon \cal{C}\ra\cal{D}$. 		
		\end{Thm}

		For a locally bounded category $\Lambda$, denote by $\Mod^b\Lambda$ the full subcategory of $\Lambda$-modules consisting of those $M\in \Mod\Lambda$ for which the set $\{x\in\Lambda\mid M(x)\neq 0\}$ is finite, and $\mod \Lambda$ the category of finite dimensional $\Lambda$-modules.

		As in Section \ref{sec:push-down-pull-up}, the push-down functor $F_\lambda\colon \Mod \Lambda\ra \Mod A$ naturally restricts to both $\Mod^b \Lambda$ and $\mod \Lambda$. By using Theorem \ref{BL, Theorem 2.12}, we conclude that $F_\lambda$ induces $G$-precoverings $F_\lambda\colon \Mod^b \Lambda\ra \Mod^b A$ and $F_\lambda\colon \mod \Lambda\ra \mod A$.

		\begin{Thm}{\rm(\cite[Theorem 6.5]{BL})}\label{BL, Thm 6.5}
			Let $\Lambda,A$ be locally bounded categories with $G$ a group acting admissibly on $\Lambda$. Let $F\colon \Lambda\ra A$ be a $G$-invariant Galois $G$-covering. Then
			\begin{enumerate}[label=\textnormal{(\arabic*)}]
				\item The push-down functor $F_\lambda\colon \Mod^b \Lambda\ra \Mod^b A$ is a $G$-precovering.
				\item The push-down functor $F_\lambda\colon \mod \Lambda\ra \mod A$ is a $G$-precovering, and in case $G$ is torsion-free, it has the following properties:
				\begin{enumerate}[label=\textnormal{(\alph*)}]
					\item If $M\in \mod \Lambda$ is indecomposable, then $F_\lambda(M)$ is indecomposable.
					\item If $M,N\in \mod \Lambda$ are indecomposable with $F_\lambda(M)\cong F_\lambda(N)$, then $N\cong g\cdot M$ for some $g\in G$.
				\end{enumerate}
			\end{enumerate}
		\end{Thm}

		We conclude this subsection with the following example.

			\begin{Ex1} (Example \ref{Kronecker 2} revisited)
			Let $Q$ be the quiver
			\begin{tikzcd}
				1 & 2
				\arrow["\alpha", shift left, from=1-1, to=1-2]
				\arrow["\beta"', shift right, from=1-1, to=1-2]
			\end{tikzcd}
			and let $\Lambda\coloneqq kQ$ be the same graded algebra as in Example \ref{Kronecker 2}. Let $\widetilde{\Lambda}$ be the quiver algebra with the quiver $\widetilde{Q}$ given by
			
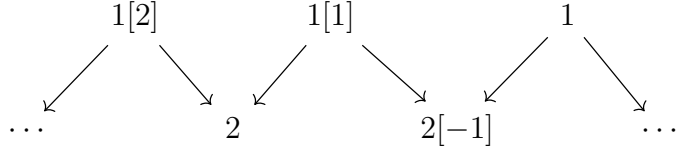
\[\begin{tikzcd}[column sep=small]
				& {1[2]} && {1[1]} && 1 \\
				\cdots && 2 && {2[-1]} && \cdots
				\arrow[from=1-2, to=2-1]
				\arrow[from=1-2, to=2-3]
				\arrow[from=1-4, to=2-3]
				\arrow[from=1-4, to=2-5]
				\arrow[from=1-6, to=2-5]
				\arrow[from=1-6, to=2-7]
			\end{tikzcd}.\]
			Then
			\begin{enumerate}[label=\textnormal{(\arabic*)}]
				\item There is a universal cover $\pi\colon \widetilde{Q}\ra Q$, and the induced functor between locally bounded categories $F\colon \widetilde{\Lambda}\ra \Lambda$ is a Galois $\bb{Z}$-covering (see Example \ref{Kron-example}).
				\item The category $\mod \widetilde{\Lambda}$ is equivalent to $\gr \Lambda$ (see \cite[Theorem 3.4]{Gre83a}).
				\item By Theorem \ref{BL, Thm 6.5} and Definition \ref{BL, Def 2.8}, the functor $F\colon \widetilde{\Lambda}\ra \Lambda$ induces a $\bb{Z}$-precovering $F_\lambda\colon \Mod^b \widetilde{\Lambda}\ra \Mod^b \Lambda$ and a Galois $\bb{Z}$-covering $F_\lambda\colon \mod \widetilde{\Lambda}\ra \mod_\infty \Lambda$.
				\item The Galois $\bb{Z}$-covering $F_\lambda\colon \mod \widetilde{\Lambda}\ra \mod_\infty \Lambda$ above induces the same covering of Auslander-Reiten quivers as $\phi\colon\Gamma_{\gr\Lambda}\ra\Gamma_{\mod_\infty\Lambda}$.
			\end{enumerate}
		\end{Ex1}

		\subsubsection{Applications to derived categories}
		\
	
		In this subsection, we will introduce some results of derived graded adjoint pairs and push-down functors. Let $\cal{A},\cal{B}$ be abelian categories, and let $F\colon \cal{A}\ra\cal{B}$ be a functor. We denote the complex category, the homotopy category and the derived category of $\cal{A}$ by $\cal{C}(\cal{A})$, $\cal{K}(\cal{A})$ and $\cal{D}(\cal{A})$, respectively.
		
		For $X^\bullet\in \cal{C}(\cal{A})$, define $F^\cal{C}(X^\bullet)$ to be the complex of which the component and the differentiation of degree $n$ are $F(X^n)$ and $F(d^n_X)$, respectively, for all $n\in \bb{Z}$. Then $F$ induces a linear functor $F^\cal{C}\colon \cal{C}^b(\cal{A})\ra \cal{C}^b(\cal{B})$, $X^\bullet\ra F^\cal{C}(X^\bullet)$, which induces functors on the homotopy and derived categories, denoted by $F^\cal{K}\colon \cal{K}^b(\cal{A})\ra \cal{K}^b(\cal{B})$ and $F^\cal{D}\colon \cal{D}^b(\cal{A})\ra \cal{D}^b(\cal{B})$, respectively.

		\begin{Prop}{\rm(\cite[Proposition 5.1]{BL})}
			Let $\cal{A},\cal{B}$ be abelian categories. If $F\colon \cal{A}\ra\cal{B}$ is an exact functor, then it induces a commutative diagram of functors
				\[\begin{tikzcd}
					{\cal{C}(\cal{A})} & {\cal{K}(\cal{A})} & {\cal{D}(\cal{A})} \\
					{\cal{C}(\cal{B})} & {\cal{K}(\cal{B})} & {\cal{D}(\cal{B})}
					\arrow["{P_{\cal{A}}}", from=1-1, to=1-2]
					\arrow["{F^{\cal{C}}}"', from=1-1, to=2-1]
					\arrow["{L_{\cal{A}}}", from=1-2, to=1-3]
					\arrow["{F^{\cal{K}}}"', from=1-2, to=2-2]
					\arrow["{F^{\cal{D}}}"', from=1-3, to=2-3]
					\arrow["{P_{\cal{B}}}", from=2-1, to=2-2]
					\arrow["{L_{\cal{B}}}", from=2-2, to=2-3]
				\end{tikzcd}\]
			where $F^\cal{C}$, is an exact functor between abelian categories, while $F^\cal{K}$ and $F^\cal{D}$ are exact functors between triangulated categories.
		\end{Prop}

		\begin{Prop}{\rm(\cite[Theorem 5.3]{BL})}
			Let $\cal{A},\cal{B}$ be abelian categories, and let $F\colon \cal{A}\ra\cal{B}$ and $E\colon \cal{B}\ra\cal{A}$ be exact functors. If $(F,E)$ is an adjoint pair, then the induced pairs $(F^\cal{C},E^\cal{C})$, $(F^\cal{K},E^\cal{K})$ and $(F^\cal{D},E^\cal{D})$ are adjoint pairs.
		\end{Prop}

		And in the case that $\cal{A}$ has essential direct sums and $(F,G)$ is a $G$-graded adjoint pair, the induced pairs $(F^\cal{C},E^\cal{C})$, $(F^\cal{K},E^\cal{K})$ and $(F^\cal{D},E^\cal{D})$ are also $G$-graded adjoint pairs.

		\begin{Prop}{\rm(\cite[Proposition 5.7]{BL})}\label{BL, Prop 5.7}
			Let $\cal{A},\cal{B}$ be abelian catgories such that $\cal{A}$ has essential direct sums and admits an action of a group $G$. Let $F\colon \cal{A}\ra\cal{B}$ and $E\colon \cal{B}\ra\cal{A}$ be exact functors. If $(F,E)$ is a $G$-graded adjoint pair, then the induced pairs $(F^\cal{C},E^\cal{C})$, $(F^\cal{K},E^\cal{K})$ and $(F^\cal{D},E^\cal{D})$ are $G$-graded adjoint pairs.
		\end{Prop}

		By Proposition \ref{BL, Prop 5.7} and \ref{BL, Prop 6.4}, we obtain a graded adjoint pair $(F_\lambda^\cal{D},F_\bullet^\cal{D})$ between the associated derived categories. The following Theorem shows that by restricting the derived push-down functor, $F$ induces a precovering between the bounded derived categories with the same properties as the push-down functor between the module categories.

		\begin{Thm}{\rm(\cite[Theorem 6.7]{BL})}\label{BL, Thm 6.7}
			Let $\Lambda,A$ be locally bounded categories with $G$ a group acting admissibly on $\Lambda$. Let $F\colon \Lambda\ra A$ be a $G$-invariant Galois $G$-covering. Then
			\begin{enumerate}[label=\textnormal{(\arabic*)}]
				\item The push-down functor $F_\lambda^\cal{D}\colon \cal{D}^b(\Mod^b \Lambda)\ra \cal{D}^b(\Mod^b A)$ is a $G$-precovering.
				\item The push-down functor $F_\lambda^\cal{D}\colon \cal{D}^b(\mod \Lambda)\ra \cal{D}^b(\mod A)$ is a $G$-precovering, and in case $G$ is torsion-free, it has the following properties:
				\begin{enumerate}[label=\textnormal{(\alph*)}]
					\item If $M^\bullet \in \cal{D}^b(\mod \Lambda)$ is indecomposable, then $F_\lambda^\cal{D}(M^\bullet)$ is indecomposable.
					\item If $M^\bullet,N^\bullet \in \cal{D}^b(\mod \Lambda)$ are indecomposable with $F_\lambda^\cal{D}(M^\bullet)\cong F_\lambda^\cal{D}(N^\bullet)$, then $N^\bullet\cong g\cdot M^\bullet$ for some $g\in G$.
				\end{enumerate}
			\end{enumerate}
		\end{Thm}

		\begin{Rem1}
			\textnormal{(1)} Theorem \ref{BL, Thm 6.7} says in particular that if $G$ is a torsion-free group, then $F_\lambda^\cal{D}\colon \cal{D}^b(\mod \Lambda)\ra \cal{D}^b(\mod A)$ is a Galois $G$-covering if and only if it is dense.
			
			\textnormal{(2)} The same result holds for the push-down functors between the complex catgories and between the homotopy catgories. 			
		\end{Rem1}

		We conclude this subsection with a classical application of Proposition \ref{BL, Prop 5.7} and Theorem \ref{BL, Thm 6.7}.

		\begin{Ex1}
			In \cite[Theorem 3.1]{Ric89}, Rickard established that for any two derived equivalent algebras $A$ and $B$, their trivial extensions $T(A)$ and $T(B)$ are also derived equivalent. Asashiba \cite{Asa97} later provided an alternative interpretation of this result using covering techniques.

			Indeed, the existence of a natural Galois $G$-covering $F\colon A^\bb{Z}\ra T(A)$, where $G=\langle \nu\rangle$ is the cyclic group generated by the Nakayama automorphism $\nu$ of $A^\bb{Z}$, is well-established. Building on this structure, Asashiba \cite[Theorem 1.5]{Asa97} demonstrates that for any pair of derived equivalent algebras $A$ and $B$, their repetitive algebras $A^\bb{Z}$ and $B^\bb{Z}$ are themselves derived equivalent. Notably, this derived equivalence for repetitive algebras is more immediately evident than the corresponding result for trivial extensions.

			By Proposition \ref{BL, Prop 5.7} and \ref{BL, Prop 6.4}, we obtain a graded adjoint pair $(F_\lambda^\cal{K},F_\bullet^\cal{K})$ between the associated homotopy categories. Following the approach of Theorem \ref{BL, Thm 6.7},  the restriction of the push-down functor $F_\lambda^\cal{K}$ induces a $G$-precovering $$F_\lambda^\cal{K}\colon\cal{K}^b(\proj A^\bb{Z})\ra\cal{K}^b(\proj T(A))$$ that inherits the same properties as the module category push-down functor (see also in \cite[Theorem 4.4]{Asa2011}).

			Since the repetitive algebras $A^\bb{Z}$ and $B^\bb{Z}$ are derived equivalent, there exists a tilting subcategory $E$ of $\cal{K}^b(\proj A^\bb{Z})$ that is equivalent to $B^\bb{Z}$. The push-down functor $F_\lambda^\cal{K}$ then induces a corresponding tilting subcategory $F_\lambda^\cal{K}(E)$ of $\cal{K}^b(\proj T(A))$ that is equivalent to $T(B)$. This construction yields the following commutative diagram:
			$$\begin{tikzcd}
\cal{K}^b(\proj A^\bb{Z}) \arrow[d, "F_\lambda^\cal{K}"'] \arrow[rr, "\cong"] &  & \cal{K}^b(\proj B^\bb{Z}) \arrow[d, "F_\lambda^\cal{K}"] \\
\cal{K}^b(\proj T(A)) \arrow[rr, "\cong", dashed]                         &  & \cal{K}^b(\proj T(B))
\end{tikzcd}$$
			where the bottom equivalence establishes the derived equivalence between $T(A)$ and $T(B)$.

			Finally, we should point out that the covering technique developed by Asashiba \cite{Asa97} naturally yields the derived equivalence of the $r$-fold trivial extensions $T^r(A):=A^\bb{Z}/\langle \nu^r\rangle$ and $T^r(B)$. This follows from the observation that the projection $A^\bb{Z}\ra T^r(A)$ constitutes a natural Galois $\langle \nu^r\rangle$-covering.
		\end{Ex1}

\section{Some further remarks on covering theory}

\subsection{Galois covering of functor categories}
\

In \cite{Pa1,Pa2}, Pastuszak investigated the relation between the Krull-Gabriel dimension of two locally bounded categories under a Galois covering by developing the Galois covering theory of functor categories. The Krull-Gabriel dimension of a locally bounded category $\Lambda$ is defined by using the category of all finitely presented contravariant $k$-linear functors from mod$\Lambda$ to mod$k$ and can be applied to study the representation type of $\Lambda$. Let $R\rightarrow A$ be a Galois covering with group $G$ between two locally bounded categories. It was shown in \cite[Theorem 1.2]{Pa2} that the Krull-Gabriel dimension KG$(R)$ of $R$ is less than or equal to the Krull-Gabriel dimension KG$(A)$ of $A$. Moreover, if $R$ is locally support-finite (see Section \ref{sec-Galois-coverings-of-representation-infinite-algebras}), $G$ is torsion-free, and $A$ contains finitely many objects, then it follows from \cite[Theorem 6.3]{Pa1} that KG$(R)=$KG$(A)$.

\subsection{Covering theory of (generalized) Brauer $G$-sets}
\

In \cite{LL1,LL2,LL3}, we defined and studied a general class of self-injective algebras - fractional Brauer configuration algebras of type $S$ (or $f_s$-BCAs for short), which is generalization of Brauer configuration algebras introduced by Green and Schroll (\cite{GS2017}). $f_s$-BCAs are defined using some combinatorial data - fractional Brauer configurations (or $f_s$-BCs for short). We developed a covering theory of $f_s$-BCs and connected it with the coverings of $f_s$-BCAs. In order to define a quotient of an $f_s$-BC by an automorphism group (may be not admissibly) acting on it, we defined the notion of (generalized) Brauer $G$-set and developed a relative covering theory of (generalized) Brauer $G$-set, by weakening the locally bijection condition in ordinary covering theory. It turns out that this relative covering theory is useful when we study the representation theory of $f_s$-BCAs.

\section*{Acknowledgements} We are very grateful to Professor Piotr Dowbor for explaining many results and questions to us on Galois covering and representation type of algebras and sending us a list of publications on coverings, which greatly improve our presentation. We also thank Fei Zeng from East China Normal University for pointing out a counterexample to \cite[Proposition~I.10.6]{E1990}. His example helped us to correct the statement in Proposition \ref{Sin}.

\bigskip
\normalem


\begin{thebibliography}{88}

	\bibitem{ARS}
	{{M. Auslander, I. Reiten and S. Smal\o,} {\em Representation Theory of
	Artin Algebras.} Cambridge Studies in Advanced Mathematics 36,
	Cambridge University Press, 1997, corrected reprint of the 1995 original.}
	
	\bibitem{Asa97}{{H. Asashiba,} A covering technique for derived equivalence. J. Algebra, \textbf{191} (1997), 382-415.}


     \bibitem{Asashiba2003}{{H.Asashiba,} On a lift of an individual stable equivalence to a standard derived equivalence for representation-finite self-injective algebras.
     Algebr. Represent. Th. \textbf{6} (4) (2003), 427-447.}

	\bibitem{Asa2011}{{H. Asashiba,} A generalization of Gabriel's Galois covering functors and derived equivalences. J. Algebra, \textbf{334} (2011), 109-149.}

	\bibitem{BGRS1985}{{R. Bautista, P. Gabriel, A. V. Rojter and L. Salmer\'on,}
	Representation-finite algebras and multiplicative bases. Invent. Math. \textbf{81}
	(1985), 217-285.}

	\bibitem{BL}{{R. Bautista and S. Liu,}
	Covering theory for linear categories with application to derived categories. J. Algebra, \textbf{406} (2014), 173-225.}

	\bibitem{BG1982}{{K. Bongartz and P. Gabriel,}
	Covering spaces in representation theory. Invent. Math. \textbf{65}
	(1982), 331-378.}

	\bibitem{BR1981}{{K. Bongartz and C. Ringel,}
	Representation-finite tree algebras. LNM 903 (Springer, 1981) 39-54.}

	\bibitem{Br-G1983}{{O. Bretscher and P. Gabriel,}
	The standard form of a representation-finite algebra. Bull. Soc. math. France, \textbf{111} (1983), 21-40.}

   \bibitem{D1996}{{P. Dowbor,}
	On the category of modules of second kind for Galois coverings. Fund. Math. \textbf{149}  (1996), 31-54.}


    \bibitem{D2001}{{P. Dowbor,}
	Stabilizer conjecture for representation-tame Galois coverings of algebras. J. Algebra, \textbf{239}  (2001), 112-149.}

    \bibitem{D20012}{{P. Dowbor,}
	Non-orbicular modules for Galois coverings. Colloq. Math. \textbf{89}  (2001), 241-309.}

	\bibitem{DLS1984}{{P. Dowbor, H. Lenzing and A. Skowro\'nski,}
	Galois coverings of algebras by locally
	support-finite categories. LNM 1177 (Springer, 1984) 91-93.}

	\bibitem{DS1985}{{P. Dowbor and A. Skowro\'nski,}
	On Galois coverings of tame algebras.  Arch. Math. \textbf{44} (1985), 522-529.}

	\bibitem{DS1986}{{P. Dowbor and A. Skowro\'nski,}
    On the representation type of locally bounded categories. Tsukuba J. Math. \textbf{10} (1986), 63-72.}

   \bibitem{DS1987}{{P. Dowbor and A. Skowro\'nski,}
	Galois coverings of representation-infinite algebras. Comment. Math. Helv. \textbf{62}  (1987), 311-337.}

	\bibitem{D}{{Y. A. Drozd,}
		Tame and wild matrix problems. LNM 832 (Springer, 1980) 242-258.}
		
	\bibitem{E1990}{{K. Erdmann,}
	{\em Blocks of tame representation type and related algebras.} LNM 1428 (Springer, 1990).}

	\bibitem{Ga1981}{{P. Gabriel,}
	The universal cover of a representation-finite algebra. LNM 903 (Springer, 1981) 68-105.}

	\bibitem{GP1993}{{Ch. Geiss and J. A. De La Pe\~na,}
    An interesting family of algebras. Arch. Math. \textbf{60} (1993), 25-35.}

	\bibitem{GG}{{R. Gordon and E. L. Green E L,} Representation theory of graded Artin algebras. J. Algebra, \textbf{76}(1) (1982), 138-152.}

	\bibitem{Gr1981}{{E. L. Green,}
	Group-graded algebras and the zero relation problem. LNM 903 (Springer, 1981) 106-115.}

	\bibitem{Gre83a}{{E. L. Green,}
	Graphs with relations, coverings and group-graded algebras. Trans. Amer. Math. Soc. \textbf{279} (1983), 297-310.}

	\bibitem{GS2017}{{\sc E.L.Green and S.Schroll,} Brauer configuration algebras: A generalization
of Brauer graph algebras. Bull. Sci. math. \textbf{141} (2017), 539-572.}

    \bibitem{LL1}{{N. Li} and {Y. Liu}, Fractional Brauer configuration algebras I: definitions and examples. J. Algebra \textbf{692} (2026), 336-378.}
		
    \bibitem{LL2}{{N. Li} and {Y. Liu}, Fractional Brauer configuration algebras II: covering theory. arXiv:2412.13445(v4).} comment: update the citation of this paper using this version.

	\bibitem{LL3}{{N. Li} and {Y. Liu}, Fractional Brauer configuration algebras III: fractional Brauer graph algebras of type MS. arXiv:2412.13449(v4).}

	\bibitem{Liu}{{S. Liu,}
	Auslander-Reiten theory in a Krull-Schmidt category. Sao Paulo J. Math. Sci., \textbf{4} (2010), no. 3, 425--472.}

	\bibitem{MMM2004}
{{E. d. N. Marcos, R. Mart{\'\i}nez-Villa, and M. R. Martins,} Hochschild cohomology of skew group rings and
  invariants, Cent. Eur. J. Math. \textbf{2} (2004), no.~2, 177--190.}

  \bibitem{MP1983}{{R. Mart\'inez-Villa and J. A. de la Pe\~na,}
	The universal cover of a quiver with relations. J. Pure Appl. Algebra, \textbf{30} (1983), 277-292.}

	\bibitem{MP2}{{R. Mart\'inez-Villa and J. A. de la Pe\~na,}
	Automorphisms of representation finite algebras.  Inv. Math. \textbf{72} (3) (1983), 359-362.}

    \bibitem{Ma1977}{{W. S. Massey,}
	{\em Algebraic topology: an introduction.} GTM 56 (Springer, 1977).}

   \bibitem{Pa1}{{G. Pastuszak,}
	On Krull-Gabriel dimension and Galois coverings.  Adv. Math. \textbf{349} (2019), 959-991.}

   \bibitem{Pa2}{{G. Pastuszak,}
	Covering theory, functor categories and the Krull-Gabriel dimension.  Preprint,1-64.}

   \bibitem{PS1983}{{Z. Pogarza\l y and A. Skowro\'nski,}
	On algebras whose indecomposable modules are multiplicity-free. Proc. London Math. Soc. \textbf{47} (1983), 463-479.}

	\bibitem{RR1985}{{I. Reiten and Ch. Riedtmann,}
Skew group algebras in the representation theory of Artin algebras. J. Algebra \textbf{92} (1985), 224-282.}

	\bibitem{Ric89}
	{J. Rickard}, {Derived categories and stable equivalence}, J. Pure Appl. Algebra, 61 (1989), 303--317.




	\bibitem{Riedtmann1980a}{{Ch. Riedtmann,} Algebren, Darstellungsk\"{o}cher,
	\"{U}berlagerungen and zur\"{u}ck.
	Comment. Math. Helv. \textbf{55} (2) (1980), 199-224.}

	\bibitem{Riedtmann1980b}{{Ch. Riedtmann,} Representation-finite self-injective
	algebras of class $A_n$. LNM 832, 1980, 449-520.}

	\bibitem{Riedtmann1983}{{Ch. Riedtmann,} Representation-finite self-injective
	algebras of class $D_n$. Compos. Math. \textbf{49} (1983), 231-282.}

	\bibitem{Sk1989}{{A. Skowro\'nski,}
	Selfinjective algebras of polynomial growth. Math. Ann. \textbf{285} (1989), 177-199.}

	\bibitem{War1969}{{R. B. Warfield,}
		A Krull-Schmidt theorem for infinite sums of modules. Proc. Amer. Math. Soc. \textbf{22} (2) (1969), 460-465.}

		
\end{thebibliography}

\end{document}